\documentclass[11pt]{article}
\usepackage[utf8]{inputenc}
\usepackage{amsmath}
\usepackage{amsthm}
\usepackage{amssymb}
\usepackage{enumerate}
\usepackage{hyperref}
\usepackage{aliascnt}
\usepackage{tikz}

\theoremstyle{plain}
\newtheorem{thm}{Theorem}[section]
\newtheorem*{thm*}{Theorem}
\newaliascnt{lem}{thm}
\newtheorem{lem}[lem]{Lemma}
\aliascntresetthe{lem}
\newaliascnt{pro}{thm}
\newtheorem{pro}[pro]{Proposition}
\aliascntresetthe{pro}
\newaliascnt{cor}{thm}

\aliascntresetthe{cor}
\newaliascnt{que}{thm}

\aliascntresetthe{que}
\newtheorem{clm}{Claim}[thm]
\newaliascnt{conj}{thm}

\aliascntresetthe{conj}

\theoremstyle{definition}
\newaliascnt{exm}{thm}

\aliascntresetthe{exm}

\newtheorem*{con}{Convention}

\newcommand{\M}{\mathcal{M}}
\newcommand{\N}{\mathcal{N}}
\newcommand{\NN}{\mathbb{N}}

\renewcommand{\P}{\mathcal{P}}
\newcommand{\Q}{\mathcal{Q}}
\newcommand{\R}{\mathcal{R}}
\renewcommand{\S}{\mathcal{S}}

\newcommand{\W}{\mathcal{W}}
\newcommand{\X}{\mathcal{X}}
\newcommand{\Y}{\mathcal{Y}}

\DeclareMathOperator{\id}{id}

\DeclareMathOperator{\tw}{tw}
\DeclareMathOperator{\coloneq}{\mathrel{\mathop:}=}

\begin{document}
\title{Linkages in Large Graphs of Bounded Tree-Width}
\author{
Jan-Oliver Fröhlich\footnote{Fachbereich Mathematik, Universität Hamburg, Hamburg, Germany.\newline jan-oliver.froehlich@math.uni-hamburg.de} 
\and Ken-ichi Kawarabayashi\footnote{National Institute of Informatics, Tokyo, Japan.}~\footnote{JST, ERATO, Kawarabayashi Large Graph Project.}
\and Theodor Müller\footnote{Fachbereich Mathematik, Universität Hamburg, Hamburg, Germany.}
\and Julian Pott\footnote{Fachbereich Mathematik, Universität Hamburg, Hamburg, Germany.}
\and Paul Wollan\footnote{Department of Computer Science, University of Rome "La Sapienza", Rome, Italy. wollan@di.uniroma1.it. Research supported by the European Research Council under the European Unions Seventh Framework Programme (FP7/2007 - 2013)/ERC Grant Agreement no.~279558.}}
\date{February 2014}
\maketitle

\begin{abstract}
We show that all sufficiently large $(2k+3)$-connected graphs of bounded tree-width are $k$-linked.
Thomassen has conjectured that all sufficiently large $(2k+2)$-connected graphs are $k$-linked.
\end{abstract}

\section{Introduction}\label{sec:introduction}

Given an integer $k\geq 1$, a graph $G$ is \emph{$k$-linked} if for any choice of $2k$ distinct vertices $s_1,\ldots, s_k$ and $t_1,\ldots, t_k$ of $G$ there are disjoint paths $P_1,\ldots, P_k$ in $G$ such that the end vertices of $P_i$ are $s_i$ and $t_i$ for $i=1,\ldots, k$.
Menger's theorem implies that every $k$-linked graph is $k$-connected.

One can conversely ask how much connectivity (as a function of $k$) is required to conclude that a graph is $k$-linked.
Larman and Mani \cite{LM} and Jung \cite{J} gave the first proofs that a sufficiently highly connected graph is also $k$-linked.
The bound was steadily improved until Bollob\'as and Thomason \cite{BT} gave the first linear bound on the necessary connectivity, showing that every $22k$-connected graph is $k$-linked.
The current best bound shows that $10k$-connected graphs are also $k$-linked \cite{TW}.

What is the best possible function $f(k)$ one could hope for which implies an $f(k)$-connected graph must also be $k$-linked?  Thomassen \cite{thomassen} conjectured that $(2k+2)$-connected graphs are $k$-linked.  However, this was quickly proven to not be the case by J\o rgensen with the following example \cite{thomassen_personal}. 
Consider the graph obtained from $K_{3k-1}$ obtained by deleting the edges of a matching of size $k$.
This graph is $(3k-3)$-connected but is not $k$-linked.
Thus, the best possible function $f(k)$ one could hope for to imply $k$-linked would be $3k-2$.
However, all known examples of graphs which are roughly $3k$-connected but not $k$-linked are similarly of bounded size, and it is possible that Thomassen's conjectured bound is correct if one assumes that the graph has sufficiently many vertices.  

In this paper, we show Thomassen's conjectured bound is almost correct with the additional assumption that the graph is large and has bounded tree-width.  This is the main result of this article.
\begin{thm}\label{thm:main}
For all integers $k$ and $w$ there exists an integer $N$ such that a graph $G$ is $k$-linked if
\[\kappa(G)\geq 2k+3,\qquad \tw(G)<w,\quad\text{and}\quad |G|\geq N.\]
where $\kappa$ is the connectivity of the graph and $\tw$ is the tree-width.
\end{thm}
The tree-width of the graph is a parameter commonly arising in the theory of graph minors;
we will delay giving the definition until \autoref{sec:outline} where we give a more in depth discussion of how tree-width arises naturally in tackling the problem.  The value $2k+2$ would be best possible;
see \autoref{sec:discussion} for examples of arbitrarily large graphs which are $(2k+1)$-connected but not $k$-linked.

Our work builds on the theory of graph minors in large, highly connected graphs begun by Böhme, Kawarabayashi, Maharry and Mohar \cite{kak}.
Recall that a graph $G$ contains $K_t$ as a \emph{minor} if there $K_t$ can be obtained from a subgraph of $G$ by repeatedly contracting edges. Böhme et al.\ showed that there exists an absolute constant $c$ such that every sufficiently large $ct$-connected graph contains $K_t$ as a minor.
This statement is not true without the assumption that the graph be sufficiently large, as there are examples of small graphs which are $(t \sqrt{\log t})$-connected but still have no $K_t$ minor \cite{Kostochka, Thomason}.
In the case where we restrict our attention to small values of $t$, one is able to get an explicit characterisation of the large $t$-connected graphs which do not contain $K_t$ as a minor.
\begin{thm}[Kawarabayashi et al.~\cite{k6large}]\label{thm:kntw}
There exists a constant $N$ such that every 6-connected graph $G$ on $N$ vertices either contains $K_6$ as a minor or there exists a vertex $v \in V(G)$ such that $G-v$ is planar.
\end{thm}
Jorgensen \cite{jorg} conjectures that \autoref{thm:kntw} holds for all graphs without the additional restriction to graphs on a large number of vertices.  
In 2010, Norine and Thomas \cite{Thomas} announced that \autoref{thm:kntw} could be generalised to arbitrary values of $t$ to either find a $K_t$ minor in a sufficiently large $t$-connected graph or alternatively, find a small set of vertices whose deletion leaves the graph planar.
They have indicated that their methodology could be used to show a similar bound of $2k+3$ on the connectivity which ensures a large graph is $k$-linked.

\section{Outline}\label{sec:outline}
In this section, we motivate our choice to restrict our attention to graphs of bounded tree-width and give an outline of the proof of \autoref{thm:main}.

We first introduce the basic definitions of tree-width.
A \emph{tree-de\-com\-pos\-ition} of a graph $G$ is a pair $(T, \X)$ where $T$ is a tree and $\X = \{X_t \subseteq V(G): t \in V(T)\}$ is a collection of subsets of $V(G)$ indexed by the vertices of $T$.  Moreover, $\X$ satisfies the following properties.
\begin{enumerate}

\item
$\bigcup_{t \in V(T)} X_t = V(G)$,

\item
for all $e \in E(G)$, there exists $t \in V(T)$ such that both ends of $e$ are contained in $X_t$, and

\item
for all $v \in V(G)$, the subset $\{t \in V(T): v \in X_t\}$ induces a connected subtree of $T$.
\end{enumerate}

The sets in $\X$ are sometimes called the \emph{bags} of the decomposition.
The \emph{width} of the decomposition is $\max_{t \in V(T)} |X_t| -1$, and the \emph{tree-width} of $G$ is the minimum width of a tree-decomposition.

Robertson and Seymour showed that if a $2k$-connected graph contains $K_{3k}$ as a minor, then it is $k$-linked \cite{RS13}.
Thus, when one considers $(2k+3)$-connected graphs which are not $k$-linked, one can further restrict attention to graphs which exclude a fixed clique minor.
This allows one to apply the excluded minor structure theorem of Robertson and Seymour \cite{RS16}.
The structure theorem can be further strengthened if one assumes the graph has large tree-width \cite{structlargetw}.
This motivates one to analyse separately the case when the tree-width is large or bounded.
The proofs of the main results in \cite{kak} and \cite{k6large} similarly split the analysis into cases based on either large or bounded tree-width.  

We continue with an outline of how the proof of \autoref{thm:main} proceeds.
Assume \autoref{thm:main} is false, and let $G$ be a $(2k+3)$-connected graph which is not $k$-linked.
Fix a set $\{s_1, \dots, s_k, t_1, \dots, t_k\}$ such that there do not exist disjoint paths $P_1, \dots, P_k$ where the ends of $P_i$ are $s_i$ and $t_i$ for all $i$.
Fix a tree-decomposition $(T, \X)$ of $G$ of minimal width $w$.  

We first exclude the possibility that $T$ has a high degree vertex.  Assume $t$ is a vertex of $T$ of large degree.
By Property~3 in the definition of a tree-decomposition, if we delete the set $X_t$ of vertices from $G$, the resulting graph must have at least $\deg_T (t)$ distinct connected components.
By the connectivity of $G$, each component contains $2k+3$ internally disjoint paths from a vertex $v$ to $2k+3$ distinct vertices in $X_t$.
If the degree of $t$ is sufficiently large, we conclude that the graph $G$ contains a subdivision of $K_{a, 2k+3}$ for some large value $a$.
We now prove that that if a graph contains such a large complete bipartite subdivision and is $2k$-connected, then it must be $k$-linked (\autoref{thm:link}).  

We conclude that the tree $T$ does not have a high degree vertex, and consequently contains a long path.
It follows that the graph $G$ has a long \emph{path decomposition}, that is, a tree-decomposition where the tree is a path.
As the bags of the decomposition are linearly ordered by their position on the path, we simply give the path decomposition as a linearly ordered set of bags $(B_1, \dots, B_t)$ for some large value $t$.
At this point in the argument, the path-decomposition $(B_1, \dots, B_t)$ may not have bounded width, but it will have the property that $|B_i \cap B_j|$ is bounded, and this will suffice for the argument to proceed.
\autoref{sec:stable_decompositions} examines this path decomposition in detail and presents a series of refinements allowing us to assume the path decomposition satisfies a set of desirable properties.
For example, we are able to assume that $|B_i \cap B_{i+1}|$ is the same for all $i$, $1 \le i < t$.
Moreover, there exist a set $\P$ of $|B_1 \cap B_2|$ disjoint paths starting in $B_1$ and ending in $B_t$.
We call these paths the \emph{foundational linkage} and they play an important role in the proof.
A further property of the path decomposition which we prove in \autoref{sec:stable_decompositions} is that for each $i$, $1 < i < t$, if there is a bridge connecting two foundational paths in $\P$ in $B_i$, then for all $j$, $1 < j < t$, there exists a bridge connecting the same foundational paths in $B_j$.
This allows us to define an auxiliary graph $H$ with vertex set $\P$ and two vertices of $\P$ adjacent in $H$ if there exists a bridge connecting them in some $B_i$ $1 < i < t$.

Return to the linkage problem at hand;
we have $2k$ terminals $s_1, \dots, s_k$ and $t_1, \dots, t_k$ which we would like to link appropriately, and $B_1, \dots, B_t$ is our path decomposition with the foundational linkage running through it.
Let the set $B_i \cap B_{i+1}$ be labeled $S_i$.
As our path decomposition developed in the previous paragraph is very long, we can assume there exists some long subsection $B_i, B_{i+1}, \dots, B_{i+a}$ such that no vertex of $s_1, \dots, s_k, t_1, \dots, t_k$ is contained in $\bigcup_i^{i+a} B_i - (S_{i-1} \cup S_{i+a})$ for some large value $a$.
By Menger's theorem, there exist $2k$ paths linking $s_1, \dots, s_k, t_1, \dots, t_k$ to the set $S_{i-1} \cup S_{i+a}$.
We attempt to link the terminals by continuing these paths into the subgraph induced by the vertex set $B_i \cup \dots \cup B_{i+a}$.
More specifically, we extend the paths along the foundational paths and attempt to link up the terminals with the bridges joining the various foundational paths in each of the $B_j$.  
By construction, the connections between foundational paths are the same in $B_j$ for all $j$, $1 < j < t$;
thus we translate the problem into a token game played on the auxiliary graph $H$.
There each terminal has a corresponding token, and the desired linkage in $G$ will exist if it is possible to slide the tokens around $H$ in such a way to match up the tokens  of the corresponding pairs of terminals.
The token game is rigorously defined in \autoref{sec:tokenmovements}, and we present a characterisation of what properties on $H$ will allow us to find the desired linkage in $G$.

The final step in the proof of \autoref{thm:main} is to derive a contradiction when $H$ doesn't have sufficient complexity to allow us to win the token game.
In order to do so, we use the high degree in $G$ and a theorem of Robertson and Seymour on crossing paths.
We give a series of technical results in preparation in \autoref{sec:relinkage} and \autoref{sec:rural_societies} and present the proof of \autoref{thm:main} in \autoref{sec:mainproof}.

\section{Stable Decompositions}\label{sec:stable_decompositions}
In this section we present a result which, roughly speaking, ensures that a highly connected, sufficiently large graph of bounded tree-width either contains a subdivision of a large complete bipartite graph or has a long path decomposition whose bags all have similar structure.

Such a theorem was first established by B\"ohme, Maharry, and Mohar in~\cite{kakbdtw} and extended by Kawarabayashi, Norine, Thomas, and Wollan in~\cite{k6bdtw}, both using techniques from~\cite{typical}.
We shall prove a further extension based on the result by Kawarabayashi et al.\ from~\cite{k6bdtw} so our terminology and methods will be close to theirs.

We begin this section with a general Lemma about nested separations.
Let $G$ be a graph.
A \emph{separation} of $G$ is an ordered pair $(A,B)$ of sets $A,B\subseteq V(G)$ such that $G[A]\cup G[B]=G$.
If $(A,B)$ is a separation of $G$, then $A\cap B$ is called its \emph{separator} and $|A\cap B|$ its \emph{order}.
Two separations $(A,B)$ and $(A',B')$ of $G$ are called \emph{nested} if either $A\subseteq A'$ and $B\supseteq B'$ or $A\supseteq A'$ and $B\subseteq B'$.
In the former case we write $(A,B)\leq (A',B')$ and in the latter $(A,B)\geq (A',B')$.
This defines a partial order $\leq$ on all separations of $G$.
A set $\S$ of separations is called \emph{nested} if the separations of $\S$ are pairwise nested, that is, $\leq$ is a linear order on $\S$.
To avoid confusion about the order of the separations in $\S$ we do not use the usual terms like smaller, larger, maximal, and minimal when talking about this linear order but instead use \emph{left}, \emph{right}, \emph{rightmost}, and \emph{leftmost}, respectively (we still use \emph{successor} and \emph{predecessor} though).
To distinguish $\leq$ from $<$ we say `left' for the former and `strictly left' for the latter (same for $\geq$ and right).

If $(A,B)$ and $(A',B')$ are both separations of $G$, then so are $(A\cap A', B\cup B')$ and $(A\cup A', B\cap B')$ and a simple calculation shows that the orders of $(A\cap A', B\cup B')$ and $(A\cup A', B\cap B')$ sum up to the same number as the orders of $(A,B)$ and $(A',B')$.
Clearly each of $(A\cap A', B\cup B')$ and $(A\cup A', B\cap B')$ is nested with both, $(A,B)$ and $(A',B')$.

For two sets $X,Y\subseteq V(G)$ we say that a separation $(A,B)$ of $G$ is an \emph{$X$--$Y$ separation} if $X\subseteq A$ and $Y\subseteq B$.
If $(A,B)$ and $(A',B')$ are $X$--$Y$ separations in $G$, then so are $(A\cap A', B\cup B')$ and $(A\cup A', B\cap B')$.
Furthermore, if $(A,B)$ and $(A',B')$ are $X$--$Y$ separations of $G$ of minimum order, say $m$, then so are $(A\cap A', B\cup B')$ and $(A\cup A', B\cap B')$ as none of the latter two can have order less than $m$ but their orders sum up to $2m$.

\begin{lem}\label{thm:nested_separations}
Let $G$ be a graph and $X,Y,Z\subseteq V(G)$.
If for every $z\in Z$ there is an $X$--$Y$ separation of $G$ of minimal order with $z$ in its separator, then there is a nested set $\S$ of $X$--$Y$ separations of minimal order such that their separators cover $Z$.
\end{lem}
\begin{proof}
Let $\S$ be a maximal nested set of $X$--$Y$ separations of minimal order in $G$ (as $\S$ is finite the existence of a leftmost and a rightmost element in any subset of $\S$ is trivial).
Suppose for a contradiction that some $z\in Z$ is not contained in any separator of the separations of $\S$.
	
Set $\S_L\coloneq \{(A,B)\in\S\mid z\in B\}$ and $\S_R\coloneq \{(A,B)\in\S\mid z\in A\}$.
Clearly $\S_L\cup \S_R = \S$ and $\S_L\cap\S_R=\emptyset$.
Moreover, if $\S_L$ and $\S_R$ are both non-empty, then the rightmost element $(A_L,B_L)$ of $\S_L$ is the predecessor of the leftmost element $(A_R,B_R)$ of $\S_R$ in $\S$.
Loosely speaking, $\S_L$ and $\S_R$ contain the separations of $\S$ ``on the left'' and ``on the right'' of $z$, respectively, and $(A_L,B_L)$ and $(A_R,B_R)$ are the separations of $\S_L$ and $\S_R$ whose separators are ``closest'' to $z$.
	
By assumption there is an $X$--$Y$ separation $(A,B)$ of minimal order in $G$ with $z\in A\cap B$.
Set
\[(A',B')\coloneq(A\cup A_L, B\cap B_L)\quad	\text{and}\quad (A'',B'')\coloneq(A'\cap A_R, B'\cup B_R)\]
(but $(A',B')\coloneq(A,B)$ if $\S_L=\emptyset$ and $(A'',B'')\coloneq(A',B')$ if $\S_R=\emptyset$).
As $(A_L,B_L)$, $(A,B)$, and $(A_R,B_R)$ are all $X$--$Y$ separations of minimal order in $G$ so must be $(A',B')$ and $(A'',B'')$.
Moreover, we have $z\in A''\cap B''$ and thus $(A'',B'')\notin\S$.

By construction we have $(A_L,B_L)\leq (A',B')$ and $(A'',B'')\leq (A_R,B_R)$.
To verify that $(A_L,B_L)\leq (A'',B'')$ we need to show $A_L\subseteq A'\cap A_R$ and $B_L \supseteq B'\cup B_R$.
All required inclusions follow from $(A_L,B_L)\leq (A',B')$ and $(A_L,B_L)\leq (A_R,B_R)$.
So by transitivity $(A'',B'')$ is right of all elements of $\S_L$ and left of all elements of $\S_R$, in particular, it is nested with all elements of $\S$, contradicting the maximality of the latter.
\end{proof}

We assume that every path comes with a fixed linear order of its vertices.
If a path arises as an $X$--$Y$ path, then we assume it is ordered from $X$ to $Y$ and if a path $Q$ arises as a subpath of some path $P$, then we assume that $Q$ is ordered in the same direction as $P$ unless explicitly stated otherwise.

Given a vertex $v$ on a path $P$ we write $Pv$ for the initial subpath of $P$ with last vertex $v$ and $vP$ for the final subpath of $P$ with first vertex $v$.
If $v$ and $w$ are both vertices of $P$, then by $vPw$ or $wPv$ we mean the subpath of $P$ that ends in $v$ and $w$ and is ordered from $v$ to $w$ or from $w$ to $v$, respectively.
By $P^{-1}$ we denote the path $P$ with inverse order.

Let $\P$ be a set of disjoint paths in some graph $G$.
We do not distinguish between $\P$ and the graph $\bigcup\P$ formed by uniting these paths;
both will be denoted by $\P$.
By a \emph{path of $\P$} we always mean an element of $\P$, not an arbitrary path in $\bigcup\P$.

Let $G$ be a graph.
For a subgraph $S\subseteq G$ an \emph{$S$-bridge in $G$} is a connected subgraph $B\subseteq G$ such that $B$ is edge-disjoint from $S$ and either $B$ is a single edge with both ends in $S$ or there is a component $C$ of $G-S$ such that $B$ consists of all edges that have at least one end in $C$.
We call a bridge \emph{trivial} in the former case and \emph{non-trivial} in the latter.
The vertices in $V(B)\cap V(S)$ and $V(B)\setminus V(S)$ are called the \emph{attachments} and the \emph{inner vertices} of $B$, respectively.
Clearly an $S$-bridge has an inner vertex if and only if it is non-trivial.
We say that an $S$-bridge $B$ \emph{attaches} to a subgraph $S'\subseteq S$ if $B$ has an attachment in $S'$.
Note that $S$-bridges are pairwise edge-disjoint and each common vertex of two $S$-bridges must be an attachment of both.

A \emph{branch vertex} of $S$ is a vertex of degree $\neq 2$ in $S$ and a \emph{segment} of $S$ is a maximal path in $S$ such that its ends are branch vertices of $S$ but none of its inner vertices are.
An $S$-bridge $B$ in $G$ is called \emph{unstable} if some segment of $S$ contains all attachments of $B$, and \emph{stable} otherwise.
If an unstable $S$-bridge $B$ has at least two attachments on a segment $P$ of $S$, then we call $P$ a \emph{host} of $B$ and say that $B$ is \emph{hosted by} $P$.
For a subgraph $H\subseteq G$ we say that two segments of $S$ are \emph{$S$-bridge adjacent} or just \emph{bridge adjacent} in $H$ if $H$ contains an $S$-bridge that attaches to both.

If a graph is the union of its segments and no two of its segments have the same end vertices, then it is called \emph{unambiguous} and \emph{ambiguous} otherwise.
It is easy to see that a graph $S$ is unambiguous if and only if all its cycles contain a least three branch vertices.
In our application $S$ will always be a union of disjoint paths so its segments are precisely these paths and $S$ is trivially unambiguous.

Let $S\subseteq G$ be unambiguous.
We say that $S'\subseteq G$ is a \emph{rerouting} of $S$ if there is a bijection $\varphi$ from the segments of $S$ to the segments of $S'$ such that every segment $P$ of $S$ has the same end vertices as $\varphi(P)$ (and thus $\varphi$ is unique by the unambiguity).
If $S'$ contains no edge of a stable $S$-bridge, then we call $S'$ a \emph{proper rerouting} of $S$.
Clearly any rerouting of the unambiguous graph $S$ has the same branch vertices as $S$ and hence is again unambiguous.

The following Lemma states two observations about proper reroutings.
The proofs are both easy and hence we omit them.
\begin{lem}\label{thm:proper_rerouting}
Let $S'$ be a proper rerouting of an unambiguous graph $S\subseteq G$ and let $\varphi$ be as in the definition.
Both of the following statements hold.
\begin{enumerate}[(i)]

\item\label{itm:pr_segments}
Every hosted $S$-bridge has a unique host.
For each segment $P$ of $S$ the segment $\varphi(P)$ of $S'$ is contained in the union of $P$ and all $S$-bridges hosted by $P$.

\item\label{itm:pr_stable}
For every stable $S$-bridge $B$ there is a stable $S'$-bridge $B'$ with $B\subseteq B'$.
Moreover, if $B$ attaches to a segment $P$ of $S$, then $B'$ attaches to $\varphi(P)$.
\end{enumerate}
\end{lem}

%
%
%
%

Note that \autoref{thm:proper_rerouting} (\ref{itm:pr_stable}) implies that no unstable $S'$-bridge contains an edge of a stable $S$-bridge.
Together with (\ref{itm:pr_segments}) this means that being a proper rerouting of an unambiguous graph is a transitive relation.

The next Lemma is attributed to Tutte;
we refer to \cite[Lemma 2.2]{k6bdtw} for a proof\footnote{
To check that Lemma 2.2 in \cite{k6bdtw} implies our \autoref{thm:tutte_bridges} note that if $S'$ is obtained from $S$ by ``a sequence of proper reroutings'' as defined in \cite{k6bdtw}, then by transitivity $S'$ is a proper rerouting of $S$ according to our definition.
And although not explicitly included in the statement, the given proof shows that no trivial $S'$-bridge can be unstable.
}.

\begin{lem}\label{thm:tutte_bridges}
Let $G$ be a graph and $S\subseteq G$ unambiguous.
There exists a proper rerouting $S'$ of $S$ in $G$ such that if $B'$ is an $S'$-bridge hosted by some segment $P'$ of $S'$, then $B'$ is non-trivial and there are vertices $v,w\in V(P')$ such that the component of $G-\{v,w\}$ that contains $B'-\{v,w\}$ is disjoint from $S'- vP'w$.
\end{lem}
This implies that the segments of $S'$ are induced paths in $G$ as trivial $S'$-bridges cannot be unstable and no two segments of $S'$ have the same end vertices.

Let $G$ be a graph.
A set of disjoint paths in $G$ is called a \emph{linkage}.
If $X,Y\subseteq V(G)$ with $k\coloneq |X|=|Y|$, then 
a set of $k$ disjoint $X$--$Y$ paths in $G$ is called an \emph{$X$\!--$Y$ linkage} or a \emph{linkage from $X$ to $Y$}.
Let $\W = (W_0,\ldots, W_l)$ be an ordered tuple of subsets of $V(G)$.
Then $l$ is the \emph{length} of $\W$, the sets $W_i$ with $0\leq i\leq l$ are its \emph{bags}, and the sets $W_{i-1}\cap W_i$ with $1\leq i\leq l$ are its \emph{adhesion sets}.
We refer to the bags $W_i$ with $1\leq i\leq l-1$ as \emph{inner} bags.
When we say that a bag $W$ of $\W$ \emph{contains} some graph $H$, we mean $H\subseteq G[W]$.
Given an inner bag $W_i$ of $\W$, the sets $W_{i-1}\cap W_{i}$ and $W_{i}\cap W_{i+1}$ are called the \emph{left} and \emph{right} adhesion set of $W_i$, respectively.
Whenever we introduce a tuple $\W$ as above without explicitly naming its elements, we shall denote them by $W_0, \ldots, W_l$ where $l$ is the length of $\W$.
For indices $0\leq j\leq k\leq l$ we use the shortcut $W_{[j,k]}\coloneq \bigcup_{i=j}^k W_i$.

The tuple $\W$ with the following five properties is called a \emph{slim decomposition} of $G$.
\begin{enumerate}[(L1)]

\item
$\bigcup\W = V(G)$ and every edge of $G$ is contained in some bag of $\W$.

\item
If $0\leq i\leq j\leq k\leq l$, then $W_i\cap W_k\subseteq W_j$.

\item
All adhesion sets of $\W$ have the same size.

\item
No bag of $\W$ contains another.

\item
$G$ contains a $(W_0\cap W_1)$--$(W_{l-1}\cap W_l)$ linkage.
\newcounter{enumi_saved}\setcounter{enumi_saved}{\value{enumi}}
\end{enumerate}

The unique size of the adhesion sets of a slim decomposition is called its \emph{adhesion}.
A linkage $\P$ as in (L5) together with an enumeration $P_1,\ldots, P_q$ of its paths is called a \emph{foundational linkage} for~$\W$ and its members are called \emph{foundational paths}.
Each path $P_{\alpha}$ contains a unique vertex of every adhesion set of $\W$ and we call this vertex the \emph{$\alpha$-vertex} of that adhesion set.
For an inner bag $W$ of $\W$ the $\alpha$-vertex in the left and right adhesion set of $W$ are called the \emph{left} and \emph{right} $\alpha$-vertex of $W$, respectively.
Note that $\P$ is allowed to contain trivial paths so $\bigcap \W$ may be non-empty.

The enumeration of a foundational linkage $\P$ for $\W$ is a formal tool to compare arbitrary linkages between adhesion sets of $\W$ to $\P$ by their `induced permutation' as detailed below.
When considering another foundational linkage $\Q =\{Q_1,\ldots, Q_q\}$ for $\W$ we shall thus always assume that it induces the same enumeration as $\P$ on $W_0\cap W_1$, in other words, $Q_{\alpha}$ and $P_{\alpha}$ start on the same vertex.

Suppose that $\W$ is a slim decomposition of some graph $G$ with foundational linkage $\P$.
Then any $\P$-bridge $B$ in $G$ is contained in a bag of $\W$, and this bag is unique unless $B$ is trivial and contained in one or more adhesion sets.

We say that a linkage $\Q$ in a graph $H$ is \emph{$p$-attached} if each path of $\Q$ is induced in $H$ and if some non-trivial $\Q$-bridge $B$ attaches to a non-trivial path $P$ of $\Q$, then either $B$ attaches to another non-trivial path of $\Q$ or there are at least $p-2$ trivial paths $Q$ of $\Q$ such that $H$ contains a $\Q$-bridge (which may be different from $B$) attaching to $P$ and $Q$.

We call a pair~$(\W,\P)$ of a slim decomposition $\W$ of $G$ and a foundational linkage $\P$ for $\W$ a \emph{regular decomposition of attachedness $p$ of $G$} if there is an integer $p$ such that the axioms (L6), (L7), and (L8) hold.
\begin{enumerate}[(L1)]\setcounter{enumi}{\value{enumi_saved}}

\item
$\P[W]$ is $p$-attached in $G[W]$ for all inner bags $W$ of $\W$.

\item
A path $P\in\P$ is trivial if $P[W]$ is trivial for some inner bag $W$ of $\W$.

\item
For every $P,Q\in\P$, if some inner bag of $\W$ contains a $\P$-bridge attaching to $P$ and $Q$, then every inner bag of $\W$ contains such a $\P$-bridge.
\addtocounter{enumi}{1}
\setcounter{enumi_saved}{\value{enumi}}
\end{enumerate}
The integer $p$ is not unique:
A regular decomposition of attachedness $p$ has attachedness $p'$ for all integers $p'\leq p$.
Note that $\P$ satisfies (L7) if and only if every vertex of $G$ either lies in at most two bags of $\W$ or in all bags.
This means that either all foundational linkages for $\W$ satisfy (L7) or none.

The next Theorem follows\footnote{
The statement of Lemma~3.1 in~\cite{k6bdtw} only asserts the existence of a minor isomorphic to $K_{a,p}$ rather than a subdivision of $K_{a,p}$ like we do.
But its proof refers to an argument in the proof of~\cite[Theorem 3.1]{typical} which actually gives a subdivision.
}
from the Lemmas~3.1, 3.2, and 3.5 in~\cite{k6bdtw}.

\begin{thm}[Kawarabayashi et al.~\cite{k6bdtw}]\label{thm:regularpd}
For all integers $a,l,p,w\geq 0$ there exists an integer $N$ with the following property.
If $G$ is a $p$-connected graph of tree-width less than $w$ with at least $N$ vertices, then either $G$ contains a subdivision of $K_{a,p}$, or $G$ has a regular decomposition of length at least $l$, adhesion at most $w$, and attachedness $p$.
\end{thm}

Note that~\cite{k6bdtw} features a stronger version of \autoref{thm:regularpd}, namely Theorem~3.8, which includes an additional axiom (L9).
We omit that axiom since our arguments do not rely on it.

Let $(\W,\P)$ be a slim decomposition of adhesion $q$ and length $l$ for a graph $G$.
Suppose that $\Q$ is a linkage from the left adhesion set of $W_i$ to the right adhesion set of $W_j$ for two indices $i$ and $j$ with $1\leq i \leq j < l$.
The enumeration $P_1,\ldots, P_q$ of $\P$ induces an enumeration $Q_1,\ldots, Q_q$ of $\Q$ where $Q_{\alpha}$ is the path of $\Q$ starting in the left $\alpha$-vertex of $W_i$.
The map $\pi : \{1,\ldots, q\} \to \{1,\ldots, q\}$ such that $Q_{\alpha
}$ ends in the right $\pi(\alpha)$-vertex of $W_j$ for $\alpha=1,\ldots, q$ is a permutation because $\Q$ is a linkage.
We call it the \emph{induced permutation of $\Q$}.
Clearly the induced permutation of $\Q$ is the composition of the induced permutations of $\Q[W_i]$, $\Q[W_{i+1}]$, \ldots, $\Q[W_j]$.
For any permutation $\pi$ of $\{1,\ldots, q\}$ and any graph $\Gamma$ on $\{1,\ldots, q\}$ we write $\pi\Gamma$ to denote the graph $(\{\pi(\alpha)\mid \alpha\in V(\Gamma)\},\{\pi(\alpha)\pi(\beta)\mid \alpha\beta\in E(\Gamma)\})$.
For a subset $X\subseteq \{1,\ldots, q\}$ we set $\Q_X\coloneq\{Q_{\alpha}\mid \alpha\in X\}$.

Keep in mind that the enumerations $\P$ induces on linkages $\Q$ as above always depend on the adhesion set where the considered linkage starts.
For example let $\Q$ be as above and for some index $i'$ with $i<i'\leq j$ set $\Q'\coloneq \Q[W_{[i',j]}]$.
Then $Q_{\alpha}[W_{[i',j]}]$ need not be the same as $Q'_{\alpha}$.
More precisely, we have $Q_{\alpha}[W_{[i',j]}]=Q'_{\tau(\alpha)}$ where $\tau$ denotes the induced permutation of $\Q[W_{[i,i'-1]}]$.

For some subgraph $H$ of $G$ the \emph{bridge graph of $\Q$ in $H$}, denoted $B(H, \Q)$, is the graph with vertex set $\{1,\ldots, q\}$ in which $\alpha\beta$ is an edge if and only if $Q_{\alpha}$ and $Q_{\beta}$ are $\Q$-bridge adjacent in $H$.
Any $\Q$-bridge $B$ in $H$ that attaches to $Q_{\alpha}$ and $Q_{\beta}$ is said to \emph{realise} the edge $\alpha\beta$.
We shall sometimes think of induced permutations as maps between bridge graphs.

For a slim decomposition $\W$ of length $l$ of $G$ with foundational linkage $\P$ we define the \emph{auxiliary graph} $\Gamma(\W,\P)\coloneq B(G[W_{[1,\ l-1]}], \P)$.
Clearly $B(G[W],\P[W])\subseteq\Gamma(\W,\P)$ for each inner bag $W$ of $\W$ and if $(\W,\P)$ is regular, then by (L8) we have equality.

Set $\lambda\coloneq \{\alpha\mid P_{\alpha}\text{ is non-tivial}\}$ and $\theta\coloneq\{\alpha\mid P_{\alpha}\text{ is trivial}\}$.
Given a subgraph $\Gamma\subseteq \Gamma(\W,\P)$ and some foundational linkage $\Q$ for $\W$, we write $G_{\Gamma}^{\Q}$ for the graph obtained by deleting $\Q\setminus \Q_{V(\Gamma)}$ from the union of $\Q$ and those $\Q$-bridges in inner bags of $\W$ that realise an edge of $\Gamma$ or attach to $\Q_{V(\Gamma)\cap\lambda}$ but to no path of $\Q_{\lambda\setminus V(\Gamma)}$.
For a subset $V\subseteq\{1,\ldots, q\}$ we write $G_{V}^{\Q}$ instead of $G_{\Gamma(\W,\P)[V]}^{\Q}$.
Note that $\Q_{\theta}=\P_{\theta}$.
Hence $G_{\lambda}^{\P}$ and $G_{\lambda}^{\Q}$ are the same graph and we denote it by $G_{\lambda}$.

A regular decomposition $(\W,\P)$ of a graph $G$ is called \emph{stable} if it satisfies the following two axioms where $\lambda\coloneq\{\alpha\mid P_{\alpha}\text{ is non-trivial}\}$.
\begin{enumerate}[(L1)]\setcounter{enumi}{\value{enumi_saved}}

\item
If $\Q$ is a linkage from the left to the right adhesion set of some inner bag of $\W$, then its induced permutation is an automorphism of $\Gamma(\W,\P)$.

\item
If $\Q$ is a linkage from the left to the right adhesion set of some inner bag $W$ of $\W$, then every edge of $B(G[W], \Q)$ with one end in $\lambda$ is also an edge of $\Gamma(\W,\P)$.
\end{enumerate}
 
Given these definitions we can further expound our strategy to prove the main theorem:
We will reduce the given linkage problem to a linkage problem with start and end vertices in $W_0 \cup W_l$ for some stable regular decomposition $(\W,\P)$ of length $l$.
The stability ensures that we maximised the number of edges of $\Gamma(\W,\P)$, i.e.\ no rerouting of $\P$ will give rise to new bridge adjacencies.
We will focus on a subset $\lambda_0\subseteq\lambda$ and show that the minimum degree of $G$ forces a high edge density in $G_{\lambda_0}^{\P}$, leading to a high number of edges in $\Gamma(\W,\P)[\lambda_0]$.
Using combinatoric arguments, which we elaborate in \autoref{sec:tokenmovements}, we show that we can find linkages using segments of $\P$ and $\P$-bridges 
in $G_{\lambda_0}^{\P}$ to realise any matching of start and end vertices in $W_0\cup W_l$, showing that $G$ is in fact $k$-linked.

We strengthen \autoref{thm:regularpd} by the assertion that the regular decomposition can be chosen to be stable.
We like to point out that, even with the left out axiom (L9) included in the definition of a regular decomposition, \autoref{thm:stablepd} would hold.
By almost the same proof as in \cite{k6bdtw} one could also obtain a stronger version of (L8) stating that for every subset $\R$ of $\P$ if some inner bag of $\W$ contains a $\P$-bridge attaching every path of $\R$ but to no path of $\P\setminus\R$, then every inner bag does.

\begin{thm}\label{thm:stablepd}
\newcounter{thm_saved}
\setcounter{thm_saved}{\value{thm}}
For all integers $a,l,p,w\geq 0$ there exists an integer $N$ with the following property.
If $G$ is a $p$-connected graph of tree-width less than $w$ with at least $N$ vertices, then either $G$ contains a subdivision of $K_{a,p}$, or $G$ has a stable regular decomposition of length at least $l$, adhesion at most $w$, and attachedness $p$.
\end{thm}

Before we start with the formal proof let us introduce its central concepts: disturbances and contractions.
Let $(\W,\P)$ be a regular decomposition of a graph $G$.
A linkage $\Q$ is called a \emph{twisting $(\W,\P)$-disturbance} if it violates (L10) and it is called a \emph{bridging $(\W,\P)$-disturbance} if it violates (L11).
By a \emph{$(\W,\P)$-disturbance} we mean either of these two and a disturbance may be twisting and bridging at the same time.
If the referred regular decomposition is clear from the context, then we shall not include it in the notation and just speak of a disturbance.
Note that a disturbance is always a linkage from the left to the right adhesion set of an inner bag of $\W$.

Given a disturbance $\Q$ in some inner bag $W$ of $\W$ which is neither the first nor the last inner bag of $\W$, it is not hard to see that replacing $\P[W]$ with $\Q$ yields a foundational linkage $\P'$ for $\W$ such that $\Gamma(\W,\P')$ properly contains $\Gamma(\W,\P)$ and we shall make this precise in the proof.
As the auxiliary graph can have at most $\binom{w}{2}$ edges, we can repeat this step until no disturbances (with respect to the current decomposition) are left and we should end up with a stable regular decomposition, given that we can somehow preserve the regularity.

This is done by ``contracting'' the decomposition in a certain way.
The technique is the same as in~\cite{kakbdtw} or~\cite{k6bdtw}.
Given a regular decomposition $(\W,\P)$ of length $l$ of some graph $G$ and a subsequence $i_1,\ldots, i_n$ of $1,\ldots, l$, the \emph{contraction of $(\W,\P)$ along $i_1,\ldots,i_n$} is the pair $(\W',\P')$ defined as follows.
We let $\W' \coloneq (W_0', W_1',\ldots, W_n')$ with $W_0'\coloneq W_{[0,\ i_1-1]}$,
\[W_j' \coloneq W_{[i_j,\ i_{j+1}-1]}\quad\text{for}\quad j=1,\ldots, n-1,\]
$W_{n}\coloneq W_{[i_n, l]}$, and $\P' = \P[W_{[1,\ n-1]}']$ (with the induced enumeration).

\begin{lem}\label{thm:contraction}
Let $(\W',\P')$ be the contraction of a regular decomposition $(\W,\P)$ of some graph $G$ of adhesion $q$ and attachedness $p$ along the sequence $i_1,\ldots,i_n$.
Then the following two statements hold.
\begin{enumerate}[(i)]

\item
$(\W', \P')$ is a regular decomposition of length $n$ of $G$ of adhesion $q$ and attachedness $p$, and $\Gamma(\W',\P') = \Gamma(\W,\P)$.

\item
The decomposition $(\W',\P')$ is stable if and only if none of the inner bags $W_{i_1}, W_{i_1+1},\ldots, W_{i_n-1}$ of~$\W$ contains a $(\W,\P)$-disturbance.
\end{enumerate}
\end{lem}

\begin{proof}
The first statement is Lemma 3.3 of~\cite{k6bdtw}.
The second statement follows from the fact that an inner bag $W'_j$ of $\W'$ contains a $(\W',\P')$-disturbance if and only if one of the bags $W_i$ of $\W$ with $i_j\leq i< i_{j+1}$ contains a $(\W,\P)$-disturbance (unless $\W'$ has no inner bag, that is, $n=1$).
The ``if'' direction is obvious and for the ``only if'' direction recall that the induced permutation of $\P'[W'_j]$ is the composition of the induced permutations of the $\P[W_i]$ with $i_j\leq i< i_{j+1}$ and every $\P'$-bridge in $W'_j$ is also a $\P$-bridge and hence must be contained in some bag $W_i$ with $i_j\leq i< i_{j+1}$.
\end{proof}

Let $\Q$ be a linkage in a graph $H$ and denote the trivial paths of $\Q$ by $\Theta$.
Let $\Q'$ be the union of $\Theta$ with a proper rerouting of $\Q\setminus\Theta$ obtained from applying \autoref{thm:tutte_bridges} to $\Q\setminus\Theta$ in $H-\Theta$.
We call $\Q'$ a \emph{bridge stabilisation of $\Q$ in $H$}.
The next Lemma tailors \autoref{thm:proper_rerouting} and \autoref{thm:tutte_bridges} to our application.

\begin{lem}\label{thm:bridge_stabilisation}
Let $\Q$ be a linkage in a graph $H$.
Denote by $\Theta$ the trivial paths of $\Q$ and let $\Q'$ be a bridge stabilisation of $\Q$ in $H$.
Let $P$ and $Q$ be paths of $\Q$ and let $P'$ and $Q'$ be the unique paths of $\Q'$ with the same end vertices as $P$ and $Q$, respectively.
Then the following statements hold.
\begin{enumerate}[(i)]

\item\label{itm:bs_segments}
$P'$ is contained in the union of $P$ with all $\Q$-bridges in $H$ that attach to $P$ but to no other path of $\Q\setminus\Theta$.

\item\label{itm:bs_bridges}
If $P$ and $Q$ are $\Q$-bridge adjacent in $H$ and one of them is non-trivial, then $P'$ and $Q'$ are $\Q'$-bridge adjacent in $H$.
	
\item\label{itm:bs_attached}
Let $Z$ be the set of end vertices of the paths of $\Q$. If $p$ is an integer such that for every vertex $x$ of $H-Z$ there is an $x$--$Z$ fan of size $p$, then $\Q'$ is $p$-attached.
\end{enumerate}
\end{lem}

\begin{proof}~
\begin{enumerate}[(i)]
\item
This is trivial if $P\in\Theta$ and follows easily from \autoref{thm:proper_rerouting} (\ref{itm:pr_segments}) otherwise.

\item
The statement follows directly from \autoref{thm:proper_rerouting} (\ref{itm:pr_stable}) if $P$ and $Q$ are both non-trivial so we may assume that $P=P'\in\Theta$ and $Q$ is non-trivial.
By assumption there is a $P$--$Q$ path $R$ in $H$.
Clearly $R\cup Q$ contains the end vertices of $Q'$.
On the other hand, by (\ref{itm:bs_segments}) it is clear that $Q\cap \Q'\subseteq Q'$.
We claim that $R\cap\Q'\subseteq Q'$.
Since $R$ is internally disjoint from $\Q$ all its inner vertices are inner vertices of some $(\Q\setminus\Theta)$-bridge $B$.
If $B$ is stable or unstable but not hosted by any path of $\Q$ (that is, it has at most one attachment), then \autoref{thm:proper_rerouting} implies that no path of $\Q'$ contains an inner vertex of $B$ and that our claim follows.
If $B$ is hosted by a path of $\Q$, then this path must clearly be $Q$ and thus by \autoref{thm:proper_rerouting} (\ref{itm:pr_segments}) $R\cap\Q'\subseteq Q'$ as claimed.
Hence $R\cup Q$ contains a $P$--$Q'$ path that is internally disjoint from $\Q'$ as desired.

\item
Clearly all paths of $\Q'$ are induced in $H$, either because they are trivial or by \autoref{thm:tutte_bridges}.
Let $B$ be a non-trivial hosted $\Q'$-bridge and let $Q'$ be the non-trivial path of $\Q'$ to which it attaches.
Then by \autoref{thm:tutte_bridges} there are vertices $v$ and $w$ on $Q'$ and a separation $(X,Y)$ of $H$ such that $V(B)\subseteq X$, $X\cap Y\subseteq \{v,w\}\cup V(\Theta)$, and apart from the inner vertices of $vQ'w$ all vertices of $\Q'$ are in $Y$, in particular, $Z\subseteq Y$.
But $B$ has an inner vertex $x$ which must be in $X\setminus Y$.
So by assumption there is an $x$--$\{v,w\}\cup V(\Theta)$ fan of size $p$ in $G[X]$ and thus also an $x$--$\Theta$ fan of size $p-2$.
It is easy to see that this can gives rise to the desired $\Q'$-bridge adjacencies in $H$.
\end{enumerate}
\end{proof}

\begin{proof}[Proof of \autoref{thm:stablepd}]
\newcounter{thm_temp}\setcounter{thm_temp}{\value{thm}}
\setcounter{thm}{\value{thm_saved}}
\setcounter{thm_saved}{\value{thm_temp}}
We will trade off some length of a regular decomposition to gain edges in its auxiliary graph.
To quantify this we define the function $f:\NN_0\to\NN_0$ by $f(m)\coloneq (zlw!)^ml$ where $z\coloneq 2^{\binom{w}{2}}$ and call a regular decomposition $(\W,\P)$ of a graph $G$ \emph{valid} if it has adhesion at most $w$, attachedness $p$, and length at least $f(m)$ where $m$ is the number of edges in the complement of $\Gamma(\W,\P)$ that are incident with at least one non-trivial path of $\P$.
	
Set $\lambda \coloneq f\left(\binom{w}{2}\right)$ and let $N$ be the integer returned by \autoref{thm:regularpd} when invoked with parameters $a$, $\lambda$, $p$, and $w$.
We claim that the assertion of \autoref{thm:stablepd} is true for this choice of $N$.
Let $G$ be a $p$-connected graph of tree-width less than $w$ with at least $N$ vertices and suppose that $G$ does not contain a subdivision of $K_{a,p}$.
Then by the choices of $N$ and $\lambda$ the graph $G$ has a valid decomposition (the foundational linkage has at most $w$ paths so there can be at most $\binom{w}{2}$ non-edges in the auxiliary graph).
Among all valid decompositions of $G$ pick $(\W,\P)$ such that the number of edges of $\Gamma(\W,\P)$ is maximal and denote the length of $(\W,\P)$ by~$n$.

We may assume that for any integer $k$ with $0\leq k\leq n-l$ one of the $l-1$ consecutive inner bags $W_{k+1}, \ldots, W_{k+l-1}$ of $\W$ contains a disturbance.
If not, then by \autoref{thm:contraction}, the contraction of $(\W,\P)$ along the sequence $k+1, k+2,\ldots, k+l$ is a stable regular decomposition of $G$ of length $l$, adhesion at most $w$, and attachedness $p$ as desired.

\begin{clm}\label{clm:sp_identity_bag}
Let $1\leq k\leq k'\leq n-1$ with $k' - k\geq lw!-1$.
Then the graph $H\coloneq G[W_{[k, k']}]$ contains a linkage $\Q$ from the left adhesion set of $W_{k}$ to the right adhesion set of $W_{k'}$ such that $B(H,\Q)$ is a proper supergraph of $\Gamma(\W,\P)$, the induced permutation $\pi$ of $\Q$ is the identity, and $\Q$ is $p$-attached in $H$.
\end{clm}

\begin{proof}
There are indices $k_0\coloneq k, k_1, \ldots, k_{w!}\coloneq k'+1$, such that for $j\in\{1,\ldots, w!\}$ we have $k_{j} - k_{j-1}\geq l$.
For each $j\in\{0,\ldots, w!-1\}$ one of the at least $l-1$ consecutive inner bags $W_{k_j+1}, W_{k_j+2}, \ldots, W_{k_{j+1}-1}$ contains a disturbance $\Q_j$ by our assumption.
Let $W_{i_j}$ be the bag of $\W$ that contains $\Q_j$ and let $\Q'_j$ be the bridge stabilisation of $\Q_j$ in $G[W_{i_j}]$.

If $\Q_j$ is a twisting $(\W,\P)$-disturbance, then so is $\Q'_j$ as they have the same induced permutation.
If $\Q_j$ is a bridging $(\W,\P)$-disturbance, then so is $\Q'_j$ by \autoref{thm:bridge_stabilisation} (\ref{itm:bs_bridges}).
The set $Z$ of end vertices of $\Q_j$ is the union of both adhesion sets of $W_{i_j}$ and clearly for every vertex $x\in W_{i_j}\setminus Z$ there is an $x$--$Z$ fan of size $p$ in $G[W_{i_j}]$ as $G$ is $p$-connected.
So by \autoref{thm:bridge_stabilisation} (\ref{itm:bs_attached}) the linkage $\Q'_j$ is $p$-attached in $G[W_{i_j}]$.

For every $j\in \{0,\ldots, w!-1\}$ denote the induced permutation of $\Q'_j$ by $\pi_j$.
Since the symmetric group $S_q$ has order at most $q!\leq w!$ we can pick\footnote{
Let $(G, \cdot)$ be a group of order $n$ and $g_1,\ldots, g_n\in G$. Then of the $n+1$ products $h_k\coloneq \prod_{i=1}^k g_i$ for $0\leq k\leq n$, two must be equal by the pigeon hole principle, say $h_k = h_l$ with $k<l$. This means $\prod_{i=k+1}^{l}g_{i} = e$, where $e$ is the neutral element of $G$.
}
indices $j_0$ and $j_1$ with $0\leq j_0\leq j_1\leq w!-1$ such that $\pi_{j_1}\circ\pi_{j_1-1}\ldots\circ\pi_{j_0} = \id$.
	
Let $\Q$ be the linkage from the left adhesion set of $W_{k}$ to the right adhesion set of $W_{k'}$ in $H$ obtained from $\P[W_{[k,k']}]$ by replacing $\P[W_{i_j}]$ with $\Q'_j$ for all $j\in\{j_0,\ldots, j_1\}$.
Of all the restrictions of $\Q$ to the bags $W_k,\ldots, W_{k'}$ only $\Q[W_{i_j}]=\Q_{j}$ with $j_0\leq j\leq j_1$ need not induce the identity permutation.
However, the composition of their induced permutations is the identity by construction and therefore the induced permutation of $\Q$ is the identity.

To see that $B(H,\Q)$ is a supergraph of $\Gamma(\W,\P)$ note that $k<i_{j_0}$ so $\Q$ and $\P$ coincide on $W_k$ and hence by (L8) we have
\[
\Gamma(\W,\P) = B(G[W_k],\P[W_k]) \subseteq\linebreak[0] B(H,\Q).
\]

It remains to show that $B(H,\Q)$ contains an edge that is not in $\Gamma(\W,\P)$.
Set $W\coloneq W_{i_{j_0}}$, $W'\coloneq W_{i_{j_0}+1}$, and $\pi\coloneq \pi_{j_0}$.
If $\Q'_{j_0}$ is a bridging disturbance, then $B_0\coloneq B(G[W],\Q[W])$ contains an edge that is not in $\Gamma(\W,\P)$.
Since $\Q$ and $\P$ coincide on all bags prior to $W$ (down to $W_k$) we must have $B_0\subseteq B(H,\Q)$.

If $\Q'_{j_0}$ is a twisting disturbance, then $j_1>j_0$, in particular, $W'$ comes before $W_{i_{j_0+1}}$ (there is at least one bag between $W_{i_{j_0}}$ and $W_{i_{j_0+1}}$, namely $W_{k_{j_0+1}}$).
This means $\Q[W']=\P[W']$ and hence we have
\[
B_1\coloneq B(G[W'],\Q[W']) = B(G[W'],\P[W']) = \Gamma(\W,\P).
\]
On the other hand, the induced permutation of the restriction of $\Q$ to all bags prior to $W'$ is $\pi$ and thus $\pi^{-1}B_1\subseteq B(H,\Q)$.
But $\pi$ is not an automorphism of $\Gamma(\W,\P)$ and therefore $\pi^{-1}B_1 = \pi^{-1}\Gamma(\W,\P)$ contains an edge that is not in $\Gamma(\W,\P)$ as desired.
This concludes the proof of \autoref{clm:sp_identity_bag}
\end{proof}
	
To exploit \autoref{clm:sp_identity_bag} we now contract subsegments of $lw!$ consecutive inner bags of $\W$ into single bags.
We assumed earlier that $(\W,\P)$ is not stable so the number $m$ of non-edges of $\Gamma(\W,\P)$ is at least~$1$ (if $\Gamma(\W,\P)$ is complete there can be no disturbances).
Set $n'\coloneq zf(m-1)$.
As $(\W,\P)$ is valid, its length $n$ is at least $f(m) = zlw! f(m-1) = n'lw!$.
Let $(\W',\P')$ be the contraction of $(\W,\P)$ along the sequence $i_1,\ldots, i_{n'}$ defined by $i_j \coloneq (j-1)lw! + 1$ for $j=1,\ldots, n'$.
Then by \autoref{thm:contraction} the pair $(\W',\P')$ is a regular decomposition of $G$ of length $n'$, adhesion at most $w$, it is $p$-attached, and $\Gamma(\W',\P')=\Gamma(\W,\P)$.

By construction every inner bag $W'_i$ of $\W'$ consists of $lw!$ consecutive inner bags of $\W$ and hence by \autoref{clm:sp_identity_bag} it contains a bridging disturbance $\Q'_i$ such $\Q'_i$ is $p$-attached in $G[W'_i]$, its induced permutation is the identity, and $B(G[W'_i], \Q'_i)$ is a proper supergraph of $\Gamma(\W',\P')$.

Clearly $\Gamma(\W',\P')$ has at most $z-1$ proper supergraphs on the same vertex set.
On the other hand, $\W'$ has at least $n' - 1 = zf(m-1) - 1$ inner bags.
By the pigeonhole principle there must be $f(m-1)$ indices $0<i_1<\ldots<i_{f(m-1)}<n'$ such that $B(G[W'_{i_j}], \Q'_{i_j})$ is the same graph $\Gamma$ for $j=1,\ldots, f(m-1)$.
	
Let $(\W'',\P'')$ be the contraction of $(\W',\P')$ along $i_1,\ldots, i_{f(m-1)}$.
Obtain the foundational linkage $\Q''$ for $\W''$ from $\P''$ by replacing $\P'[W_{i_j}]$ with $\Q_{i_j}$ for $1\leq j\leq f(m-1)$.
By construction $\W''$ is a slim decomposition of $G$ of length $f(m-1)$ and of adhesion at most $w$.
$\Q''$ is a foundational linkage for $\W''$ that satisfies (L7) because $\P''$ does.
By construction $\Q''$ is $p$-attached and $B(G[W''],\Q''[W'']) = \Gamma$ for all inner bags $W''$ of $\W''$.
Hence $(\W'',\P'')$ is regular decomposition of~$G$.
But it is valid and its auxiliary graph $\Gamma$ has more edges than $\Gamma(\W,\P)$, contradicting our initial choice of $(\W,\P)$.
\setcounter{thm}{\value{thm_saved}}
\end{proof}

\section{Token Movements}\label{sec:tokenmovements}
Consider the following token game.
We place distinguishable tokens on the vertices of a graph $H$, at most one per vertex.
A move consists of sliding a token along the edges of $H$ to a new vertex without passing through vertices which are occupied by other tokens.
Which placements of tokens can be obtained from each other by a sequence of moves?

A rather well-known instance of this problem is the 15-puzzle where tokens $1,\ldots, 15$ are placed on the 4-by-4 grid.
It has been observed as early as 1879 by Johnson~\cite{15puzzle1} that in this case there are two placements of the tokens which cannot be obtained from each other by any number of moves.

Clearly the problem gets easier the more ``unoccupied'' vertices there are.
The hardest case with $|H|-1$ tokens was tackled comprehensively by Wilson \cite{wilson} in 1974 but before we turn to his solution we present a formal account of the token game and show how it helps with the linkage problem.
 
Throughout this section let $H$ be a graph and let $\X$ always denote a sequence $\X=X_0,\ldots, X_n$ of vertex sets of $H$ and $\M$ a non-empty sequence $\M=M_1,\ldots, M_n$ of non-trivial paths in $H$.
In our model the sets $X_i$ are ``occupied vertices'', the paths $M_i$ are paths along which the tokens are moved, and $i$ is the ``move count''.

Formally, a pair $(\X,\M)$ is called a \emph{movement on $H$} if for $i=1,\ldots, n$
\begin{enumerate}[(M1)]
	\item the set $X_{i-1}\bigtriangleup X_i$ contains precisely the two end vertices of $M_i$, and
	\item $M_i$ is disjoint from $X_{i-1}\cap X_i$.
\end{enumerate}
Then $n$ is the \emph{length} of $(\X,\M)$, the sets in $\X$ are its \emph{intermediate configurations}, in particular, $X_0$ and $X_n$ are its \emph{first} and \emph{last configuration}, respectively.
The paths in $\M$ are the \emph{moves} of $(\X,\M)$.
A movement with first configuration $X$ and last configuration $Y$ is called an \emph{$X$--$Y$ movement}.
Note that our formal notion of token movements allows a move $M_i$ to have both ends in $X_{i-1}$ or both in $X_i$.
In our intuitive account of the token game this corresponds to ``destroying'' or ``creating'' a pair of tokens on the end vertices of $M_i$.

Let us state some obvious facts about movements.
If $\M$ is a non-empty sequence of non-trivial paths in $H$ and one intermediate configuration $X_i$ is given, then there is a unique sequence $\X$ such that $(\X,\M)$ satisfies (M1).
A pair $(\X,\M)$ is a movement if and only if $((X_{i-1}, X_i), (M_i))$ is a movement for $i=1,\ldots, n$.
This easily implies the following Lemma so we spare the proof.

\begin{lem}\label{thm:concatenation}
Let $(\X, \M) = ((X_0,\ldots, X_n), (M_1,\ldots, M_n))$ and $(\Y, \N) = ((Y_0,\ldots, Y_m), (N_1,\ldots, N_m))$ be movements on $H$ and let $Z\subseteq V(H)$.
\begin{enumerate}[(i)]

\item\label{itm:cc_concatenation}
If $X_n=Y_0$, then the pair
\[\big((X_0,\ldots, X_n=Y_0, \ldots, Y_m), (M_1,\ldots M_n, N_1, \ldots, N_m)\big)\]
is a movement.
We denote it by $(\X, \M)\oplus (\Y, \N)$ and call it the \emph{concatenation} of $(\X, \M)$ and $(\Y, \N)$.

\item\label{itm:cc_union}
If every move of $\M$ is disjoint from $Z$, then the pair
\[\big((X_0\cup Z,\ldots, X_n\cup Z), (M_1,\ldots M_n)\big)\]
is a movement and we denote it by $(\X\cup Z, \M)$.
\end{enumerate}
\end{lem}

Let $(\X,\M)$ be a movement.
For $i=1,\ldots, n$ let $R_i$ be the graph with vertex set $(X_{i-1}\times\{i-1\})\cup (X_i\times\{i\})$ and the following edges:
\begin{enumerate}

\item
$(x,i-1)(x,i)$ for each $x\in X_{i-1}\cap X_i$, and

\item
$(x,j)(y,k)$ where $x$, $y$ are the end vertices of $M_i$ and $j$, $k$ the unique indices such that $(x,j), (y,k)\in V(R_i)$.
\end{enumerate}
Define a multigraph $\R$ with vertex set $\bigcup_{i=0}^n (X_i\times\{i\})$ where the multiplicity of an edge is the number of graphs $R_i$ containing it.
Observe that two graphs $R_i$ and $R_j$ with $i < j$ are edge-disjoint unless $j=i+1$ and $M_i$ and $M_j$ both end in the same two vertices $x$, $y$ of $X_j$, in which case they share one edge, namely $(x,j)(y,j)$.
Our reason to prefer the above definition of $\R$ over just taking the simple graph $\bigcup_{i=1}^n R_i$ is to avoid a special case in the following argument.

Every graph $R_i$ is $1$-regular.
Hence in $\R$ every vertex $(x,i)$ with $0<i<n$ has degree~$2$ as $(x,i)$ is a vertex of $R_j$ if an only if $j=i$ or $j=i+1$.
Every vertex $(x,i)$ with $i=0$ or $i=n$ has degree~$1$ as it only lies in $R_1$ or in $R_n$.
This implies that a component of $\R$ is either a cycle (possibly of length~$2$) avoiding $(X_0\times\{0\})\cup (X_n\times\{n\})$ or a non-trivial path with both end vertices in $(X_0\times\{0\})\cup (X_n\times\{n\})$.
We denote the subgraph of $\R$ consisting of these paths by $\R(\X,\M)$.
Intuitively, each path of $\R(\X,\M)$ traces the position of one token over the course of the token movement or of one pair of tokens which is destroyed or created during the movement.

For vertex sets $X$ and $Y$ we call any $1$-regular graph on $(X\times\{0\})\cup (Y\times\{\infty\})$ an \emph{$(X,Y)$-pairing}.
An $(X,Y)$-pairing is said to be \emph{balanced} if its edges form a perfect matching from $X\times\{0\}$ to $Y\times\{\infty\}$, that is, each edge has one end vertex in $X\times\{0\}$ and the other in $Y\times\{\infty\}$.

The components of $\R(\X,\M)$ induce a $1$-regular graph on $(X_0\times\{0\}) \cup (X_n\times\{n\})$ where two vertices form an edge if and only if they are in the same component of $\R(\X,\M)$.
To make this formally independent of the index $n$, we replace each vertex $(x,n)$ by $(x,\infty)$.
The obtained graph $L$ is an $(X_0, X_n)$-pairing and we call it the \emph{induced pairing} of the movement $(\X,\M)$.
A movement $(\X,\M)$ with induced pairing $L$ is called an \emph{$L$-movement}.
If a movement induces a balanced pairing, then we call the movement \emph{balanced} as well.

Given two sets $X$ and $Y$ and a bijection $\varphi: X\to Y$ we denote by $L(\varphi)$ the balanced $X$--$Y$ pairing where $(x,0)(y,\infty)$ is an edge of $L(\varphi)$ if and only if $y=\varphi(x)$.
Clearly an $X$--$Y$ pairing $L$ is balanced if and only if there is a bijection $\varphi:X\to Y$ with $L=L(\varphi)$.

Given sets $X$, $Y$, and $Z$ let $L_X$ be an $X$--$Y$ pairing and $L_Z$ a $Y$--$Z$ pairing.
Denote by $L_X\oplus L_Z$ the graph on $(X\times\{0\}) \cup (Z\times\{\infty\})$ where two vertices are connected by an edge if and only if they lie in the same component of $L_X\cup L(\id_Y)\cup L_Z$.
The components of $L_X\cup L(\id_Y)\cup L_Z$ are either paths with both ends in $(X\times\{0\}) \cup (Z\times\{\infty\})$ or cycles avoiding that set.
So $L_X\oplus L_Z$ is an $X$--$Z$ pairing end we call it the \emph{concatenation} of $L_X$ and $L_Z$.
The next Lemma is an obvious consequence of this construction (and \autoref{thm:concatenation} (\ref{itm:cc_concatenation})).

\begin{lem}\label{thm:concatenation_pairing}
The induced pairing of the concatenation of two movements is the concatenation of their induced pairings.
\end{lem}

Let $(\X,\M)$ be a movement on $H$.
A vertex $x$ of $H$ is called $(\X,\M)$-\emph{singular} if no move of $\M$ contains $x$ as an inner vertex and $I_x\coloneq\{i\mid x\in X_i\}$ is an \emph{integer interval}, that is, a possibly empty sequence of consecutive integers.
Furthermore, $x$ is called \emph{strongly $(\X,\M)$-singular} if it is $(\X,\M)$-singular and $I_x$ is empty or contains one of $0$ and $n$ where $n$ denotes the length of $(\X,\M)$.
We say that a set $W\subseteq V(H)$ is \emph{$(\X,\M)$-singular} or \emph{strongly $(\X,\M)$-singular} if all its vertices are.
If the referred movement is clear from the context, then we shall drop it from the notation and just write singular or strongly singular.

Note that any vertex $v$ of $H$ that is contained in at most one move of $\M$ is strongly $(\X,\M)$-singular.
Furthermore, $v$ is singular but not strongly singular if it is contained in precisely two moves but neither in the first nor in the last configuration.

The following Lemma shows how to obtain linkages in a graph~$G$ from movements on the auxiliary graph of a regular decomposition of~$G$.
It enables us to apply the results about token movements from this section to our linkage problem.

\begin{lem}\label{thm:tokengame}
Let $(\W,\P)$ be a stable regular decomposition of some graph $G$ and set $\lambda\coloneq\{\alpha\mid P_{\alpha}\text{ is non-trivial}\}$ and $\theta\coloneq\{\alpha\mid P_{\alpha}\text{ is trivial}\}$.
Let $(\X,\M)$ be a movement of length $n$ on a subgraph $\Gamma\subseteq \Gamma(\W,\P)$ and denote its induced pairing by $L$.
If $\theta$ is $(\X,\M)$-singular and $W_a$ and $W_b$ are inner bags of $\W$ with $b-a=2n-1$, then there is a linkage $\Q\subseteq G_{\Gamma}^{\P}[W_{[a,b]}]$ and a bijection $\varphi:E(L)\to \Q$ such that for each $e\in E(L)$ the path $\varphi(e)$ ends in the left $\alpha$-vertex of $W_a$ if and only if $(\alpha,0)\in e$ and $\varphi(e)$ ends in the right $\alpha$-vertex of $W_b$ if and only if $(\alpha,\infty)\in e$.
\end{lem}

\begin{proof}
Let us start with the general observation that for every connected subgraph $\Gamma_0 \subseteq \Gamma(\W,\P)$ and every inner bag $W$ of $\W$ the graph $G_{\Gamma_0}^{\P}[W]$ is connected:
If $\alpha\beta$ is an edge of $\Gamma_0$, then some inner bag of $\W$ contains a $\P$-bridge realising $\alpha\beta$ and so does $W$ by (L8).
In particular, $G_{\Gamma_0}^{\P}[W]$, contains a $P_{\alpha}$--$P_{\beta}$ path.
So $\P_{V(\Gamma_0)}[W]$ must be contained in one component of $G_{\Gamma_0}^{\P}[W]$ as $\Gamma_0$ is connected.
But any vertex of $G_{\Gamma_0}^{\P}[W]$ is in $\P_{V(\Gamma_0)}$ or in a $\P$-bridge attaching to it.
Therefore $G_{\Gamma_0}^{\P}[W]$ is connected.

The proof is by induction on $n$.
Denote the end vertices of $M_1$ by $\alpha$ and $\beta$, that is, $X_0\bigtriangleup X_1 = \{\alpha,\beta\}$.
By definition the induced pairing $L_1$ of $((X_0,X_1), (M_1))$ contains the edges $(\gamma,0)(\gamma,\infty)$ with $\gamma\in X_0\cap X_1$ and w.l.o.g.\ precisely one of $(\alpha,0)(\beta,0)$, $(\alpha,0)(\beta,\infty)$, and $(\alpha,\infty)(\beta,\infty)$.
The above observation implies that $G_{M_1}^{\P}[W_a]$ is connected.
Hence $P_{\alpha}[W_{[a,a+1]}] \cup G_{M_1}^{\P}[W_a] \cup P_{\beta}[W_{[a,a+1]}]$ is connected and thus contains a path $Q$ such that $\Q_1\coloneq \{Q\}\cup \P_{X_0\cap X_1}[W_{[a,a+1]}]$ satisfies the following.
There is a bijection $\varphi_1:E(L_1)\to \Q_1$ such that for each $e\in E(L_1)$ the path $\varphi_1(e)$ ends in the left $\gamma$-vertex of $W_a$ if and only if $(\gamma,0)\in e$ and $\varphi_1(e)$ ends in the right $\gamma$-vertex of $W_{a+1}$ if and only if $(\gamma,\infty)\in e$.
Moreover, the paths of $\Q_1$ are internally disjoint from $W_{a+1}\cap W_{a+2}$.

In the base case $n=1$ the linkage $\Q\coloneq\Q_1$ is as desired.
Suppose that $n\geq 2$.
Then $((X_1,\ldots, X_n), (M_2,\ldots, M_n))$ is a movement and we denote its induced permutation by $L_2$.
\autoref{thm:concatenation_pairing} implies $L = L_1\oplus L_2$.
By induction there is a linkage $\Q_2\subseteq G_{\Gamma}^{\P}[W_{[a+2, b]}]$ and a bijection $\varphi_2:E(L_2)\to \Q_2$ such that for any $e\in E(L_2)$ the path $\varphi_2(e)$ ends in the left $\alpha$-vertex of $W_{a+2}$ (which is the right $\alpha$-vertex of $W_{a+1}$) if and only if $(\alpha, 0)\in e$ and in the right $\alpha$-vertex of $W_b$ if and only if $(\alpha,\infty)\in e$.

Clearly for every $\gamma\in X_1$ the $\gamma$-vertex of $W_{a+1}\cap W_{a+2}$ has degree at most~$1$ in $\Q_1$ and in $\Q_2$.
If a path of $\Q_1$ contains the $\gamma$-vertex of $W_{a+1}\cap W_{a+2}$ and $\gamma\notin X_1$, then $\gamma\in\theta$ so by assumption $I_{\gamma} = \{i\mid \gamma\in X_i\}$ is an integer interval which contains $0$ but not $1$.
This means that no path of $\Q_2$ contains the unique vertex of $P_{\gamma}$.
If the union $\Q_1\cup \Q_2$ of the two graphs $\Q_1$ and $\Q_2$ contains no cycle, then it is a linkage $\Q$ as desired.
Otherwise it only contains such a linkage.
\end{proof}

In the rest of this section we shall construct suitable movements as input for \autoref{thm:tokengame}.
Our first tool to this end is the following powerful theorem\footnote{Wilson stated his theorem for graphs which are neither bipartite, nor a cycle, nor a certain graph $\theta_0$. If $H$ properly contains a triangle, then it satisfies all these conditions and if $H$ itself is a triangle, then our theorem is obviously true.} of Wilson.

\begin{thm}[Wilson 1974]\label{thm:wilson}
	Let $k$ be a postive integer and let $H$ be a graph on $n\geq k+1$ vertices.
	If $H$ is $2$-connected and contains a triangle, then for every bijection $\varphi:X\to Y$ of sets $X, Y\subseteq V(H)$ with $|X| = k = |Y|$ there is a $L(\varphi)$-movement of length $m\leq n!/(n-k)!$ on $H$.
\end{thm}

The given bound on $m$ is not included in the original statement but not too hard to check:
Suppose that $(\X,\M)$ is a shortest $L(\varphi)$-movement and $m$ is its length.
Since $L$ is balanced we may assume that no tokes are ``created'' or ``destroyed'' during the movement, that is, all intermediate configurations have the same size and for every $i$ with $1\leq i\leq m$ there is an injection $\varphi_i:X\to V(H)$ such that the induced pairing of $((X_0,\ldots, X_i), (M_1,\ldots, M_i))$ is $L(\varphi_i)$.
If there were $i<j$ with $\varphi_i=\varphi_j$, then
\[\big((X_0,\ldots, X_i = X_j, X_{j+1},\ldots, X_m),(M_1,\ldots, M_i, M_{j+1}, \ldots, M_m)\big)\]
was an $L(\varphi)$-movement of length $m-j+i<m$ contradicting our choice of $(\X,\M)$.
But there are at most $n!/(n-k)!$ injections from $X$ to $V(H)$ so we must have $m\leq n!/(n-k)!$.

For our application we need to generate $L$-movements where $L$ is not necessarily balanced.
Furthermore, \autoref{thm:tokengame} requires the vertices of $\theta$ to be singular with respect to the generated movement.
\autoref{thm:stars} and \autoref{thm:maxdeg} give a direct construction of movements if some subgraph of $H$ is a large star.
\autoref{thm:corecase} provides an interface to \autoref{thm:wilson} that incorporates the above requirements.
The proofs of these three Lemmas require a few tools:
\autoref{thm:simpletokens} simply states that for sets $X$ and $Y$ of equal size there is a short balanced $X$--$Y$ movement.
\autoref{thm:chooseXY} exploits this to show that instead of generating movements for every choice of $X,Y\subseteq V(H)$ and any $(X,Y)$-pairing $L$ it suffices to consider just one choice of $X$ and $Y$.
\autoref{thm:chooseA} allows us to move  strongly singular vertices from $X$ to $Y$ and vice versa without spoiling the existence of the desired $X$--$Y$ movement.

We call a set $A$ of vertices in a graph $H$ \emph{marginal} if $H-A$ is connected and every vertex of $A$ has a neighbour in $H-A$. 

\begin{lem}\label{thm:simpletokens}
For any two distinct vertex sets~$X$ and~$Y$ of some size~$k$ in a connected graph $H$ and any marginal set $A\subseteq V(H)$ there is a balanced $X$--$Y$ movement of length at most~$k$ on~$H$ such that $A$ is strongly singular.
\end{lem}

\begin{proof}
We may assume that $H$ is a tree and that all vertices of $A$ are leaves of this tree.
This already implies that vertices of $A$ cannot be inner vertices of moves.
Moreover, we may assume that $X\cap Y\cap A=\emptyset$.

We apply induction on $|H|$.
The base case $|H| = 1$ is trivial.
For $|H|>1$ let $e$ be an edge of $H$.
If the two components $H_1$ and $H_2$ of $H-e$ each contain the same number of vertices from $X$ as from $Y$, then for $i=1,2$ we set $X_i\coloneq X\cap V(H_i)$ and $Y_i\coloneq Y\cap V(H_i)$.
By induction there is a balanced $X_i$--$Y_i$ movement $(\X_i,\M_i)$ of length at most $|X_i|$ on $H_i$ such that each vertex of $A$ is strongly $(\X_i,\M_i)$-singular where $i=1,2$.
By \autoref{thm:concatenation} $(\X,\M)\coloneq (\X_1\cup X_2,\M_1)\oplus (\X_2\cup Y_1,\M_2)$ is an $X$--$Y$ movement of length at most $|X_1|+|X_2| = |X| = k$ as desired.
Clearly $(\X,\M)$ is balanced and $A$ is strongly $(\X,\M)$-singular as $H_1$ and $H_2$ are disjoint.
	
So we may assume that for every edge $e$ of $H$ one component of $H-e$ contains more vertices from $Y$ than from $X$ and direct $e$ towards its end vertex lying in this component.
As every directed tree has a sink, there is a vertex $y$ of $H$ such that every incident edge $e$ is incoming, that is, the component of $H-e$ not containing $y$ contains more vertices of $X$ than of $Y$.
As $|X| = |Y|$, this can only be if $y$ is a leaf in $H$ and $y\in Y\setminus X$.	
	
Let $M$ be any $X$--$y$ path and denote its first vertex by $x$.
At most one of $x\in Y$ and $x\in A$ can be true by assumption.
Clearly $((\{x\},\{y\}),(M))$ is an $\{x\}$--$\{y\}$ movement and since $H-y$ is connected, by induction there is a balanced $(X\setminus\{x\})$--$(Y\setminus\{y\})$ movement $(\X',\M')$ of length at most $k-1$ on $H-y$ such that $A$ is strongly singular w.r.t.\ both movements.
As before, \autoref{thm:concatenation} implies that
\[(\X,\M)\coloneq \big((X, (X\setminus\{x\})\cup \{y\}),(M)\big)\oplus (\X'\cup\{y\},\M')\]
is an $X$--$Y$ movement of length at most $k$.
Clearly $(\X,\M)$ is balanced and by construction $A$ is strongly $(\X,\M)$-singular.
\end{proof}

\begin{lem}\label{thm:chooseXY}
Let $k$ be a positive integer and $H$ a connected graph with a marginal set $A$.
Suppose that $X,X', Y',Y\subseteq V(H)$ are sets with $|X|+|Y| = 2k$, $|X'|=|X|$, and $|Y'|=|Y|$ such that $(X\cup X')\cap (Y'\cup Y)$ does not intersect $A$.
If for each $(X', Y')$-pairing $L'$ there is an $L'$-movement $(\X',\M')$ of length at most $n'$ on $H$ such that $A$ is strongly $(\X',\M')$-singular, then for each $(X,Y)$-pairing $L$ there is an $L$-movement $(\X,\M)$ of length at most $n' + 2k$ such that $A$ is $(\X,\M)$-singular and all vertices of $A$ that are not strongly $(\X,\M)$-singular are in $(X'\cup Y')\setminus (X\cup Y)$.
\end{lem}
\begin{proof}
Let $(\X_X, \M_X)$ be a balanced $X$--$X'$ movement of length at most $|X|$ and let $(\X_Y,\M_Y)$ be a balanced $Y'$--$Y$ movement of length at most $|Y|$ such that $A$ is strongly singular w.r.t.\ both movements.
These exist by \autoref{thm:simpletokens}.
For any $X'$--$Y'$ movement $(\X',\M')$ such that $A$ is strongly $(\X',\M')$-singular,
\[(\X,\M)\coloneq (\X_X,\M_X)\oplus (\X',\M')\oplus (\X_Y,\M_Y)\]
is a movement of length at most $|X|+ n' +|Y| = n' + 2k$ by \autoref{thm:concatenation}.

In a slight abuse of the notation we shall write $a\in \M_X$, $a\in \M'$, and $a\in \M_Y$ for a vertex $a\in A$ if there is a move of $\M_X$, $\M'$, and $\M_Y$, respectively, that contains $a$.
Consequently, we write $a\notin M_X$, etc.\ if there is no such move.
The set $A$ is strongly singular w.r.t.\ each of $(\X_X,\M_X)$, $(\X',\M')$, and $(\X_Y,\M_Y)$.
Therefore all moves of $\M$ are internally disjoint from $A$ and each $a\in A$ is contained in at most one move from each of $\M_X$, $\M'$, and $\M_Y$.
Moreover, for each $a\in A$
\begin{enumerate}
\item
$a\in \M_X$ if and only if precisely one of $a\in X$ and $a\in X'$ is true,

\item
$a\in \M'$ if and only if precisely one of $a\in X'$ and $a\in Y'$ is true, and

\item
$a\in\M_Y$ if and only if precisely one of $a\in Y'$ and $a\in Y$ is true.
\end{enumerate}
Clearly $A\setminus (X\cup X'\cup Y'\cup Y)$ is strongly $(\X,\M)$-singular as none of its vertices is contained in a path of $\M$.

Let $a\in X\cap A$.
Then by assumption $a\notin Y\cup Y'$ and thus $a\notin \M_Y$.
If $a\in X'$, then $a\in \M'$ and $a\notin \M_X$.
Otherwise $a\notin X'$ and therefore $a\in \M_X$ and $a\notin \M'$.
In either case $a$ is in at most one move of $\M$ and hence $X\cap A$ is strongly $(\X,\M)$-singular.
A symmetric argument shows that $Y\cap A$ is strongly $(\X,\M)$-singular.

Let $a\in (X'\cup Y')\cap A$ with $a\notin X\cup Y$.
Then $a\in X'\bigtriangleup Y'$ so $a\in \M'$ and precisely one of $a\in \M_X$ and $a\in \M_Y$ is true.

We conclude that every vertex of $a\in A$ is $(\X,\M)$-singular and it is even strongly $(\X,\M)$-singular if and only if $a\notin(X'\cup Y')\setminus (X\cup Y)$.

The induced pairings $L_X$ of $(\X_X,\M_X)$ and $L_Y$ of $(\X_Y,\M_Y)$ are both balanced and it is not hard to see that for a suitable choice of $L'$ the induced pairing $L_X\oplus L'\oplus L_Y$ of $(\X,\M)$ equals $L$.
\end{proof}

\begin{lem}\label{thm:chooseA}
Let $H$ be a connected graph and let $X,Y\subseteq V(H)$.
Suppose that $L$ is an $(X,Y)$-pairing and $(\X,\M)$ an $L$-movement of length $n$.
If $x\in X\cup Y$ is strongly $(\X,\M)$-singular, then the following statements hold.

\begin{enumerate}[(i)]

\item
$(\X\bigtriangleup x, \M)$ is an $(L\bigtriangleup x)$-movement of length $n$ where $\X\bigtriangleup x\coloneq(X_0\bigtriangleup \{x\}, \ldots, X_n\bigtriangleup \{x\})$ and $L\bigtriangleup x$ denotes the graph obtained from $L$ by replacing $(x,0)$ with $(x,\infty)$ or vice versa (at most one of these can be a vertex of $L$).

\item
A vertex $y\in V(H)$ is (strongly) $(\X,\M)$-singular if and only if it is (strongly) $(\X\bigtriangleup x,\M)$-singular.
\end{enumerate}
\end{lem}

\begin{proof}
Clearly $(\X\bigtriangleup x,\M)$ is an $(L\bigtriangleup x)$-movement of length $n$.
As its intermediate configurations differ from those of $\X$ only in $x$, the last assertion is trivial for $y\neq x$.
For $y=x$ note that $\{i\mid x\notin X_i\}$ is an integer interval containing precisely one of $0$ and $n$ because $\{i\mid x\in X_i\}$ is.
\end{proof}

In the final three Lemmas of this section we put our tools to use and construct movements under certain assumptions about the graph.
Note that it is not hard to improve on the upper bounds given for the lengths of the generated movements with more complex proofs.
However, in our main proof we have an arbitrarily long stable regular decomposition at our disposal, so the input movements for \autoref{thm:tokengame} can be arbitrarily long as well.

\begin{lem}\label{thm:stars}
Let $k$ be a positive integer and $H$ a connected graph with a marginal set $A$.
If one of
\begin{enumerate}[a)]

\item
$|A|\geq 2k-1$ and

\item
$|N_H(v)\cap N_H(w)\cap A|\geq 2k-3$ for some edge $vw$ of $H-A$
\end{enumerate}
holds, then for any $X$--$Y$ pairing $L$ such that $X,Y\subseteq V(H)$ with $|X|+|Y| = 2k$ and $X\cap Y\cap A=\emptyset$ there is an $L$-movement $(\X,\M)$ of length at most $3k$ on $H$ such that $A$ is $(\X,\M)$-singular.
\end{lem}
The basic argument of the proof is that that if we place tokens on the leaves of a star but not on its centre, then we can clearly ``destroy'' any given pair of tokens by moving one on top of the other through the centre of the star.

\begin{proof}
Suppose that a) holds.
Let $N_A\subseteq A$ with $|N_A|=2k-1$.
There are sets $X', Y'\subseteq V(H)$ such that
\begin{enumerate}

\item
$|X'|=|X|$ and $|Y'| = |Y|$,

\item
$X\cap N_A\subseteq X'$,

\item
$Y\cap N_A\subseteq Y'$, and

\item
$N_A\subseteq X'\cup Y'$ and $X'\cap Y'\cap A=\emptyset$.

\end{enumerate}
By \autoref{thm:chooseXY} it suffices to show that for each $X'$--$Y'$ pairing $L'$ there is an $L'$-movement $(\X',\M')$ of length at most $k$ on $H$ such that $A$ is strongly $(\X',\M')$-singular.
Assume w.l.o.g.\ that the unique vertex of $(X'\cup Y') \setminus N_A$ is in $X'$.
Repeated application of \autoref{thm:chooseA} implies that the desired $L'$-movement $(\X',\M')$ exists if and only if for every $(X'\cup Y')$--$\emptyset$ pairing $L''$ there is an $L''$-movement $(\X'',\M'')$ of length at most $k$ on $H$ such that $A$ is strongly $(\X'',\M'')$-singular.

Let $L''$ be any $(X'\cup Y')$--$\emptyset$ pairing.
Then $E(L'')=\{(x_i, 0)(y_i,0) \mid i=1,\ldots, k\}$ where $(X'\cup Y')\cap N_A = \{x_1,\ldots, x_k, y_2,\ldots, y_k\}$ and $(X'\cup Y')\setminus N_A = \{y_1\}$.
For $i=0,\ldots, k$ set $X_i\coloneq \{x_j,y_j\mid j>i\}$.
For $i=1,\ldots, k$ let $M_i$ be an $x_i$--$y_i$ path in $H$ that is internally disjoint from $A$.
Then $(\X'',\M'')\coloneq ((X_0,\ldots, X_k), (M_1,\ldots, M_k))$ is an $L''$-movement of length $k$ and obviously $A$ is strongly $(\X'',\M'')$-singular.

Suppose that b) holds and let $N_A\subseteq N_H(v)\cap N_H(w)\cap A$ with $|N_A|=2k-3$ and set $N_B\coloneq\{v,w\}$.
There are sets $X', Y'\subseteq V(H)$ such that
\begin{enumerate}

\item
$|X'|=|X|$ and $|Y'| = |Y|$,

\item
$X\cap N_A\subseteq X'$ and $X'\cap A\subseteq X\cup N_A$,

\item
$Y\cap N_A\subseteq Y'$ and $Y'\cap A\subseteq Y\cup N_A$,

\item
$N_A\subseteq X'\cup Y'$ and $X'\cap Y'\cap A=\emptyset$,


\item
$N_B\subseteq X'$ or $X'\subset N_A\cup N_B$, and

\item
$N_B\subseteq Y'$ or $Y'\subset N_A\cup N_B$.
\end{enumerate}

By \autoref{thm:chooseXY} (see case a) for the details) it suffices to find an $L'$-movement $(\X',\M')$ of length at most $k$ on $H$ such that $A$ is strongly $(\X',\M')$-singular where $L'$ is any $X'$--$Y'$ pairing.
Since $|(X'\cup Y')\setminus N_A| = 3$ we may asssume w.l.o.g.\ that $N_B\subseteq X'$ and $Y'\subseteq N_A\cup \{v\}$.
So either there is $z\in X'\setminus (N_A\cup N_B)$ or $v\in Y'$.
By repeated application of \autoref{thm:chooseA} we may assume that $N_A\subseteq X'$.
This means that $L'$ has the vertices $\bar{N}_A\coloneq N_A\times \{0\}$, $\bar{v}\coloneq (v,0)$, $\bar{w}\coloneq (w,0)$, and $\bar{z}\coloneq (z,0)$ in the first case or $\bar{z}\coloneq (v,\infty)$ in the second case.
So $L'$ must satisfy one of the following.
\begin{enumerate}
\item No edge of $L'$ has both ends in $\{\bar{v}, \bar{w},\bar{z}\}$.
\item $\bar{v}\bar{w}\in E(L')$.
\item $\bar{v}\bar{z}\in E(L')$.
\item $\bar{w}\bar{z}\in E(L')$.
\end{enumerate}
This leaves us with eight cases in total.
Since construction is almost the same for all cases we provide the details for only one of them:
We assume that $v\in Y'$ and $\bar{w}\bar{z}\in E(L')$.
Then $L'$ has edges $(w,0)(v,\infty)$ and $\{(x_i,0)(y_i,0)\mid i=1,\ldots, k-1\}$ where $x_1\coloneq v$ and $X'\cap N_A = \{x_2,\ldots, x_{k-1}, y_1,\ldots, y_{k-1}\}$.
For $i=0,\ldots, k-1$ set $X_i\coloneq \{w\}\cup \bigcup_{j>i}\{x_j, y_j\}$ and let $X_k\coloneq \{v\}$.
Set $M_1\coloneq vy_1$ and $M_i\coloneq x_ivy_i$ for $i=2,\ldots, k-1$ and let $M_k$ be a $w$--$z$ path in $H$ that is internally disjoint from $A$.
Then $(\X',\M')\coloneq ((X_0,\ldots, X_k), (M_1,\ldots, M_k))$ is an $L'$-movement and $A$ is strongly $(\X',\M')$-singular.
\end{proof}

\begin{lem}\label{thm:maxdeg}
Let $k$ be a positive integer and $H$ a connected graph with a marginal set $A$.
	Let $X,Y\subseteq V(H)$ with $|X|+|Y| = 2k$ and $X\cap Y\cap A=\emptyset$.
	Suppose that there is a vertex $v$ of $H-(X\cup Y\cup A)$ such that 
	\[2 |N_H(v)\setminus A| + |N_H(v)\cap A| \geq 2k+1.\]
	Then for any $(X,Y)$-pairing $L$ there is an $L$-movement of length at most $k(k+2)$ on $H$ such that $A$ is singular.
\end{lem}

Although the basic idea is still the same as in \autoref{thm:stars} it gets a little more complicated here as our star might not have enough leaves to hold all tokens at the same time.
Hence we prefer an inductive argument over an explicit construction.

\begin{proof}
Set $N_A\coloneq N_H(v)\cap A$ and $N_B\coloneq N_H(v)\setminus A$.
If $|N_A|\geq 2k-1$, then we are done by \autoref{thm:stars} as $3k\leq k(k+2)$.
So we may assume that $|N_A|\leq 2k-2$.
Under this additional assumption we prove a slightly stronger statement than that of \autoref{thm:maxdeg} by induction on~$k$:
We not only require that $A$ is singular but also that all vertices of $A$ that are not strongly singular are in $N_A\setminus (X\cup Y)$.

The base case $k=1$ is trivial.
Suppose that $k\geq 2$.
There are sets $X',Y'\subseteq V(H)$ such that
\begin{enumerate}

\item
$|X'|=|X|$ and $|Y'| = |Y|$,

\item
$X\cap N_A\subseteq X'$ and $X'\cap A\subseteq X\cup N_A$,

\item
$Y\cap N_A\subseteq Y'$ and $Y'\cap A\subseteq Y\cup N_A$,

\item
$N_A\subseteq X'\cup Y'$ and $X'\cap Y'\cap A=\emptyset$,

\item
$N_B\subseteq X'$ or $X'\subset N_A\cup N_B$,

\item
$N_B\subseteq Y'$ or $Y'\subset N_A\cup N_B$, and

\item
$v\notin X'$ and $v\notin Y'$.
\end{enumerate}

By \autoref{thm:chooseXY} it suffices to find an $L'$-movement $(\X',\M')$ of length at most $k^2$ such that $A$ is strongly $(\X',\M')$-singular where $L'$ is any $X'$--$Y'$ pairing.

If there are $x,y \in X'\cap N_H(v)$ such that $(x,0)(y,0)\in E(L')$, then set $X''\coloneq X'\setminus \{x,y\}$, $Y''\coloneq Y'$, $H''\coloneq H-(A\setminus (X''\cup Y''))$, and $L''\coloneq L' -\{(x,0),(y,0)\}$.
We have $N_{H''}(v)\setminus A = N_B$ and $N_{H''}(v)\cap A = N_A\setminus\{x,y\}$ as $N_A\subseteq X'\cup Y'$.
This means
\[2|N_{H''}(v)\setminus A| + |N_{H''}(v)\cap A| \geq 2|N_B| + |N_A| - 2\geq 2k-1.\]
Hence by induction there is an $L''$-movement $(\X'',\M'')$ of length at most $(k+1)(k-1)$ on $H''$ such that $A$ is singular and all vertices of $A$ that are not strongly singular are in $N_A\setminus (X''\cup Y'')$.
Since $N_A\cap V(H'')\subseteq X''\cup Y''$ the set $A$ is strongly $(\X'',\M'')$-singular.
Then by construction $(\X',\M')\coloneq((X', X''), (xvy))\oplus (\X'',\M'')$ is an $L'$-movement of length at most $k^2$ and $A$ is strongly $(\X',\M')$-singular.

The case $x,y \in Y'\cap N_H(v)$ with $(x,\infty)(y,\infty)\in E(L')$ is symmetric.
If there are $x\in X'\cap N_H(v)$ and $y\in Y'\cap N_H(v)$ such that $(x,0)(y,\infty)\in E(L')$ and at least one of $x$ and $y$ is in $N_A$, then the desired movement exists by \autoref{thm:chooseA} and one of the previous cases.

By assumption
\[2|N_H(v)|\geq 2 |N_B| + |N_A| \geq 2k+1\]
and thus $|N_H(v)|\geq k+1$.
If $N_B\subseteq X'$, then $N_H(v)\subseteq X'\cup (Y'\cap A)$ and there is a pair as above by the pigeon hole principle.
Hence we may assume that $X'\subset N_B$ and by symmetry also that $Y'\subset N_B$.
This implies that $N_A=\emptyset$ and that $L'$ is balanced.

So we have $|X'|=k=|Y'|$, $X',Y'\subseteq N_B$ and $|N_B|\geq k+1$.
It is easy to see that there is an $L'$-movement $(\X',\M')$ of length at most $2k\leq k^2$ on $H[\{v\}\cup N_B]$ such that $A$ is strongly $(\X',\M')$-singular.
\end{proof}

\begin{lem}\label{thm:corecase}	
	Let $n\in\NN$ and let $f:\NN_0\to \NN_0$ be the map that is recursively defined by setting $f(0)\coloneq 0$ and $f(k)\coloneq 2k +2n! +4 + f(k-1)$ for $k >0$.
	Let $k$ be a positive integer and let $H$ be a connected graph on at most $n$ vertices with a marginal set $A$.
	Let $X,Y\subseteq V(H)$ with $|X|+|Y| = 2k$ and $X\cap Y\cap A=\emptyset$ such that neither $X$ nor $Y$ contains all vertices of $H-A$.
	Suppose that there is a block $D$ of $H-A$ such that $D$ contains a triangle and $2|D| + |N(D)| \geq 2k+3$.
	Then for any $(X,Y)$-pairing $L$ there is an $L$-movement of length at most $f(k)$ on $H$ such that $A$ is singular.
\end{lem}

\begin{proof}
Set $N_A\coloneq N(D)\cap A$ and $N_B\coloneq N(D)\setminus A$.
If $|N_A|\geq 2k-1$, then we are done by \autoref{thm:stars} as $3k\leq f(k)$.
So we may assume that $|N_A|\leq 2k-2$.
Under this additional assumption we prove a slightly stronger statement than that of \autoref{thm:corecase} by induction on~$k$:
We not only require that $A$ is singular but also that all vertices of $A$ that are not strongly singular are in $N_A\setminus (X\cup Y)$.
As always the base case $k=1$ is trivial.
Suppose that $k\geq 2$.
\begin{clm}\label{clm:core_D_edge}
Suppose that $|V(D)\setminus X|\geq 1$ and that there is an edge $(x,0)(y,0)\in E(L)$ with $x\in V(D)$ and $y\in V(D)\cup N(D)$.
Let $A'\subseteq A\setminus (X\cup Y)$ with $|A'|\leq 1$.
Then there is an $L$-movement of length at most $|D|! + 1 + f(k-1)$ on $H-A'$ such that $A$ is singular and every vertex of $A$ that is not strongly singular is in $N_A\setminus (X\cup Y)$.
\end{clm}
\begin{proof}
Let $y'$ be a neighbour of $y$ in $D$.
Here is a sketch of the idea:
Move the token from $x$ to $y'$ by a movement on $D$ which we can generate with Wilsons's \autoref{thm:wilson} and then add the move $yy'$.
This ``destroys'' one pair of tokens and allows us to invoke induction.

We assume $x\neq y'$ (in the case $x=y'$ we can skip the construction of $(\X_{\varphi},\M_{\varphi})$ in this paragraph).
Set $X'\coloneq (X\setminus\{x\})\cup\{y'\}$ if $y'\notin X$ and $X'\coloneq X$ otherwise.
The vertices $x$ and $y'$ are both in the $2$-connected graph $D$ which contains a triangle.
By definition $|X\cap V(D)| = |X'\cap V(D)|$ and by assumption both sets are smaller than $|D|$.
Let $\varphi:X\to X'$ any bijection with $\varphi|_{X\setminus V(D)} = \id|_{X\setminus V(D)}$ and $\varphi(x)=y'$.
By \autoref{thm:wilson} there is a balanced $L(\varphi|_{V(D)})$-movement of length at most $|D|!$ on $D$ so by \autoref{thm:concatenation} (\ref{itm:cc_union}) there is a balanced $L(\varphi)$-movement $(\X_{\varphi},\M_{\varphi})$ of length at most $|D|!$ on $H$ such that all its moves are contained in $D$.

Set $X''\coloneq X'\setminus \{y,y'\}$ and let $L'$ be the $X'$--$X''$ pairing with edge set $\{(z,0)(z,\infty)\mid z\in X''\}\cup\{(y,0)(y',0)\}$.
Clearly $((X', X''), (yy'))$ is an $L'$-movement.
Let $L''$ be the $X''$--$Y$ pairing obtained from $L$ by deleting the edge $(x,0)(y,0)$ and substituting every vertex $(z,0)$ with $(\varphi(z), 0)$.
By definition we have $L = L(\varphi)\oplus L'\oplus L''$.

The set $A''\coloneq A'\cup (A\cap \{y\})$ has at most~$2$ elements and thus $2|D| + |N(D)\setminus A''|\geq 2k+1$.
So by induction there is an $L''$-movement $(\X'',\M'')$ of length at most $f(k-1)$ on $H-A''$ such that $A$ is $(\X'',\M'')$-singular and every vertex of $A$ that is not strongly  $(\X'',\M'')$-singular is in $N_A\setminus (X\cup Y)$.
Hence the movement
\[(\X,\M)\coloneq (\X_{\varphi},\M_{\varphi}) \oplus \big((X', X''), (yy')\big) \oplus (\X'',\M'')\]
on $H-A'$ has induced pairing $L$ by \autoref{thm:concatenation_pairing} and length at most $|D|! + 1 + f(k-1)$.
Every move of $\M$ that contains a vertex of $A\setminus\{y\}$ is in $\M''$.
Hence $A\setminus\{y\}$ is $(\X,\M)$-singular and every vertex of $A\setminus\{y\}$ that is not strongly $(\X,\M)$-singular is in $N_A\setminus (X\cup Y)$.
If $y\notin A$, then we are done.
But if $y\in A$, then our construction of $(\X'',\M'')$ ensures that no move of $\M''$ contains $y$.
Therefore $y$ is strongly $(\X,\M)$-singular.
\end{proof}

\begin{clm}\label{clm:core_ND_edge}
Suppose that $|V(D)\setminus X|\geq 2$ and that $L$ has an edge $(x,0)(y,0)$ with $x, y\in N(D)$.
Then there is an $L$-movement of length at most $2|D|!+2+f(k-1)$ on $H$ such that $A$ is singular and every vertex of $A$ that is not strongly singular is in $N_A\setminus (X\cup Y)$.
\end{clm}
\begin{proof}
The proof is very similar to that of \autoref{clm:core_D_edge}.
Let $y'$ be a neighbour of $y$ in $D$.
We assume $y'\in X$ (in the case $y'\notin X$ we can skip the construction of $(\X_{\varphi},\M_{\varphi})$ in this paragraph).
Let $z\in V(D)\setminus X$ and let $X'\coloneq (X\setminus\{y'\})\cup \{z\}$.
Let $\varphi:X\to X'$ be any bijection with $\varphi|_{X\setminus V(D)}=\id_{X\setminus V(D)}$ and $\varphi(y')=z$.
Applying \autoref{thm:wilson} and \autoref{thm:concatenation} as in the proof of \autoref{clm:core_D_edge} we obtain a balanced $L(\varphi)$-movement $(\X_{\varphi},\M_{\varphi})$ of length at most $|D|!$ such that its moves are contained in $D$ (in fact, we could ``free'' the vertex $y'$ with only $|D|$ moves by shifting each token on a $y'$--$z$ path in $D$ by one position towards $z$, but we stick with the proof of \autoref{clm:core_D_edge} here for simplicity).

Set $X''\coloneq (X'\setminus\{y\})\cup \{y'\}$ and let $\varphi':X'\to X''$ be the bijection that maps $y$ to $y'$ and every other element to itself.
Clearly $((X',X''),(yy'))$ is an $L(\varphi')$-movement.
Let $L''$ be the $X''$--$Y$ pairing obtained from $L$ by substituting every vertex $(z,0)$ with $(\varphi'\circ\varphi(z), 0)$.
It is not hard to see that this construction implies $L=L(\varphi)\oplus L(\varphi')\oplus L''$.
Since $(0,x)(0,y')$ is an edge of $L''$ with $x\in V(D)\cup N(D)$ and $y'\in V(D)$ we can apply \autoref{clm:core_D_edge} to obtain an $L''$-movement $(\X'',\M'')$ of length at most $|D|!+1+f(k-1)$ on $H-(\{y\}\cap A)$ (note that $y\in A\cap X$ implies $y\notin Y$ by assumption) such that $A\setminus \{y\}$ is $(\X'',\M'')$-singular and every vertex of $A\setminus\{y\}$ that is not strongly $(\X'',\M'')$-singular is in $N_A\setminus(X\cup Y)$.
Hence the movement
\[(\X,\M)\coloneq (\X_{\varphi},\M_{\varphi}) \oplus \big((X', X''), (yy')\big) \oplus (\X'',\M'')\]
on $H$ has induced pairing $L$ by \autoref{thm:concatenation_pairing} and length at most $2|D|! + 2 + f(k-1)$.
The argument that $A$ is $(\X,\M)$-singular and the only vertices of $A$ that are not strongly $(\X,\M)$-singular are in $N_A\setminus(X\cup Y)$ is the same as in the proof of \autoref{clm:core_D_edge}.
\end{proof}

Pick any vertex $v\in V(D)$.
There are sets $X',Y'\subseteq V(H)$ such that
\begin{enumerate}
\item
$|X'|=|X|$ and $|Y'| = |Y|$,

\item
$X\cap N_A\subseteq X'$ and $X'\cap A\subseteq X\cup N_A$,

\item
$Y\cap N_A\subseteq Y'$ and $Y'\cap A\subseteq Y\cup N_A$,

\item
$N_A\subseteq X'\cup Y'$ and $X'\cap Y'\cap A=\emptyset$,

\item
$N_B\subseteq X'$ or $X'\subset N_A\cup N_B$,

\item
$N_B\subseteq Y'$ or $Y'\subset N_A\cup N_B$,

\item
$v\notin X'$ and $v\notin Y'$,

\item
$V(D)\cup N_B\subseteq X'\cup\{v\}$ or $X'\subset V(D)\cup N(D)$, and

\item
$V(D)\cup N_B\subseteq Y'\cup\{v\}$ or $Y'\subset V(D)\cup N(D)$.
\end{enumerate}
By \autoref{thm:chooseXY} it suffices to find an $L'$-movement $(\X',\M')$ of length at most $f(k)-2k$ on $H$ such that $A$ is strongly $(\X',\M')$-singular where $L'$ is any $X'$--$Y'$ pairing.

Since $n\geq |D|$ we have $f(k)-2k\geq 2|D|!+2+f(k-1)$ and by assumption $v\in V(D)\setminus X'$.
If $L'$ has an edge $(0,x)(0,y)$ with $x\in V(D)$ and $y\in V(D)\cup N(D)$, then by \autoref{clm:core_D_edge} there is an $L'$-movement $(\X',\M')$ of length at most $f(k)-2k$ on $H$ such that $A$ is strongly $(\X',\M')$-singular (recall that $N_A\setminus (X'\cup Y')$ is empty by choice of $X'$ and $Y'$).
So we may assume that $L'$ contains no such edge and by \autoref{thm:chooseA} we may also assume that it has no edge $(x,0)(y,\infty)$ with $x\in V(D)$ and $y\in N_A$.

Counting the edges of $L'$ that are incident with a vertex of $(V(D)\cup N(D))\times\{0\}$ we obtain the lower bound
\[\|L'\|\geq |X'\cap V(D)| + |X'\cap N_B|/2 + |(X\cup Y) \cap N_A|/2.\]
If $V(D)\cup N_B\subseteq X'\cup\{v\}$, then $|X'\cap V(D)|=|D|-1$ and $|X'\cap N_B| = |N_B|$.
Since $|(X'\cup Y') \cap N_A| = |N_A|$ this means
\[2k=2\|L'\|\geq 2(|D|-1) + |N_B| + |N_A| \geq 2|D| + |N(D)| - 2 \geq 2k + 1,\]
a contradiction.
So we must have $X'\subset V(D)\cup N(D)$ and $|V(D)\setminus X'|\geq 2$.
Applying \autoref{clm:core_ND_edge} in the same way as \autoref{clm:core_D_edge} above we deduce that no edge of $L'$ has both ends in $X\times\{0\}$ or one end in $X\times\{0\}$ and the other in $N_A\times\{\infty\}$.
By symmetry we can obtain statements like \autoref{clm:core_D_edge} and \autoref{clm:core_ND_edge} for $Y$ instead of $X$ thus by the same argument as above we may also assume that $Y'\subset V(D)\cup N(D)$ and that no edge of $L'$ has both ends in $Y\times\{\infty\}$ or one end in $N_A\times\{0\}$ and the other in $Y\times\{\infty\}$.
Hence $L'$ is balanced and $N_A=\emptyset$.
Let $\varphi':X'\to Y'$ be the bijection with $L'=L(\varphi')$.

In the rest of the proof we apply the same techniques that we have already used in the proof of \autoref{clm:core_D_edge} and again in that of \autoref{clm:core_ND_edge} so from now on we only sketch how to construct the desired movements.
Furthermore, all constructed movements use only vertices of $V(D)\cup N_B$ for their moves so $A$ is trivially strongly singular w.r.t.\ them.

If $N_B\setminus X'\neq\emptyset$, then we have $X'\subset N_B$ by assumption, so $|N_B|\geq k+1$ and thus also $Y'\subset N_B$.
This is basically the same situation as at the end of the proof for \autoref{thm:maxdeg} so we find an $L'$-movement of length at most $2k\leq f(k)$.
We may therefore assume that $N_B\subseteq X'\cap Y'$.

\begin{clm}\label{clm:core_D_balanced}
Suppose that $L'$ has an edge $(x,0)(y,\infty)$ with $x\in V(D)$ and $y\in N_B$.
Then there is an $L'$-movement $(\X',\M')$ of length at most $|D|!+2+f(k-1)$ such that the moves of $\M'$ are disjoint from $A$.
\end{clm}

\begin{proof}
Let $y'$ be a neighbour of $y$ in $V(D)$.
Since $D$ is $2$-connected, $y'$ has two distinct neighbours $y_l$ and $y_r$ in $D$.
Using \autoref{thm:wilson} we generate a balanced movement of length at most $|D|!$ on $H$ such that all its moves are in $D$ and its induced pairing has the edge $(x,0)(y_l,\infty)$ and its final configuration does not contain $y'$ or $y_r$.
Adding the two moves $yy'y_r$ and $y_ly'y$ then results in a movement $(\X_x, \M_x)$ of length at most $|D|!+2$ whose induced pairing $L_x$ contains the edge $(x,0)(y,0)$.

It is not hard to see that there is a pairing $L''$ such that $L_x\oplus L'' = L$ and this pairing must have the edge $(y,0)(y,\infty)$.
By induction there is an $L''$-movement $(\X'', \M'')$ of length at most $f(k-1)$ such that none if its moves contains $y$.
So $(\X',\M')\coloneq (\X_x, \M_x)\oplus (\X'', \M'')$ is an $L'$ movement of length at most $|D|!+2 + f(k-1)$ as desired.
\end{proof}

\begin{clm}\label{clm:core_ND_balanced}
Suppose that $L'$ has an edge $(x,0)(y,\infty)$ with $x, y\in N_B$.
Then there is an $L'$-movement $(\X',\M')$ of length at most $2|D|!+4+f(k-1)$ such that the moves of $\M'$ are disjoint from $A$.
\end{clm}

\begin{proof}
Let $x'$ be a neighbour of $x$ in $V(D)$.
Since $D$ is $2$-connected, $x'$ has two distinct neighbours $x_l$ and $x_r$ in $D$.
With the same construction as in \autoref{clm:core_D_balanced} we can generate a movement $(\X_x,\M_x)$ of length at most $|D|!+2$ such that its induced pairing $L_x$ contains the edge $(x,0)(x_r,\infty)$ and $(x'',0)(x,\infty)$ for some vertex $x''\in V(D)\cap X'$.
There is a pairing $L''$ such that $L=L_x\oplus L''$ and $L''$ contains the edge $(x_r,0)(y,\infty)$.

By \autoref{clm:core_D_balanced} there is an $L''$-movement $(\X'',\M'')$ of length at most $|D|!+2+f(k-1)$. 
So $(\X',\M')\coloneq (\X_x, \M_x)\oplus (\X'', \M'')$ is an $L'$ movement of length at most $2|D|! + 4 + f(k-1)$ as desired.
\end{proof}

Since $f(k)-2k\geq 2|D|! + 4 + f(k-1)$ we may assume that $N_B\cap Y'=\emptyset$ and thus $N_B=\emptyset$ by \autoref{clm:core_D_balanced} and \autoref{clm:core_ND_balanced}.
This means $X',Y'\subseteq V(D)$ and therefore by \autoref{thm:wilson} there is an $L'$-movement $(\X',\M')$ of length at most $|D|!\leq n!\leq f(k)-2k$.
This concludes the induction and thus also the proof of \autoref{thm:corecase}.
\end{proof}

\section{Relinkages}\label{sec:relinkage}
This section collects several Lemma that compare different foundational linkages for the same stable regular decomposition of a graph.
To avoid tedious repetitions we use the following convention throughout the section.
\begin{con}
Let $(\W,\P)$ be a stable regular decomposition of some length $l\geq 3$ and attachedness $p$ of a $p$-connected graph $G$ and set $\lambda\coloneq\{\alpha\mid P_{\alpha} \text{ is non-trivial}\}$ and $\theta\coloneq\{\alpha\mid P_{\alpha} \text{ is trivial}\}$.
Let $D$ be a block of $\Gamma(\W,\P)[\lambda]$ and let $\kappa$ be the set of all cut-vertices of $\Gamma(\W,\P)[\lambda]$ that are in $D$.
\end{con}

\begin{lem}\label{thm:global_disturbance}
Let $\Q$ be a foundational linkage.
If $\alpha\beta$ is an edge of $\Gamma(\W,\Q)$ with $\alpha\in\lambda$ or $\beta\in\lambda$, then $\alpha\beta$ is an edge of $\Gamma(\W,\P)$.
\end{lem}

\begin{proof}
Some inner bag $W_k$ of $\W$ contains a $\Q$-bridge $B$ realising $\alpha\beta$, that is, $B$ attaches to $Q_{\alpha}$ and $Q_{\beta}$.
For $i=1,\ldots, k-1$ the induced permutation $\pi_i$ of $\Q[W_i]$ is an automorphism of $\Gamma(\W,\P)$ by (L10) and hence so is the induced permutation $\pi = \prod_{i=1}^{k-1}\pi_i$ of $\Q[W_{[1,k-1]}]$.

Clearly the restriction of any induced permutation to $\theta$ is always the identity, so $\pi(\alpha)\in\lambda$ or $\pi(\beta)\in\lambda$.
Therefore $\pi(\alpha)\pi(\beta)$ must be an edge of $\Gamma(\W,\P)$ by (L11) as $B$ attaches to $\Q[W]_{\pi(\alpha)}$ and $\Q[W]_{\pi(\beta)}$.
Since $\pi$ is an automorphism this means that $\alpha\beta$ is an edge of $\Gamma(\W,\P)$.
\end{proof}

The previous Lemma allows us to make statements about any foundational linkage $\Q$ just by looking at $\Gamma(\W,\P)$, in particular, for every $\alpha\in\lambda$ the neighbourhood $N(\alpha)$ of $\alpha$ in $\Gamma(\W,\P)$ contains all neighbours of $\alpha$ in $\Gamma(\W,\Q)$.
The following Lemma applies this argument.

\begin{lem}\label{thm:stable_bridges}
Let $\Q$ be a foundational linkage such that $\Q[W]$ is $p$-attached in $G[W]$ for each inner bag $W$ of $\W$.
If $\lambda_0$ is a subset of $\lambda$ such that $|N(\alpha)\cap\theta|\leq p-3$ for each $\alpha\in\lambda_0$, then every non-trivial $\Q$-bridge in an inner bag of $\W$ that attaches to a path of $\Q_{\lambda_0}$ must attach to at least one other path of $\Q_{\lambda}$.
\end{lem}
\begin{proof}
Suppose for a contradiction that some inner bag $W$ of $\W$ contains a $\Q$-bridge $B$ that attaches to some path $Q_{\alpha}[W]$ with $\alpha\in\lambda_0$ but to no other path of $\Q_{\lambda}[W]$.
Recall that either all foundational linkages for $\W$ satisfy (L7) or none does and $\P$ witnesses the former.
Hence by (L7) a path of $\Q[W]$ is non-trivial if and only if it is in $\Q_{\lambda}[W]$.
So by $p$-attachedness $Q_{\alpha}[W]$ is bridge adjacent to at least $p-2$ paths of $\Q_{\theta}$ in $G[W]$.
Therefore in $\Gamma(\W,\Q)$ the vertex $\alpha$ is adjacent to at least $p-2$ vertices of $\theta$ and by \autoref{thm:global_disturbance} so it must be in $\Gamma(\W,\P)$, giving the desired contradiction.
\end{proof}

\begin{lem}\label{thm:block_separates}
Let $\Q$ be a foundational linkage.
Every $\Q$-bridge $B$ that attaches to a path of $\Q_{\lambda\setminus V(D)}$ has no edge or inner vertex in $G_D^{\Q}$, in particular, it can attach to at most one path of $\Q_{V(D)}$.
\end{lem}
\begin{proof}
By assumption $B$ attaches to some path $Q_{\alpha}$ with $\alpha\in\lambda\setminus V(D)$.
This rules out the possibility that $B$ attaches to only one path of $\Q_{\lambda}$ that happens to be in $\Q_{V(D)}$.
So if $B$ has an edge or inner vertex in $G_D^{\Q}$, then it must realise an edge of $D$.
Hence $B$ attaches to paths $Q_{\beta}$ and $Q_{\gamma}$ with $\beta,\gamma\in V(D)$.
This means that $\alpha\beta$ and $\alpha\gamma$ are both edges of $\Gamma(\W,\Q)$ and thus of $\Gamma(\W,\P)$ by \autoref{thm:global_disturbance}.
But $D$ is a block of $\Gamma(\W,\P)[\lambda]$ so no vertex of $\lambda\setminus V(D)$ can have two neighbours in $D$.
\end{proof}

Given two foundational linkages $\Q$ and $\Q'$ and a set $\lambda_0\subseteq \lambda$, we say that $\Q'$ is a \emph{$(\Q,\lambda_0)$-relinkage} or a \emph{relinkage of $\Q$ on $\lambda_0$} if $Q'_{\alpha}=Q_{\alpha}$ for $\alpha\notin \lambda_0$ and $\Q'_{\lambda_0}\subseteq G_{\lambda_0}^{\Q}$.

\begin{lem}\label{thm:relinkage}
If $\Q$ is a $(\P,V(D))$-relinkage and $\Q'$ a $(\Q,V(D))$-relinkage, then $G_{D}^{\Q'}\subseteq G_{D}^{\Q}$, in particular, $G_{D}^{\Q}\subseteq G_{D}^{\P}$.
\end{lem}
\begin{proof}
Clearly $G_D^{\Q}$ and $G_D^{\Q'}$ are induced subgraphs of $G$ so it suffices to show $V(G_D^{\Q'})\subseteq V(G_D^{\Q})$.
Suppose for a contradiction that there is a vertex $w\in V(G_D^{\Q'})\setminus V(G_D^{\Q})$.
We have $G_D^{\Q'}\cap \Q' = \Q'_{V(D)}\subseteq G_D^{\Q}$ so $w$ must be an inner vertex of a $\Q'$-bridge $B'$.
But $w$ is in $G_{\lambda}-G_D^{\Q}$ and thus in a $\Q$-bridge attaching to a path of $\Q_{\lambda\setminus V(D)}$, in particular, there is a $w$--$\Q_{\lambda\setminus V(D)}$ path $R$ that avoids $G_D^{\Q}\supseteq \Q'_{V(D)}$.
This means $R\subseteq B'$ and thus $B'$ attaches to a path of $\Q'_{\lambda\setminus V(D)}=\Q_{\lambda\setminus V(D)}$, a contradiction to \autoref{thm:block_separates}.
Clearly $\P$ itself is a $(\P,V(D))$-relinkage so $G_{D}^{\Q}\subseteq G_{D}^{\P}$ follows from a special case of the statement we just proved.
\end{proof}

\begin{lem}\label{thm:make_attached}
Let $\Q$ be a $(\P,V(D))$-relinkage.
If in $\Gamma(\W,\P)$ we have $|N(\alpha)\cap\theta|\leq p-3$ for all $\alpha\in\lambda\setminus V(D)$, then there is a $(\Q, V(D))$-relinkage $\Q'$ such that for every inner bag $W$ of $\W$ the linkage $\Q'[W]$ is $p$-attached in $G[W]$ and has the same induced permutation as $\Q[W]$.
Moreover, $\Gamma(\W,\Q')$ contains all edges of $\Gamma(\W,\Q)$ that have at least one end in $\lambda$.
\end{lem}
\begin{proof}
Suppose that some non-trivial $\Q$-bridge $B$ in an inner bag $W$ of $\W$ attaches to a path $Q_{\alpha}=P_{\alpha}$ with $\alpha\in\lambda\setminus V(D)$ but to no other path of $\Q_{\lambda}$.
Then $B$ is also a $\P$-bridge and $\P[W]$ is $p$-attached in $G[W]$ by (L6) so $P_{\alpha}[W]$ must be bridge adjacent to at least $p-2$ paths of $\P_{\theta}$ in $G[W]$ and thus $\alpha$ has at least $p-2$ neighbours in $\theta$, a contradiction.
Hence every non-trivial $\Q$-bridge that attaches to a path of $\Q_{\lambda\setminus V(D)}$ must attach to at least one other path of $\Q_{\lambda}$.

For every inner bag $W_i$ of $\W$ let $\Q'_i$ be the bridge stabilisation of $\Q[W_i]$ in $G[W_i]$.
Then $\Q'_i$ has the same induced permutation as $\Q[W_i]$.
Note that the set $Z$ of all end vertices of the paths of $\Q[W_i]$ is the union of the left and right adhesion set of $W_i$.
So by the $p$-connectivity of $G$ for every vertex $x$ of $G[W_i]-Z$ there is an $x$--$Z$ fan of size $p$ in $G[W_i]$.
This means that $\Q'_i$ is $p$-attached in $G[W_i]$ by \autoref{thm:bridge_stabilisation} (\ref{itm:bs_attached}).

Hence $\Q'\coloneq\bigcup_{i=1}^{l-1} \Q'_i$ is a foundational linkage with $\Q'[W_i] = \Q'_i$ for $i=1,\ldots, l-1$.
Therefore $\Q'[W]$ is $p$-attached in $G[W]$ and $\Q'[W]$ has the same induced permutations as $\Q[W]$ for every inner bag $W$ of $\W$.
There is no $\Q$-bridge that attaches to precisely one path of $\Q_{\lambda\setminus V(D)}$ but to no other path of $\Q_{\lambda}$ so we have $\Q'_{\lambda\setminus V(D)} = \Q_{\lambda\setminus V(D)}$ by \autoref{thm:bridge_stabilisation} (\ref{itm:bs_segments}).
The same result implies $\Q'_{V(D)}\subseteq G_{D}^{\Q}$ so $\Q'$ is indeed a relinkage of $\Q$ on $V(D)$.

Finally, \autoref{thm:bridge_stabilisation} (\ref{itm:bs_bridges}) states that $\Gamma(\W,\Q')$ contains all those edges of $\Gamma(\W,\Q)$ that have at least one end in $\lambda$.
\end{proof}

The ``compressed'' linkages presented next will allow us to fulfil the size requirement that \autoref{thm:corecase} imposes on our block $D$ as detailed in \autoref{thm:compressed}.
Given a subset $\lambda_0\subseteq \lambda$ and a foundational linkage $\Q$, we say that $\Q$ is \emph{compressed to $\lambda_0$} or \emph{$\lambda_0$-compressed} if there is no vertex $v$ of $G_{\lambda_0}^{\Q}$ such that $G_{\lambda_0}^{\Q}-v$ contains $|\lambda_0|$ disjoint paths from the first to the last adhesion set of $\W$ and $v$ has a neighbour in $G_{\lambda}-G_{\lambda_0}^{\Q}$.

\begin{lem}\label{thm:make_compressed}
Suppose that in $\Gamma(\W,\P)$ we have $|N(\alpha)\cap\theta|\leq p-3$ for all $\alpha\in\lambda\setminus V(D)$ and let $\Q$ be a $(\P,V(D))$-relinkage.
Then there is a $V(D)$-compressed $(\Q,V(D))$-relinkage $\Q'$ such that for every inner bag $W$ of $\W$ the linkage $\Q'[W]$ is $p$-attached in $G[W]$.
\end{lem}

\begin{proof}
Clearly $\Q$ itself is a $(\Q,V(D))$-relinkage.
Among all $(\Q,V(D))$-relinkages pick $\Q'$ such that $G_D^{\Q'}$ is minimal.
By \autoref{thm:relinkage} and \autoref{thm:make_attached} we may assume that we picked $\Q'$ such that for every inner bag of $W$ of $\W$ the linkage $\Q'[W]$ is $p$-attached in $G[W]$.

It remains to show that $\Q'$ is $V(D)$-compressed.
Suppose not, that is, there is a vertex $v$ of $G_D^{\Q'}$ such that $v$ has a neighbour in $G_{\lambda}-G_D^{\Q'}$ and $G_D^{\Q'}-v$ contains an $X$--$Y$ linkage $\Q''$ where $X$ and $Y$ denote the intersection of $V(G_D^{\Q'})$ with the first and last adhesion set of $\W$, respectively.

By \autoref{thm:relinkage} we have $G_D^{\Q''}\subseteq G_D^{\Q'}\subseteq G_D^{\Q}$ and thus $\Q''$ is a $(\Q,V(D))$-relinkage as well.
This implies $G_D^{\Q''} = G_D^{\Q'}$ by the minimality of $G_D^{\Q'}$.
The vertex $v$ does not lie on a path of $\Q''$ by construction so it must be in a $\Q''$-bridge $B''$.
But $v$ has a neighbour $w$ in $G_{\lambda}-G_D^{\Q'}$ and there is a $w$--$\Q'_{\lambda\setminus V(D)}$ path $R$ that avoids $G_D^{\Q'}$.
This means $R\subseteq B''$ and thus $B''$ attaches to a path of $\Q''_{\lambda\setminus V(D)}$, contradicting \autoref{thm:block_separates}.
\end{proof}

\begin{lem}\label{thm:compressed}
Let $\Q$ be a $V(D)$-compressed foundational linkage.
Let $V$ be the set of all inner vertices of paths of $\Q_{\kappa}$ that have degree at least~$3$ in $G_D^{\Q}$.
Then the following statements are true.

\begin{enumerate}[(i)]
\item\label{itm:cpr_covering}
Either $2|D|+|N(D)\cap \theta|\geq p$ or $V(G_D^{\Q}) = V(\Q_{V(D)})$ and $\kappa\neq\emptyset$.

\item\label{itm:cpr_wbound}
Either $2|D|+|N(D)| \geq p$ or there is $\alpha\in\kappa$ such that $|Q_{\beta}| \leq |V\cap V(Q_{\alpha})| + 1$ for all $\beta\in V(D)\setminus \kappa$.
\end{enumerate}
\end{lem}
Note that $V(G_D^{\Q}) = V(\Q_{V(D)})$ implies that every $\Q$-bridge in an inner bag of $\W$ that realises an edge of $D$ must be trivial.
\begin{proof}~
\begin{enumerate}[(i)]

\item
Denote by $X$ and $Y$ the intersection of $G_D^{\Q}$ with the first and last adhesion set of $\W$, respectively.
Let $Z$ be the union of $X$, $Y$, and the set of all vertices of $G_D^{\Q}$ that have a neighbour in $G_{\lambda}-G_D^{\Q}$.
Clearly $Z\subseteq V(\Q_{\kappa})\cup X\cup Y$.
Moreover, $G_D^{\Q}-z$ does not contain an $X$--$Y$ linkage for any $z\in Z$:
For $z\in X\cup Y$ this is trivial and for the remaining vertices of $Z$ it holds by the assumption that $\Q$ is $V(D)$-compressed.
Therefore for every $z\in Z$ there is an $X$--$Y$ separation $(A_z,B_z)$ of $G_D^{\Q}$ of order at most $|D|$ with $z\in A_z\cap B_z$.
On the other hand, $\Q_{V(D)}$ is a set of $|D|$ disjoint $X$--$Y$ paths in $G_D^{\Q}$ so every $X$--$Y$ separation has order at least $|D|$.
Hence by \autoref{thm:nested_separations} there is a nested set $\S$ of $X$--$Y$ separations of $G_D^{\Q}$, each of order $|D|$, such that $Z\subseteq Z_0$ where $Z_0$ denotes the set of all vertices that lie in a separator of a separation of $\S$.

We may assume that $(X, V(G_D^{\Q}))\in\S$ and $(V(G_D^{\Q}), Y)\in\S$ so for any vertex $v$ of $G_D^{\Q}-(X\cup Y)$ there are $(A_L,B_L)\in\S$ and $(A_R,B_R)\in\S$ such that $(A_L,B_L)$ is rightmost with $v\in B_L\setminus A_L$ and $(A_R,B_R)$ is leftmost with $v\in A_R\setminus B_R$.
Set $S_L\coloneq A_L\cap B_L$ and $S_R\coloneq A_R\cap B_R$.

Let $z$ be any vertex of $Z_0$ ``between'' $S_L$ and $S_R$, more precisely, $z\in(B_L\setminus A_L) \cap (A_R\setminus B_R)$.
There is a separation $(A_M,B_M)\in\S$ such that its separator $S_M\coloneq A_M\cap B_M$ contains $z$.
Then $z$ witnesses that $A_M\nsubseteq A_L$ and $B_M\nsubseteq B_R$ and thus $(A_M,B_M)$ is neither left of $(A_L,B_L)$ nor right of $(A_R,B_R)$.
But $\S$ is nested and therefore $(A_M,B_M)$ is strictly right of $(A_L,B_L)$ and strictly left of $(A_R,B_R)$.
This means $v\in S_M$ otherwise $(A_M,B_M)$ would be a better choice for $(A_L,B_L)$ or for $(A_R,B_R)$.
So any separator of a separation of $\S$ that contains a vertex of $(B_L\setminus A_L) \cap (A_R\setminus B_R)$ must also contain $v$.

If $v\notin Z_0$, then $(B_L\setminus A_L) \cap (A_R\setminus B_R)\cap Z_0 = \emptyset$.
This means that $S_L\cup S_R$ separates $v$ from $Z$ in $G_D^{\Q}$.
So $S_L\cup S_R\cup V(\Q_{N(D)\cap\theta})$ separates $v$ from $G-G_D^{\Q}$ in $G$.
By the connectivity of $G$ we therefore have 
\[2|D| + |N(D)\cap \theta|\geq\left|S_L\cup S_R \cup V(\Q_{N(D)\cap \theta})\right|\geq p.\]

So we may assume that $V(G_D^{\Q}) = Z_0$
Since every separator of a separation of $\S$ consists of one vertex from each path of $\Q_{V(D)}$ this means $V(\Q_{V(D)})\subseteq V(G_D^{\Q}) = Z_0\subseteq V(\Q_{V(D)})$.
If $\kappa=\emptyset$, then $X\cup Y\cup V(\Q_{N(D)\cap \theta})$ separates $G_D^{\Q}-(X\cup Y)$ from $G-G_D^{\Q}$ in $G$ so this is just a special case of the above argument.

\item
We may assume $\kappa\neq\emptyset$ by (\ref{itm:cpr_covering}) and $\kappa\neq V(D)$ since the statement is trivially true in the case $\kappa=V(D)$.
Pick $\alpha\in\kappa$ such that $|V\cap V(Q_{\alpha})|$ is maximal and let $\beta\in V(D)\setminus\kappa$.
For any inner vertex $v$ of $Q_{\beta}$ define $(A_L,B_L)$ and $(A_R,B_R)$ as in the proof of (\ref{itm:cpr_covering}) and set $V_v\coloneq V\cap (B_L\setminus A_L) \cap (A_R\setminus B_R)$.

By (\ref{itm:cpr_covering}) we have $V_v\subseteq Z_0$ and every separator of a separation of $\S$ that contains a vertex of $V_v$ must also contain $v$.
This means that $V_v\cap V_{v'}=\emptyset$ for distinct inner vertices $v$ and $v'$ of $Q_{\beta}$ since no separator of a separation of $\S$ contains two vertices on the same path of $\Q_{V(D)}$.

Furthermore, $S_L\cup S_R\cup V_v$ separates $v$ from $V(\Q_{\kappa})\cup X\cup Y\supseteq Z$ in $G_D^{\Q}$ so by the same argument as in (\ref{itm:cpr_covering}) we have $2|D| + |N(D)\cap\theta| + |V_v|\geq p$.
Then $|N(D)\cap \lambda |\geq |V_v|$ would imply $2|D| + |N(D)|\geq p$ so we may assume that $|N(D)\cap \lambda | < |V_v|$ for all inner vertices $v$ of $Q_{\beta}$.
Clearly $N(D)\cap \lambda$ is a disjoint union of the sets $(N(\gamma)\cap \lambda)\setminus V(D)$ with $\gamma\in\kappa$ and these sets are all non-empty.
Hence $|\kappa|\leq |N(D)\cap \lambda|$ and thus $|\kappa | + 1\leq |V_v|$ for all inner vertices $v$ of $Q_{\beta}$.

Write $V$ for the inner vertices of $\Q_{\beta}$.
Statement (\ref{itm:cpr_wbound}) easily follows from
\[
|V| (|\kappa| + 1)\leq \sum_{v\in V}|V_v|\leq |V|\leq |\kappa|\cdot |V\cap V(Q_{\alpha})|.
\qedhere\]
\end{enumerate}
\end{proof}

\section{Rural Societies}\label{sec:rural_societies}
In this section we present the answer of Robertson and Seymour to the question whether or not a graph can be drawn in the plane with specified vertices on the boundary of the outer face in a prescribed order.
We will apply their result to subgraphs of a graph with a stable decomposition.

A \emph{society} is a pair $(G, \Omega)$ where $G$ is a graph and $\Omega$ is a cyclic permutation of a subset of $V(G)$ which we denote by $\bar{\Omega}$.
A society $(G,\Omega)$ is called \emph{rural} if there is a drawing of $G$ in a closed disc $D$ such that $V(G)\cap \partial D = \bar{\Omega}$ and $\Omega$ coincides with a cyclic permutation of $\bar{\Omega}$ arising from traversing $\partial D$ in one of its orientations.
We say that a society $(G,\Omega)$ is \emph{$k$-connected} for an integer $k$ if there is no separation $(A,B)$ of $G$ with $|A\cap B| < k$ and $\bar{\Omega}\subseteq B \neq V(G)$.
For any subset $X\subseteq \bar{\Omega}$ denote by $\Omega|X$ the map on $X$ defined by $x\mapsto\Omega^k(x)$ where $k$ is the smallest positive integer such that $\Omega^k(x)\in X$ (chosen for each $x$ individually).
Since $\Omega$ is a cyclic permutation so is $\Omega|X$.

Given two internally disjoint paths $P$ and $Q$ in $G$ we write $PQ$ for the cyclic permutation of $V(P\cup Q)$ that maps each vertex of $P$ to its successor on $P$ if there is one and to the first vertex of $Q-P$ otherwise and that maps each vertex of $Q-P$ to its successor on $Q-P$ if there is one and to the first vertex of $P$ otherwise.

Let $R$ and $S$ be disjoint $\bar{\Omega}$-paths in a society $(G, \Omega)$, with end vertices $r_1$, $r_2$ and $s_1$, $s_2$, respectively.
We say that $\{R, S\}$ is a \emph{cross in $(G,\Omega)$}, if $\Omega|\{r_1,r_2,s_1,s_2\} = (r_1s_1r_2s_2)$ or $\Omega|\{r_1,r_2,s_1,s_2\} = (s_2r_2s_1r_1)$.

The following is an easy consequence of Theorems 2.3 and 2.4 in~\cite{gm9}.
\begin{thm}[Robertson \& Seymour 1990]\label{thm:cross}
Any $4$-connected society is rural or contains a cross.
\end{thm}

In our application we always want to find a cross.
To prevent the society from being rural we force it to violate the implication given in following Lemma which is a simple consequence of Euler's formula.

\begin{lem}\label{thm:euler}
Let $(G,\Omega)$ be a rural society.
If the vertices in $V(G)\setminus \bar{\Omega}$ have degree at least~$6$ on average, then $\sum_{v\in \bar{\Omega}}d_G(v)\leq 4|\bar{\Omega}|-6$.
\end{lem}

\begin{proof}
Since $(G,\Omega)$ is rural there is a drawing of $G$ in a closed disc $D$ with $V(G)\cap \partial D = \bar{\Omega}$.
Let $H$ be the graph obtained by adding one extra vertex $w$ outside $D$ and joining it by an edge to every vertex on $\partial D$.
Writing $b\coloneq |\bar{\Omega}|$ and $i\coloneq |V(G)\setminus \bar{\Omega}|$, Euler's formula implies
\[\|G\| + b = \|H\| \leq 3|H| - 6 = 3(i+b) - 3\]
and thus $\|G\|\leq 3i+2b-3$.
Our assertion then follows from
\[\sum_{v\in \bar{\Omega}}d_G(v) + 6i\leq \sum_{v\in V(G)}d_G(v)= 2\|G\| \leq 6i+ 4b-6\qedhere\]
\end{proof}

In our main proof we will deal with societies where the permutation $\Omega$ is induced by paths (see \autoref{thm:rural_bridge} and \autoref{thm:rural_cutpath}).
But every inner vertex on such a path that has degree~$2$ in $G$ adds slack to the bound provided by \autoref{thm:euler} as it counts~$2$ on the left side but~$4$ on the right.
This is remedied in the following Lemma which allows us to apply \autoref{thm:euler} to a ``reduced'' society where these vertices are suppressed.

\begin{lem}\label{thm:better_society}
Let $(G,\Omega)$ be a society and let $P$ be a path in $G$ such that all inner vertices of $P$ have degree~$2$ in $G$.
Denote by $G'$ the graph obtained from $G$ by suppressing all inner vertices of $P$ and set $\Omega'\coloneq \Omega|V(G')$.
Then $(G',\Omega')$ is rural if and only if $(G,\Omega)$ is.
\end{lem}

\begin{proof}
The graph $G$ is a subdivision of $G'$ so every drawing of $G$ gives a drawing of $G'$ and vice versa.
Hence a drawing witnessing that $(G,\Omega)$ is rural can easily be modified to witness that $(G',\Omega')$ is rural and vice versa.
\end{proof}

Two vertices $a$ and $b$ of some graph $H$ are called \emph{twins} if $N_H(a)\setminus\{b\} = N_H(b)\setminus\{a\}$.
Clearly $a$ and $b$ are twins if and only if the transposition $(ab)$ is an automorphism of $H$.

\begin{lem}\label{thm:rural_bridge}
Let $G$ be a $p$-connected graph and let $(\W,\P)$ be a stable regular decomposition of $G$ of length at least~$3$ and attachedness~$p$.
Set $\theta\coloneq\{\alpha\mid P_{\alpha}\text{ is trivial}\}$ and $\lambda\coloneq\{\alpha\mid P_{\alpha}\text{ is non-trivial}\}$.
Let $\alpha\beta$ be an edge of $\Gamma(\W,\P)[\lambda]$ such that $|N(\alpha)\cap\theta|\leq p-3$, $|N(\beta)\cap\theta|\leq p-3$, and for $N_{\alpha\beta}\coloneq N(\alpha)\cap N(\beta)$ we have $N_{\alpha\beta}\subseteq \theta$ and $|N_{\alpha\beta}|\leq p-5$.
If $\alpha$ and $\beta$ are not twins, then the society $(G_{\alpha\beta}^{\P}, P_{\alpha}P_{\beta}^{-1})$ is rural.
\end{lem}

\begin{proof}~	
\begin{clm}\label{clm:rb_connected}
Every $\P$-bridge with an edge in $G_{\alpha\beta}^{\P}$ must attach to $P_{\alpha}$ and $P_{\beta}$, in particular, $G_{\alpha\beta}^{\P}-P_{\alpha}$ and $G_{\alpha\beta}^{\P}-P_{\beta}$ are both connected.
\end{clm}
\begin{proof}
By \autoref{thm:stable_bridges} every non-trivial $\P$-bridge that attaches to $P_{\alpha}$ or $P_{\beta}$ must attach to another path of $\P_{\lambda}$.
Since $P_{\alpha}$ and $P_{\beta}$ are induced this means that all $\P$-bridges with an edge in $G_{\alpha\beta}^{\P}$ must realise the edge $\alpha\beta$ and hence attach to $P_{\alpha}$ and $P_{\beta}$.
\end{proof}

\begin{clm}\label{clm:rb_ab_separates}
The set $Z$ of all vertices of $G_{\alpha\beta}^{\P}$ that are end vertices of $P_{\alpha}$ or $P_{\beta}$ or have a neighbour in $G-(G_{\alpha\beta}^{\P}\cup \P_{N_{\alpha\beta}})$ is contained in $V(P_{\alpha}\cup P_{\beta})$.
\end{clm}
\begin{proof}
Any vertex $v$ of $G_{\alpha\beta}^{\P}-(P_{\alpha}\cup P_{\beta})$ is an inner vertex of some non-trivial $\P$-bridge $B$ that attaches to $P_{\alpha}$ and $P_{\beta}$.
Since $G_{\alpha\beta}^{\P}$ contains all inner vertices of $B$ the neighbours of $v$ in $G-G_{\alpha\beta}^{\P}$ must be attachments of $B$.
But if $B$ attaches to a path $P_{\gamma}$ with $\gamma\neq\alpha,\beta$, then $\gamma\in N_{\alpha\beta}$ and therefore all neighbours of $v$ are in $G_{\alpha\beta}^{\P}\cup \P_{N_{\alpha\beta}}$.
\end{proof}

\begin{clm}\label{clm:rb_big_society}
The society $(G_{\alpha\beta}^{\P}, P_{\alpha}P_{\beta}^{-1})$ is rural if and only if the society $(G_{\alpha\beta}^{\P}, P_{\alpha}P_{\beta}^{-1}|Z)$ is.
\end{clm}
\begin{proof}
Clearly $(G_{\alpha\beta}^{\P}, P_{\alpha}P_{\beta}^{-1}|Z)$ is rural if $(G_{\alpha\beta}^{\P}, P_{\alpha}P_{\beta}^{-1})$ is.
For the converse suppose that $(G_{\alpha\beta}^{\P}, P_{\alpha}P_{\beta}^{-1}|Z)$ is rural, that is, there is a drawing of $G_{\alpha\beta}^{\P}$ in a closed disc $D$ such that $G_{\alpha\beta}^{\P}\cap \partial D = Z$ and one orientation of $\partial D$ induces the cyclic permutation $P_{\alpha}P_{\beta}^{-1}|Z$ on $Z$.

For the rurality of $(G_{\alpha\beta}^{\P}, P_{\alpha}P_{\beta}^{-1})$ and $(G_{\alpha\beta}^{\P}, P_{\alpha}P_{\beta}^{-1}|Z)$ it does not matter whether the first vertices of $P_{\alpha}$ and $P_{\beta}$ are adjacent in $G_{\alpha\beta}^{\P}$ or not and the same is true for the last vertices of $P_{\alpha}$ and $P_{\beta}$.
So we may assume that both edges exist and we denote the cycle that they form together with the paths $P_{\alpha}$ and $P_{\beta}$ by $C$.

The closed disc $D'$ bounded by $C$ is contained in $D$.
It is not hard to see that the interior of $D'$ is the only region of $D-C$ that has vertices of both $P_{\alpha}$ and $P_{\beta}$ on its boundary.
But every edge of $G_{\alpha\beta}^{\P}$ lies on $C$ or in a $\P$-bridge $B$ with $B-(\P\setminus\{P_{\alpha},P_{\beta}\})\subseteq G_{\alpha\beta}^{\P}$.
By \autoref{clm:rb_connected} such a bridge $B$ must attach to $P_{\alpha}$ and $P_{\beta}$ and in the considered drawing it must therefore be contained in $D'$.
This means $G_{\alpha\beta}^{\P}\subseteq D'$ which implies that $(G_{\alpha\beta}^{\P}, P_{\alpha}P_{\beta}^{-1})$ is rural as desired.
\end{proof}

\begin{clm}\label{clm:rb_4_connected}
For $H\coloneq G_{\alpha\beta}^{\P}$ and $\Omega\coloneq P_{\alpha}P_{\beta}^{-1}|Z$ the society $(H,\Omega)$ is $4$-connected.
\end{clm}
\begin{proof}
Note that $\bar{\Omega} = Z$ since $Z\subseteq V(P_{\alpha}\cup P_{\beta})$ by \autoref{clm:rb_ab_separates}.
Set $T\coloneq V(\P_{N_{\alpha\beta}})$.
Clearly $Z\cup T$ separates $H$ from $G-H$ so for every vertex $v$ of $H-Z$ there is a $v$--$T\cup Z$ fan of size at least $p$ in $G$ as $G$ is $p$-connected.
Since $|T|\leq p-5$ this fan contains a $v$--$Z$ fan of size at least $4$ such that all its paths are contained in $H$.
This means that $(H,\Omega)$ is $4$-connected as desired.
\end{proof}

By the \emph{off-road edges} of a cross $\{R,S\}$ in $(H,\Omega)$ we mean the edges in $E(R\cup S)\setminus E(P_{\alpha}\cup P_{\beta})$.
We call a component of $R\cap(P_{\alpha}\cup P_{\beta})$ that contains an end vertex of $R$ a \emph{tail of $R$}.
We define the \emph{tails of $S$} similarly.

\begin{clm}\label{clm:rb_tails}
If $\{R,S\}$ is a cross in $(H,\Omega)$ whose set $E$ of off-road edges is minimal, then for every $z\in Z\setminus V(R\cup S)$ each $z$--$(R\cup S)$ path in $P_{\alpha}\cup P_{\beta}$ ends in a tail of $R$ or $S$.
\end{clm}
\begin{proof}
Suppose not, that is, there is a $Z$--$(R\cup S)$ path $T$ in $P_{\alpha}\cup P_{\beta}$ such that its last vertex $t$ does not lie in a tail of $R$ or $S$.
W.l.o.g.\ we may assume that $t$ is on $R$.
Since $t$ is not in a tail of $R$ the paths $Rt$ and $tR$ must both contain an edge that is not in $P_{\alpha}\cup P_{\beta}$ so $E(T\cup Rt \cup S)\setminus E(P_{\alpha}\cup P_{\beta})$ and $E(T\cup tR \cup S)\setminus E(P_{\alpha}\cup P_{\beta})$ are both proper subsets of $E$.
But one of $\{T\cup Rt, S\}$ and $\{T\cup tR, S\}$ is a cross in $(H, \Omega)$, a contradiction.
\end{proof}

Suppose now that $\alpha$ and $\beta$ are not twins.
\begin{clm}\label{clm:rb_no_cross}
$(H,\Omega)$ does not contain a cross.
\end{clm}
\begin{proof}
If $(H,\Omega)$ contains a cross, then we may pick a cross $\{R,S\}$ in $(H,\Omega)$ such that its set $E$ of off-road edges is minimal.
Since $Z\subseteq V(P_{\alpha}\cup P_{\beta})$ we may assume w.l.o.g.\ that $\{R,S\}$ satisfies one of the following.
\begin{enumerate}

\item
$R$ and $S$ both have their ends on $P_{\alpha}$.

\item
$R$ has both ends on $P_{\alpha}$. $S$ has one end on $P_{\alpha}$ and one on $P_{\beta}$.

\item
$R$ and $S$ both have one end on $P_{\alpha}$ and one on $P_{\beta}$.
\end{enumerate}

We reduce the first case to the second.
As $P_{\beta}$ contains a vertex of $Z$ but no end of $R$ or $S$ it must be disjoint from $R\cup S$ by \autoref{clm:rb_tails}.
But $R$ and $S$ both contain a vertex outside $P_{\alpha}$ (recall that $P_{\alpha}$ is induced by (L6)) so $R\cup S$ meets $H-P_{\alpha}$ which is connected by \autoref{clm:rb_connected}.
	
Therefore there is a $P_{\beta}$--$(R\cup S)$ in $H-P_{\alpha}$, in particular, there is a $Z$--$(R\cup S)$ path $T$ with its first vertex $z$ in $Z\cap V(P_{\beta})$ and we may assume that its last vertex $t$ is on $S$.
Denote by $v$ the end of $S$ that separates the ends of $R$ in $P_{\alpha}$.
	
Then $\{R, vSt\cup T\}$ is a cross in $(H,\Omega)$ and we may pick a cross $\{R',S'\}$ in $(H,\Omega)$ such that its set $E'$ of off-road edges is minimal and contained in the set $F$ of off-road edges of $\{R, vSt\cup T\}$.
If $R'\cup S'$ contains no edge of $T$, then $E'$ is a proper subset of $E$ as it does not contain $E(S)\setminus E(vSt)$, a contradiction to the minimality of $E$.
Hence $R'\cup S'$ contains an edge of $T$ and hence must meet $P_{\beta}$.
So by \autoref{clm:rb_tails} one of its paths, say $S'$ ends in $P_{\beta}$ as desired.
	
On the other hand, all off-road edges of $\{R',S'\}$ that are incident with $P_{\beta}$ are in $T$ and therefore the remaining three ends of $R'$ and $S'$ must all be on $P_{\alpha}$.
Hence $\{R',S'\}$ is a cross as in the second case.
	
In the second case we reroute $P_{\alpha}$ along $R$, more precisely, we obtain a foundational linkage $\Q$ from $\P$ by replacing the subpath of $P_{\alpha}$ between the two end vertices of $R$ with $R$.
	
The first vertex of $R\cup S$ encountered when following $P_{\beta}$ from either of its ends belongs to a tail of $R$ or $S$ by \autoref{clm:rb_tails}.
Obviously a tail contains precisely one end of $R$ or $S$.
Since $R$ has no end on $P_{\beta}$ and $S$ only one, $(R\cup S)\cap P_{\beta}$ is a tail of $S$, in particular, $R$ is disjoint from $P_{\beta}$ and hence the paths of $\Q$ are indeed disjoint.
	
Clearly $S$ must end in an inner vertex $z$ of $P_{\alpha}$.
By the definition of $Z$ there is a $\P$-bridge $B$ in some inner bag $W$ of $\W$ that attaches to $z$ and to some path $P_{\gamma}$ with $\gamma\in N(\alpha)\setminus N(\beta)$.
But $B\cup S$ is contained in a $\Q$-bridge in $G[W]$ and therefore $\beta\gamma$ is an edge of $B(G[W], \Q[W])$ and thus of $\Gamma(\W,\Q)$ but not of $\Gamma(\W,\P)$.
This contradicts \autoref{thm:global_disturbance}.
	
In the third case \autoref{clm:rb_tails} ensures that the first and last vertex of $P_{\alpha}$ and of $P_{\beta}$ in $R\cup S$ is always in a tail and clearly these tails must all be distinct.
Hence by replacing the tails of $R$ and $S$ with suitable initial and final segments of $P_{\alpha}$ and $P_{\beta}$ we obtain paths $P'_{\alpha}$ and $P'_{\beta}$ such that the foundational linkage $\Q\coloneq(\P\setminus \{P_{\alpha},P_{\beta}\})\cup \{P'_{\alpha},P'_{\beta}\}$ has the induced permutation $(\alpha\beta)$.
Since $P_{\gamma}=Q_{\gamma}$ for all $\gamma\notin\{\alpha,\beta\}$ it is easy to see the there must be an inner bag $W$ of $\W$ such that $\Q[W]$ has induced permutation $(\alpha\beta)$.
But clearly $(\alpha\beta)$ is an automorphism of $\Gamma(\W,\P)$ if and only if $\alpha$ and $\beta$ are twins in $\Gamma(\W,\P)$.
Hence $\Q[W]$ is a twisting disturbance by the assumption that $\alpha$ and $\beta$ are not twins.
This contradicts the stability of $(\W,\P)$ and concludes the proof of \autoref{clm:rb_no_cross}.
\end{proof}

By \autoref{clm:rb_4_connected} and \autoref{thm:cross} the society $(H,\Omega)$ is rural or contains a cross.
But \autoref{clm:rb_no_cross} rules out the latter so $(H,\Omega)$ is rural and by \autoref{clm:rb_big_society} so is $(G_{\alpha\beta}^{\P}, P_{\alpha}P_{\beta}^{-1})$.
\end{proof}

In the previous Lemma we have shown how certain crosses in the graph $H=G_{\alpha\beta}^{\P}$ ``between'' two bridge-adjacent paths $P_{\alpha}$ and $P_{\beta}$ of $\P$ give rise to disturbances.
The next Lemma has a similar flavour;
here the graph $H$ will be the subgraph of $G$ ``between'' $P_{\alpha}$ and $Q_{\alpha}$ where $\alpha$ is a cut-vertex of $\Gamma(\W,\P)[\lambda]$ and $\Q$ a relinkage of $\P$.

\begin{lem}\label{thm:rural_cutpath}
Let $G$ be a $p$-connected graph with a stable regular decomposition $(\W,\P)$ of attachedness $p$ and set $\lambda\coloneq\{\alpha\mid P_{\alpha} \text{ is non-trivial}\}$ and $\theta\coloneq\{\alpha\mid P_{\alpha} \text{ is trivial}\}$.
Let $D$ be a block of $\Gamma(\W,\P)[\lambda]$ and let $\kappa$ be the set of cut-vertices of $\Gamma(\W,\P)[\lambda]$ that are in $D$.
If $|N(\alpha)\cap\theta|\leq p-4$ for all $\alpha\in\lambda$, then there is a $V(D)$-compressed $(\P,V(D))$-relinkage $\Q$ such that  $\Q[W]$ is $p$-attached in $G[W]$ for all inner bags $W$ of $\W$ and for any $\alpha\in\kappa$ and any separation $(\lambda_1,\lambda_2)$ of $\Gamma(\W,\P)[\lambda]$ such that $\lambda_1\cap\lambda_2=\{\alpha\}$ and $N(\alpha)\cap\lambda_2 = N(\alpha)\cap V(D)$ the following statements hold where $H\coloneq G_{\lambda_2}^{\P}\cap G_{\lambda_1}^{\Q}$, $q_1$ and $q_2$ are the first and last vertex of $Q_{\alpha}$, and $Z_1$ and $Z_2$ denote the vertices of $H-\{q_1,q_2\}$ that have a neighbour in $G_{\lambda}-G_{\lambda_2}^{\P}$ and $G_{\lambda}-G_{\lambda_1}^{\Q}$, respectively.

\begin{enumerate}[(i)]

\item\label{itm:rc_Z}
We have $Z_1\subseteq V(P_{\alpha})$ and $Z_2\subseteq V(Q_{\alpha})$.
Furthermore, $Z\coloneq \{q_1,q_2\}\cup Z_1\cup Z_2$ separates $H$ from $G_{\lambda}-H$ in $G-\P_{N(\alpha)\cap\theta}$.

\item\label{itm:rc_connected}
The graph $H$ is connected and contains $Q_{\alpha}$.
The path $P_{\alpha}$ ends in $q_2$.

\item\label{itm:rc_blocks}
Every cut-vertex of $H$ is an inner vertex of $Q_{\alpha}$ and is contained in precisely two blocks of $H$.

\item\label{itm:rc_blockZ}
Every block $H'$ of $H$ that is not a single edge contains a vertex of $Z_1\setminus V(Q_{\alpha})$ and a vertex of $Z_2\setminus V(P_{\alpha})$ that is not a cut-vertex of $H$.
Furthermore, $Q_{\alpha}[W]$ contains a vertex of $Z_2$ for every inner bag~$W$ of~$\W$.

\item\label{itm:rc_new_P}
There is $(\P,V(D))$-relinkage $\P'$ with $\P'=(\Q\setminus\{Q_{\alpha}\})\cup\{P'_{\alpha}\}$ and $P'_{\alpha}\subseteq H$ such that $Z_1\subseteq V(P'_{\alpha})$, $V(P'_{\alpha}\cap Q_{\alpha})$ consists of $q_1$, $q_2$, and all cut-vertices of $H$, and $\P'[W]$ is $p$-attached in $G[W]$ for all inner bags $W$ of $\W$.

\item\label{itm:rc_rural}
Let $H'$ be a block of $H$ that is not a single edge.
Then $P'\coloneq H'\cap P'_{\alpha}$ and $Q'\coloneq H'\cap Q_{\alpha}$ are internally disjoint paths with common first vertex $q'_1$ and common last vertex $q'_2$ and the society $(H',P'Q'^{-1})$ is rural.
\end{enumerate}
\end{lem}

\begin{figure} [htpb]   
\begin{center}
\includegraphics{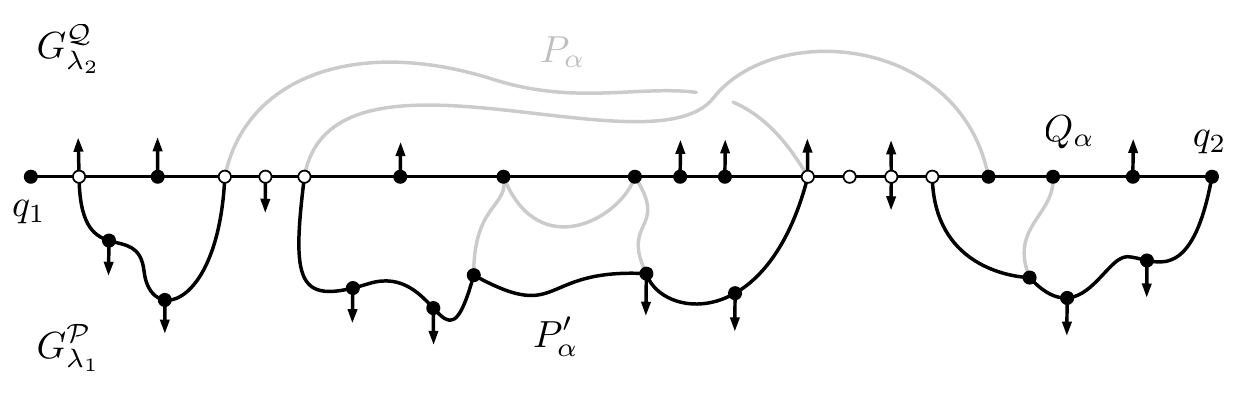}
\caption{The graph $H=G_{\lambda_2}^{\P}\cap G_{\lambda_1}^{\Q}$.}
\label{fig:H}
\end{center}
\end{figure}
\autoref{fig:H} gives an impression of~$H$.
The upper (straight) black $q_1$--$q_2$ path is $Q_{\alpha}$ and everything above it belongs to $G_{\lambda_2}^{\Q}$.
The lower (curvy) black path is $P'_{\alpha}$ and everything below it belongs to $G_{\lambda_1}^{\P}$.
The grey paths are subpaths of $P_{\alpha}$ and, as shown, $P_{\alpha}$ need not be contained in $H$ and need not contain the vertices of $P_{\alpha}\cap P'_{\alpha}$ in the same order as $P'_{\alpha}$.
The white vertices are the cut-vertices of $H$.
The vertices with an arrow up or down symbolise vertices of $Z_2$ and $Z_1$, respectively.
The blocks of $H$ that are not single edges are bounded by cycles in $P'_{\alpha}\cup Q_{\alpha}$ and \autoref{thm:rural_cutpath} (\ref{itm:rc_rural}) states that the part of $H$ ``inside'' such a cycle forms a rural society.

\begin{proof}
For a $(\P,V(D))$-relinkage $\Q$ and $\beta\in\kappa$ any $G_{D}^{\Q}$-path $P\subseteq P_{\beta}$ such that some inner vertex of $P$ has a neighbour in $G_{\lambda}-G_D^{\P}$ is called an \emph{$\beta$-outlet of $\Q$}.
By the \emph{outlet graph of $\Q$} we mean the union of all components of $\P_{\kappa}-G_D^{\Q}$ that have a neighbour in $G_{\lambda}-G_D^{\P}$.
In other words, the outlet graph of $\Q$ is obtained from the union of all $\beta$-outlets for all $\beta\in\kappa$ by deleting the vertices of $G_D^{\Q}$.

Clearly $\P$ itself is a $(\P,V(D))$-relinkage.
Among all $(\P,V(D))$-relinkages pick $\Q'$ such that its outlet graph is maximal.
By \autoref{thm:make_compressed} there is a $V(D)$-compressed $(\Q',V(D))$-relinkage $\Q$ such that $\Q[W]$ is $p$-attached in $G[W]$ for all inner bags $W$ of $\W$.
Note that $G_D^{\Q}\subseteq G_D^{\Q'}$ by \autoref{thm:relinkage}, so the outlet graph of $\Q$ is a supergraph of that of $\Q'$.
Hence by choice of $\Q'$, they must be identical, in particular, the outlet graph of $\Q$ is maximal among the outlet graphs of all $(\P,V(D))$-relinkages.

\begin{clm}\label{clm:rc_stable}
For any foundational linkage $\R$ of $\W$ we have $G_{\lambda_1}^{\R}\cup G_{\lambda_2}^{\R} = G_{\lambda}$ and $G_{\lambda_1}^{\R}\cap G_{\lambda_2}^{\R} = R_{\alpha}$.
\end{clm}
\begin{proof}
By \autoref{thm:global_disturbance} we have $\Gamma(\W,\R)[\lambda]\subseteq \Gamma(\W,\P)[\lambda]$, so $(\lambda_1,\lambda_2)$ is also a separation of $\Gamma(\W,\R)[\lambda]$.
Hence each $\R$-bridge in an inner bag of $\W$ has all its attachments in $\R_{\lambda_1\cup\theta}$ or all in $\R_{\lambda_2\cup\theta}$ and thus $G_{\lambda_1}^{\R}\cup G_{\lambda_2}^{\R} = G_{\lambda}$.
The induced path $R_{\alpha}$ is contained in $G_{\lambda_1}^{\R}\cap G_{\lambda_2}^{\R}$ by definition.
If $G_{\lambda_1}^{\R}\cap G_{\lambda_2}^{\R}$ contains a vertex that is not on $R_{\alpha}$, then it must be in a non-trivial $\R$-bridge that attaches to $R_{\alpha}$ but to no other path of $\R_{\lambda}$.
Such a bridge does not exist by \autoref{thm:stable_bridges} (applied to $\lambda_0\coloneq\lambda$).
\end{proof}

\begin{clm}\label{clm:rc_bridges}
For every vertex $v$ of $H-P_{\alpha}$ there is a $v$--$Z_2$ path in $H-P_{\alpha}$ and for every vertex $v$ of $H-Q_{\alpha}$ there is a $v$--$Z_1$ path in $H-Q_{\alpha}$.
\end{clm}
\begin{proof}
Let $v$ be a vertex of $H-P_{\alpha}\subseteq G_{\lambda_2}^{\P}-P_{\alpha}$.
Then there is $\beta\in\lambda_2\setminus\lambda_1$ such that $v$ is on $P_{\beta}$ or $v$ is an inner vertex of some non-trivial $\P$-bridge attaching to $P_{\beta}$ by \autoref{thm:stable_bridges} and the assumption that $|N(\alpha)\cap \theta|\leq p-4$.
In either case $G_{\lambda_2}^{\P}-P_{\alpha}$ contains a path $R$ from $v$ to the first vertex $p$ of $P_{\beta}$.
But $p$ is also the first vertex of $Q_{\beta}$ and therefore it is contained in $G_{\lambda}-G_{\lambda_1}^{\Q}$.
Pick $w$ on $R$ such that $Rw$ is a maximal initial subpath of $R$ that is still contained in $H$.
Then $w\neq p$ and the successor of $w$ on $R$ must be in $G_{\lambda}-G_{\lambda_1}^{\Q}$.
This means $w\in Z_2$ as desired.
If $v$ is in $H-Q_{\alpha}$, then the argument is similar but slightly simpler as $Q_{\beta}=P_{\beta}$ for all $\beta\in\lambda_1\setminus\lambda_2$.
\end{proof}

\begin{enumerate}[(i)]

\item
Any vertex of $G_{\lambda_2}^{\P}$ that has a neighbour in $G_{\lambda_1}^{\P}-G_{\lambda_2}^{\P}$ must be on $P_{\alpha}$ by \autoref{clm:rc_stable}.
This shows $Z_1\subseteq V(P_{\alpha})$ and by a similar argument $Z_2\subseteq V(Q_{\alpha})$.

A neighbour $v$ of $H$ in $G$ either is in no inner bag of $\W$, it is in $G_{\lambda}$, or it is in $\P_{\theta}$.
In the first case $v$ can only be adjacent to $q_1$ or $q_2$ as these are the only vertices of $H$ in the first and last adhesion set of $\W$.

In the second case, note that $\Q$ is a $(\P,\lambda_2)$-relinkage since $V(D)\subseteq\lambda_2$ and thus \autoref{thm:relinkage} yields $G_{\lambda_2}^{\Q}\subseteq G_{\lambda_2}^{\P}$ which together with \autoref{clm:rc_stable} implies
\begin{align*}
G_{\lambda}
	&= G_{\lambda_1}^{\P}\cup G_{\lambda_2}^{\P}
	= G_{\lambda_1}^{\P}\cup (G_{\lambda_2}^{\P}\cap G_{\lambda_1}^{\Q}) \cup (G_{\lambda_2}^{\P}\cap G_{\lambda_2}^{\Q})\\
	&= G_{\lambda_1}^{\P}\cup H\cup G_{\lambda_2}^{\Q}.
\end{align*}
Hence $v$ is in $G_{\lambda}- G_{\lambda_1}^{\Q}$ or in $G_{\lambda}-G_{\lambda_2}^{\P}$ and thus all neighbours of $v$ in $H$ are in $Z_2$ or $Z_1$, respectively.

In the third case $v$ is the unique vertex of some path $P_{\beta}$ with $\beta\in\theta$.
Let $w$ be a neighbour of $v$ in $H$.
Either $w$ is on $P_{\alpha}$ or there is a $w$--$Z_2$ path in $H$ by \autoref{clm:rc_bridges} which ends on $Q_{\alpha}$ as shown above.
So $\alpha\beta$ is an edge of $\Gamma(\W,\P)$ or of $\Gamma(\W,\Q)$.
The former implies $\beta\in N(\alpha)$ directly and the latter does with the help of \autoref{thm:global_disturbance}.
Hence we have shown that $Z\cup V(\P_{N(\alpha)\cap\theta})$ separates $H$ from the rest of $G$ concluding the proof of (\ref{itm:rc_Z}).

\item
We have $Q_{\alpha}\subseteq G_{\lambda_1}^{\Q}$ by definition and $Q_{\alpha}\subseteq G_{\lambda_2}^{\P}$ since $\Q$ is a $(\P,\lambda_2)$-relinkage.
Hence $Q_{\alpha}\subseteq H$ and some component $C$ of $H$ contains $Q_{\alpha}$.
Suppose that $v$ is a vertex of $H\cap P_{\alpha}$.
Let $w$ be the vertex of $P_{\alpha}$ such that $wP_{\alpha}v$ is a maximal subpath of $P_{\alpha}$ that is still contained in $H$.
Since $P_{\alpha}\subseteq G_{\lambda_2}^{\P}$ we must have $w\in \{q_1\}\cup Z_2\subseteq V(Q_{\alpha})$ and hence $v$ is in $C$.
For any vertex $v$ of $H-P_{\alpha}$ there is a $v$--$Z_2$ path in $H$ by \autoref{clm:rc_bridges} which ends on $Q_{\alpha}$ by (\ref{itm:rc_Z}).
This means that $v$ is in $C$ and hence $H$ is connected.

For every inner bag $W$ of $\W$ the induced permutation $\pi$ of $\Q[W]$ maps each element of $\lambda_1\setminus\lambda_2$ to itself as $\Q$ is $(\P,\lambda_2)$-relinkage.
Moreover, $\pi$ is an automorphism of $\Gamma(\W,\P)$ by (L10) and $\alpha$ is the unique vertex of $\lambda_2$ that has a neighbour in $\lambda_1\setminus\lambda_2$.
This shows $\pi(\alpha)=\alpha$.
Hence $Q_{\alpha}$ and $P_{\alpha}$ must have the same end vertex, namely $q_2$.

\item
Let $v$ be a cut-vertex of $H$.
By (\ref{itm:rc_connected}) it suffices to show that all components of $H-v$ contain a vertex of $Q_{\alpha}$.
First note that every component of $H-v$ contains a vertex of $Z$:
If a vertex $w$ of $H-v$ is not in $Z$, then by (\ref{itm:rc_Z}) and the connectivity of $G$ there is a $w$--$Z$ fan of size at least $p - |N(\alpha)\cap\theta|\geq 2$ in $H$ and at most one of its paths contains $v$.
But any vertex $z\in Z\setminus V(Q_{\alpha})$ is on $P_{\alpha}$ by (\ref{itm:rc_Z}) and the paths $q_1P_{\alpha}z$ and $zP_{\alpha}q_2$ do both meet $Q_{\alpha}$ but at most one can contain $v$ (given that $z\neq v$).
So every component of $H-v$ must contain a vertex of $Q_{\alpha}$ as claimed.
\setcounter{enumi_saved}{\value{enumi}}
\end{enumerate}

\begin{clm}\label{clm:rc_outlet}
A $Q_{\alpha}$-path $P\subseteq P_{\alpha}\cap H$ is an $\alpha$-outlet if and only if some inner vertex of $P$ is in $Z_1$, in particular, every vertex of $Z_1\setminus V(Q_{\alpha})$ lies in a unique $\alpha$-outlet.
Denoting the union of all $\alpha$-outlets by $U$, no two components of $Q_{\alpha}-U$ lie in the same component of $H-U$.
\end{clm}
\begin{proof}
Clearly $Q_{\alpha}\subseteq G_D^{\Q}\cap H\subseteq G_{\lambda_2}^{\Q}\cap G_{\lambda_1}^{\Q} = Q_{\alpha}$ by \autoref{clm:rc_stable}.
Suppose that $P\subseteq P_{\alpha}\cap H$ has some inner vertex $z_1\in Z_1$.
Then $P$ is a $G_D^{\Q}$-path and $z_1$ has a neighbour in $G_{\lambda}-G_{\lambda_2}^{\P}\subseteq G_{\lambda}-G_D^{\P}$ so $P$ is an $\alpha$-outlet.

Before we prove the converse implication let us show that $H\subseteq G_D^{\P}$.
If some vertex $v$ of $H\subseteq G_{\lambda_2}^{\P}$ is not in $G_D^{\P}$, then there is $\beta\in\lambda_2\setminus V(D)$ such that $v$ is on $P_{\beta}$ or $v$ is an inner vertex of a non-trivial $\P$-bridge attaching to $P_{\beta}$.
But $v$ is in $H-P_{\alpha}$ so by \autoref{clm:rc_bridges} and (\ref{itm:rc_Z}) there is a $v$--$Q_{\alpha}$ path in $H$ and hence $\alpha\beta$ is an edge of $\Gamma(\W,\Q)[\lambda]$ and thus of $\Gamma(\W,\P)[\lambda]$ by \autoref{thm:global_disturbance}.
But $(\lambda_1,\lambda_2)$ is chosen such that $N(\alpha)\cap\lambda \subseteq \lambda_1\cup V(D)$, a contradiction.

Suppose that $P$ is an $\alpha$-outlet.
Then some inner vertex $z$ of $P$ has a neighbour in $G_{\lambda}-G_D^{\P}\subseteq G_{\lambda}-H$.
So $z\in Z_1\cup Z_2$ and therefore $z\in Z_1$ as $z\notin V(Q_{\alpha})\supseteq Z_2$ by (\ref{itm:rc_Z}).

To conclude the proof of the claim we may assume for a contradiction that $Q_{\alpha}$ contains vertices $r_1$, $r$, and $r_2$ in this order such that $H-U$ contains an $r_1$--$r_2$ path $R$ and $r$ is the end vertex of an $\alpha$-outlet.
Let $\Q'$ be the foundational linkage obtained from $\Q$ by replacing the subpath $r_1Q_{\alpha}r_2$ of $Q_{\alpha}$ with $R$.
Clearly $\Q'$ is a $(\P, V(D))$-relinkage.
It suffices to show that the outlet graph of $\Q'$ properly contains that of $\Q$ to derive a contradiction to our choice of $\Q$.
By choice of $R$ and the construction of $\Q'$ each $\beta$-outlet of $\Q$ for any $\beta\in\kappa$ is internally disjoint from $\Q'$ and hence is contained in a $\beta$-outlet of $\Q'$.
But $r$ is not on $Q'_{\alpha}$ so it is an inner vertex of some $\alpha$-outlet of $\Q'$ so the outlet graph of $\Q'$ contains that of $\Q$ properly as desired.
\end{proof}

\begin{clm}\label{clm:rc_minsociety}
Let $r_1$ and $r_2$ be the end vertices of an $\alpha$-outlet $P$ of $\Q$.
Then $r_1Q_{\alpha}r_2$ contains a vertex of $Z_2\setminus V(P_{\alpha})$.
\end{clm}
\begin{proof}
We assume that $r_1$ and $r_2$ occur on $Q_{\alpha}$ in this order.
Set $Q\coloneq r_1Q_{\alpha}r_2$.
Clearly $P\cup Q$ is a cycle.
Since $P_{\alpha}$ is induced in $G$, some inner vertex $v$ of $Q$ is not on $P_{\alpha}$.
By \autoref{clm:rc_bridges} there is a $v$--$Z_2$ path $R$ in $H-P_{\alpha}$ and its last vertex $z_2$ must be on $Q_{\alpha}$ (see (\ref{itm:rc_Z})) but not on $P_{\alpha}$.
Finally, \autoref{clm:rc_outlet} implies that $v$ and $z_2$ must be in the same component of $Q_{\alpha}-P$ so both are on $Q$ as desired.
\end{proof}

\begin{enumerate}[(i)]
\setcounter{enumi}{\value{enumi_saved}}
\item
Clearly $H'$ contains a cycle.
Since $Q_{\alpha}$ is induced in $G$ there must be a vertex $v$ in $H'-Q_{\alpha}$ and the $v$--$Z_1$ path in $H-Q_{\alpha}$ that exists by \autoref{clm:rc_bridges} avoids all cut-vertices of $H$ by (\ref{itm:rc_blocks}) and thus lies in $H'-Q_{\alpha}$.
So $H'$ contains a vertex of $Z_1-V(Q_{\alpha})$ which lies on $P_{\alpha}$ by (\ref{itm:rc_Z}) and thus also an $\alpha$-outlet by \autoref{clm:rc_outlet}.
So by \autoref{clm:rc_minsociety} we must also have a vertex of $Z_2\setminus V(P_{\alpha})$ in $H'$ that is neither the first nor the last vertex of $Q_{\alpha}$ in $H'$.

For any inner bag $W$ of $\W$ the end vertices of $Q_{\alpha}[W]$ are cut-vertices of $H$.
By (L8) $G[W]$ contains a $\P$-bridge realising some edge of $D$ that is incident with $\alpha$.
So some vertex of $P_{\alpha}$ has a neighbour in $G_{\lambda}-G_{\lambda_1}^{\P}$.
If $Q_{\alpha}[W]=P_{\alpha}[W]$, then $G_{\lambda_1}^{\Q}[W] = G_{\lambda_1}^{\P}[W]$ so this neighbour is also in $G_{\lambda}-G_{\lambda_1}^{\Q}$ and hence $Q_{\alpha}[W]$ contains a vertex of $Z_2$.
If $Q_{\alpha}[W]\neq P_{\alpha}[W]$, then some block of $H$ in $G[W]$ is not a single edge so by the previous paragraph $Q_{\alpha}[W]$ contains a vertex of $Z_2$.
\setcounter{enumi_saved}{\value{enumi}}
\end{enumerate}

\begin{clm}\label{clm:rc_separator}
Every $Z_1$--$Z_2$ path in $H$ is a $q_1$--$q_2$ separator in $H$.
\end{clm}
\begin{proof}
Suppose not, that is, $H$ contains a $q_1$--$q_2$ path $Q'_{\alpha}$ and a $Z_1$--$Z_2$ path $R$ such that $R$ and $Q'_{\alpha}$ are disjoint.
Clearly $H\cap\Q = Q_{\alpha}$ so $\Q'\coloneq (\Q\setminus\{Q_{\alpha}\})\cup\{Q'_{\alpha}\}$ is a foundational linkage.
The last vertex $r_2$ of $R$ is in $Z_2$ and hence has a neighbour in $G_{\lambda} - G_{\lambda_1}^{\Q}$.
So there is an $r_2$--$\Q'_{\lambda_2\setminus\lambda_1}$ path $R_2$ that meets $H$ only in $r_2$.
Similarly, for the first vertex $r_1$ of $R$ there is an $r_1$--$\Q'_{\lambda_1\setminus\lambda_2}$ path $R_1$ that meets $H$ only in $r_1$.
Then $R_1\cup R\cup R_2$ witness that $\Gamma(\W,\Q')$ has an edge with one end in $\lambda_1\setminus\lambda_2$ and the other in $\lambda_2\setminus\lambda_1$, contradicting \autoref{thm:global_disturbance}.
\end{proof}

\begin{clm}\label{clm:rc_local_P}
Let $H'$ be a block of $H$.
Then $Q\coloneq H'\cap Q_{\alpha}$ is a path and its first vertex $q'_1$ equals $q_1$ or is a cut-vertex of $H$ and its last vertex $q'_2$ equals $q_2$ or is a cut-vertex of $H$.
Furthermore, there is a $q'_1$--$q'_2$ path $P\subseteq H'$ that is internally disjoint from $Q$ such that $Z_1\cap V(H')\subseteq V(P)$ and if a $P$-bridge $B$ in $H'$ has no inner vertex on $Q_{\alpha}$, then for every $z_1\in Z_1\cap V(H')$ the attachments of $B$ are either all on $Pz_1$ or all on $z_1P$.
\end{clm}
\begin{proof}
It follows easily from (\ref{itm:rc_blocks}) that $Q$ is a path and $q'_1$ and $q'_2$ are as claimed.
If $H'$ is the single edge $q'_1q'_2$, then the statement is trivial with $P=Q$ so suppose not.
Our first step is to show the existence of a $q'_1$--$q'_2$ path $R\subseteq H'$ that is internally disjoint from $Q$.

By (\ref{itm:rc_blockZ}) some inner vertex $z_2$ of $Q$ is in $Z_2\setminus V(P_{\alpha})$.
Since $H'$ is $2$-connected there is a $Q$-path $R\subseteq H'-z_2$ with first vertex $r_1$ on $Qz_2$ and last vertex $r_2$ on $z_2Q$.
Pick $R$ such that $r_1Qr_2$ is maximal.
We claim that $r_1=q'_1$ and $r_2=q'_2$.

Suppose for a contradiction that $r_2\neq q'_2$.
By the same argument as before there is $Q$-path $S\subseteq H'-r_2$ with first vertex $s_1$ on $Qr_2$ and last vertex $s_2$ on $r_2Q$.
Note that $s_1$ must be an inner vertex of $r_1Qr_2$ by choice of $R$.
Similarly, $Q$ separates $R$ from $S$ in $H'$ otherwise there was a $Q$-path from $r_1$ to $s_2$ again contradicting our choice of $R$.

But $S$ has an inner vertex $v$ as $Q$ is induced and \autoref{clm:rc_bridges} asserts the existence of a $v$--$Z_1$ path $S'$ in $H-Q_{\alpha}$ which must be disjoint from $R$ as $Q$ separates $S$ from $R$.
So there is a $Z_1$--$Z_2$ path in $z_2Qs_1 \cup s_1Sv \cup S'$ which is disjoint from $Q_{\alpha}r_1Rr_2Q_{\alpha}$ by construction, a contradiction to \autoref{clm:rc_separator}.
This shows $r_2=q'_2$ and by symmetry also $r_1=q'_1$.

Among all $q'_1$--$q'_2$ paths in $H'$ that are internally disjoint from $Q$ pick $P$ such that $P$ contains as few edges outside $P_{\alpha}$ as possible.
To show that $P$ contains all vertices of $Z_1\cap V(H')$ let $z_1\in Z_1\cap V(H')$.
We may assume $z_1\neq q'_1,q'_2$.
If $z_1$ is an inner vertex of $Q$, then $Q$ contains a $Z_1$--$Z_2$ path that is disjoint from $P$, a contradiction to \autoref{clm:rc_separator}.
So there is an $\alpha$-outlet $R$ which has $z_1$ as an inner vertex.
Then $Rz_1\cup Q$ and $z_1R\cup Q$ both contain a $Z_1$--$Z_2$ path and by \autoref{clm:rc_separator} $P$ must intersect both paths.
But $P$ is internally disjoint from $Q$ so it contains a vertex $t_1$ of $Rz_1$ and a vertex $t_2$ of $z_1R$.
If some edge of $t_1Pt_2$ is not on $P_{\alpha}$, then $P'\coloneq q'_1Pt_1P_{\alpha}t_2Pq'_2$ is $q'_1$--$q'_2$ path in $H'$ that is internally disjoint from $Q$ and has fewer edges outside $P_{\alpha}$ than $P$, contradicting our choice of $P$.
This means $t_1Rt_2\subseteq P$ and therefore $z_1$ is on $P$.

Finally, suppose that for some $z_1\in Z_1$ there is a $P$-bridge $B$ in $H'$ with no inner vertex in $Q_{\alpha}$ and attachments $t_1, t_2\neq z_1$ such that $t_1$ is on $Pz_1$ and $t_2$ is on $z_1P$ (this implies $z_1\neq q'_1,q'_2$).
Let $R$ be the $\alpha$-outlet containing $z_1$ and denote its end vertices by $r_1$ and $r_2$.
By \autoref{clm:rc_minsociety} some inner vertex $z_2$ of $r_1Qr_2$ is in $Z_2$.

If $B$ has an attachment in $z_1P-R$, then $z_1Rr_2 \cup z_2Qr_2$ contains a $Z_1$--$Z_2$ path that does not separate $q'_1$ from $q'_2$ in $H'$ and therefore does not separate $q_1$ from $q_2$ in $H$, contradicting \autoref{clm:rc_separator}.
So $B$ has no attachment in $z_1P-R$ and a similar argument implies that $B$ has no attachment in $Pz_1-R$.
So all attachments of $B$ must be in $P\cap R\subseteq P_{\alpha}$.
As $R\cup B$ contains a cycle and $P_{\alpha}$ is induced some vertex $v$ of $B$ is not on $P_{\alpha}$.
But then \autoref{clm:rc_bridges} implies the existence of a $v$--$Z_2$ path that avoids $P_{\alpha}$ and hence uses only inner vertices of $B$, in particular, some inner vertex of $B$ is in $Z_2\subseteq V(Q_{\alpha})$, contradicting our assumption and concluding the proof of this claim.
\end{proof}

\begin{enumerate}[(i)]
\setcounter{enumi}{\value{enumi_saved}}
\item
Applying \autoref{clm:rc_local_P} to every block $H'$ of $H$ and uniting the obtained paths $P$ gives a $q_1$--$q_2$ path $R\subseteq H$ such that $Z_1\subseteq V(R)$ and $V(R\cap Q_{\alpha})$ consists of $q_1$, $q_2$, and all cut-vertices of $H$.
Moreover, for every $z_1\in Z_1$ a $P$-bridge $B$ in $H$ that has no inner vertex in $Q_{\alpha}$ has all its attachments in $Rz_1$ or all in $z_1R$.

Set $\Q'\coloneq (\Q\setminus \{Q_{\alpha}\})\cup\{R\})$.
Let $\P'$ be the foundational linkage obtained by uniting the bridge stabilisation of $\Q'[W]$ in $G[W]$ for all inner bags $W$ of $\W$.
Then $\P'[W]$ is $p$-attached in $G[W]$ for all inner bags $W$ of $\W$ by \autoref{thm:bridge_stabilisation}.

To show $P'_{\beta} = Q_{\beta}$ for all $\beta\in\lambda\setminus\{\alpha\}$ it suffices by \autoref{thm:bridge_stabilisation} to check that every non-trivial $\Q'$-bridge $B'$ that attaches to $Q'_{\beta}$ attaches to at least one other path of $\Q'_{\lambda}$.
If $B'$ is disjoint from $H$ it is also a $\Q$-bridge and thus attaches to some path $Q_{\gamma}=Q'_{\gamma}$ with $\gamma\in\lambda\setminus\{\alpha,\beta\}$ by \autoref{clm:rc_stable}.
If $B'$ contains a vertex of $H$, then it attaches to $Q'_{\alpha}=R$ as $H$ is connected (see (\ref{itm:rc_connected})) and $\Q'\cap H = R$.

To verify $P'_{\alpha}\subseteq H$ we need to show $B'\subseteq H$ for every $\Q'$-bridge $B'$ that attaches to $R$ but to no other path of $\Q'_{\lambda}$.
Clearly for every vertex $v$ of $G_{\lambda_1}^{\P}-P_{\alpha}$ there is a $v$--$\P_{\lambda_1\setminus\{\alpha\}}$ path in $G_{\lambda_1}^{\P}-P_{\alpha}$.
Similarly, for every vertex $v$ of $G_{\lambda_2}^{\Q}-Q_{\alpha}$ there is a $v$--$\Q_{\lambda_2\setminus\{\alpha\}}$ path in $G_{\lambda_2}^{\Q}-Q_{\alpha}$.
But $Q'_{\beta} = P_{\beta}$ for all $\beta\in\lambda_1\setminus\{\alpha\}$ and $Q'_{\beta} = Q_{\beta}$ for all $\beta\in\lambda_2\setminus\{\alpha\}$ and $G_{\lambda} - H = (G_{\lambda_1}^{\P}-P_{\alpha})\cup (G_{\lambda_2}^{\Q}-Q_{\alpha})$.
This means that $B'$ cannot contain a vertex of $G_{\lambda}-H$ and thus $B'\subseteq H$ as desired.

We have just shown that every bridge $B'$ as above is an $R$-bridge in $H$.
By construction and the properties (\ref{itm:rc_Z}) and (\ref{itm:rc_blockZ}) every component of $Q_{\alpha}-R$ contains a vertex of $Z_2$ and hence lies in a $\Q'$-bridge attaching to some path $Q'_{\beta}$ with $\beta\in\lambda_2\setminus\{\alpha\}$.
So $B'$ is an $R$-bridge in $H$ with no inner vertex in $Q_{\alpha}$ and therefore there must be $z_1, z'_1\in Z_1\cup\{q_1,q_2\}$ such that $z_1Rz'_1$ contains all attachments of $B'$ and no inner vertex of $z_1Rz'_1$ is in $Z_1$.
By \autoref{thm:bridge_stabilisation} this implies that $P'_{\alpha}$ contains no vertex of $Q_{\alpha}-R$ and $Z_1\subseteq V(P'_{\alpha})$.
On the other hand, $P'_{\alpha}$ must clearly contain the end vertices of $R$ and all cut-vertices of $H$.
This concludes the proof of (\ref{itm:rc_new_P}).

\item
We first show that $(H',\Omega)$ is rural where $\Omega\coloneq P'Q'^{-1}|Z$ where $Z'\coloneq Z\cap V(H')$.
Since $H$ is connected and $H\cap \P'=P_{\alpha}$ we must have $\beta\in N(\alpha)\cap\theta$ for each path $P_{\beta}$ with $\beta\in\theta$ whose unique vertex has a neighbour in $H$.
So the set $T$ of all vertices of $\P_{\theta}$ that are adjacent to some vertex of $H'$ has size at most $p-4$ by assumption.
Clearly $Z'\cup T$ separates $H'$ from the rest of $G$ so for every vertex $v$ of $H'-Z'$ there is a $v$--$(Z'\cup T)$ fan of size at least $p$ and hence a $v$--$Z$ fan of size at least $4$.
Hence $(H',\Omega)$ is $4$-connected and hence it is rural or contains a cross by \autoref{thm:cross}.

Suppose for a contradiction that $(H',\Omega)$ contains a cross.
By the \emph{off-road edges} of a cross $\{R, S\}$ in $(H',\Omega)$ we mean edge set $E(R\cup S)\setminus E(P'\cup Q')$.
We call a component of $R\cap (P'\cup Q')$ that contains an end of $R$ a \emph{tail of $R$} and define the \emph{tails of $S$} similarly.

\begin{clm}\label{clm:rc_tails}
	If $\{R,S\}$ is a cross in $(H',\Omega)$ such that its set of off-road edges is minimal, then for every $z\in Z$ that is not in $R\cup S$ the two $z$--$(R\cup S)$ paths in $P'\cup Q'$ both end in a tail of $R$ or $S$.
\end{clm}
The proof is the same as for \autoref{clm:rb_tails} so we spare it.

\begin{clm}\label{clm:rc_blockstable}
Every non-trivial $(P'\cup Q')$-bridge $B$ in $H'$ has an attachment in $P'-Q'$ and in $Q'-P'$.
\end{clm}
\begin{proof}
Let $v$ be an inner vertex of $B$.
Then $H-Q_{\alpha}$ contains a $v$--$Z_1$ path by \autoref{clm:rc_bridges} so $B$ must attach to $P'$.
Note that $v$ is in a non-trivial $\P'$-bridge $B'$ and $B'\subseteq G_{\lambda_2}^{\P}$ since $Z_1\subseteq V(P'_{\alpha})$.
Furthermore, $B'$ must attach to a path $P'_{\beta}=Q_{\beta}$ with $\beta\in\lambda_2\setminus\lambda_1$:
This is clear if $B'$ does not attach to $P'_{\alpha}$ and follows from \autoref{clm:rc_stable} if it does.
So $B'$ contains a path $R$ from $v$ to $G_{\lambda_2}^{\Q}-Q_{\alpha}$ that avoids $P'$.
But any such path contains a vertex of $Z_2$ (see (\ref{itm:rc_Z})) and $R$ does not contain $q'_1$ and $q'_2$ so some initial segment of $R$ is a $v$--$Z_2$ path in $H'-P'$ as desired.
\end{proof}

\begin{clm}\label{clm:rc_goodcross}
There is a cross $\{R',S'\}$ in $(H',\Omega)$ such that its set of off-road edges is minimal and neither $P'$ nor $Q'$ contains all ends of $R'$ and $S'$.
\end{clm}
\begin{proof}
Pick a cross $\{R,S\}$ in $(H',\Omega)$ such that its set $E$ of off-road edges is minimal.
We may assume that $P'$ contains all ends of $R$ and $S$.
By (\ref{itm:rc_blockZ}) some inner vertex $z_2$ of $Q'$ is in $Z_2$.
So if $R\cup S$ contains an inner vertex of $Q'$, then $Q'-P'$ contains a $Z_2$--$(R\cup S)$ path $T$ whose last vertex $t$ is an inner vertex of $R$ say.
Clearly one of $\{Rt\cup T, S\}$ and $\{tR\cup T, S\}$ is a cross in $(H',\Omega)$ whose set of off-road edges is contained in that of $\{R,S\}$ and hence is minimal as well.
So either we find a cross $\{R',S'\}$ as desired or $Q'-P'$ is disjoint from $R\cup S$.

But $(R\cup S) - P'$ must be non-empty as $P'$ is induced in $G$.
So by \autoref{clm:rc_blockstable} there is a $Q'$--$(R\cup S)$ path in $H'-P'$, in particular, there is a $Z_2$--$(R\cup S)$ path $T$ in $H'-P'$ and we may assume that its last vertex $t$ is on $R$.
Again one of $\{Rt\cup T, S\}$ and $\{tR\cup T, S\}$ is a cross in $(H',\Omega)$ and we denote its set of off-road edges by $F$.
Pick a cross $(R',S')$ in $(H,\Omega)$ such that its set $E'$ of off-road edges minimal and $E'\subseteq F$.

Since $t$ is not on $P'$ each of $Rt$ and $tR$ contains an edge that is not in $P'\cup Q'$ so $F\setminus E(T)$ is a proper subset of $E$.
This means that $E'$ must contain an edge of $T$ by minimality of $E$ and hence it must contain $F\cap E(T)$ so $R'\cup S'$ contains a vertex of $Q'-P'$ and we have already seen that we are done in this case, concluding the proof of the claim.
\end{proof}

\begin{clm}\label{clm:rc_goodpaths}
	For $i=1,2$ there is a $q'_i$--$(R'\cup S')$ path $T_i$ in $H'$ such that $T_1$ and $T_2$ end on one path of $\{R',S'\}$ and the other path has its ends in $Z_1$ and $Z_2$.
\end{clm}
\begin{proof}
It is easy to see that by construction one path of $\{R',S'\}$, say $S'$, has one end in $Z_1\setminus \{q'_1,q'_2\}$ and the other in $Z_2\setminus \{q'_1,q'_2\}$.
If for some $i$ the vertex $q'_i$ is in $R'\cup S'$, then it must be on $R'$ and there is a trivial $q'_i$--$R'$ path $T_i$.
We may thus assume that neither of $q'_1$ and $q'_2$ is in $R'\cup S'$.

So $P'\cup Q'$ contains two $q'_1$--$(R'\cup S')$ paths $T_1$ and $T'_1$ that meet only in $q'_1$.
By \autoref{clm:rc_tails} $T_1$ and $T'_1$ must both end in a tail of $R'$ or $S'$.
But $(R',S')$ is a cross and no inner vertex of $T_1\cup T'_1$ is an end of $R'$ or $S'$ so we may assume that $T_1$ meets a tail of $R'$.
By the same argument we find a $q'_2$--$(R'\cup S')$ path $T_2$ that end in a tail of $R'$.
\end{proof}

To conclude the proof that $(H',\Omega)$ is rural note that \autoref{clm:rc_goodpaths} implies the existence of a $Z_1$--$Z_2$ path in $H$ that does not separate $q_1$ from $q_2$ in $H$ and hence contradicts \autoref{clm:rc_separator}.
So $(H',\Omega)$ is rural and (\ref{itm:rc_rural}) follows from this final claim:

\begin{clm}\label{clm:rc_big_society}
The society $(H',\Omega)$ is rural if and only if the society $(H',P'Q'^{-1})$ is.
\end{clm}
This holds by a simpler version of the proof of \autoref{clm:rb_big_society} where \autoref{clm:rc_blockstable} takes the role of \autoref{clm:rb_connected}. 
\end{enumerate}
\end{proof}
\section{Constructing a Linkage}\label{sec:mainproof}

In our main theorem we want to construct the desired linkage in a long stable regular decomposition of the given graph.
That decomposition is obtained by applying \autoref{thm:stablepd} which may instead give a subdivision of $K_{a,p}$.
This outcome is even better for our purpose as stated by the following Lemma.
\begin{lem}\label{thm:link}
Every $2k$-connected graph containing a $TK_{2k,2k}$ is $k$-linked.
\end{lem}

\begin{proof}
Let $G$ be a $2k$-connected graph and $S = \{s_1,\ldots, s_k\}$ and $T = \{t_1,\ldots, t_k\}$ two disjoint sets in $V(G)$ of size $k$ each.
We need to find a system of $k$ disjoint $S$--$T$ paths linking $s_i$ to $t_i$ for $i=1,\ldots, k$.

By assumption $G$ contains a subdivision of $K_{2k,2k}$, so there are disjoint sets $A, B\subseteq V(G)$ of size $2k$ each and a system $\Q$ of internally disjoint paths in $G$ such that for every pair $(a,b)$ with $a\in A$ and $b\in B$ there exists a unique $a$--$b$ path in $\Q$ which we denote by $Q_{ab}$.

By the connectivity of $G$, there is a system $\P$ of $2k$ disjoint $(S\cup T)$--$(A\cup B)$ paths (with trivial members if $(S\cup T)\cap (A\cup B)\neq \emptyset$).
Pick $\P$ such that it has as few edges outside of $\Q$ as possible.
Our aim is to find suitable paths of $\Q$ to link up the paths of $\P$ as desired.
We denote by $A_1$ and $B_1$ the vertices of $A$ and $B$, respectively, in which a path of $\P$ ends, and let $A_0 \coloneq A \setminus A_1$ and $B_0 \coloneq B\setminus B_1$.

The paths of $\P$ use the system $\Q$ sparingly:
Suppose that for some pair $(a, b)$ with $a\in A_0$ and $b\in B$, the path $Q_{ab}$ intersects a path of $\P$.
Follow $Q_{ab}$ from $a$ to the first vertex $v$ it shares with any path of $\P$, say $P$.
Replacing $P$ by $Pv \cup Q_{ab}v$ in $\P$ does not give a system with fewer edges outside $\Q$ by our choice of $\P$.
In particular, the final segment $vP$ of $P$ must have no edges outside $\Q$.
This means $vP = vQ_{ab}$, that is, $P$ is the only path of $\P$ meeting $Q_{ab}$ and after doing so for the first time it just follows $Q_{ab}$ to $b$.
Clearly the symmetric argument works if $a\in A$ and $b\in B_0$.
Hence
\begin{enumerate}
	\item $Q_{ab}$ with $a\in A_0$ and $b\in B_0$ is disjoint from all paths of $\P$,
	\item $Q_{ab}$ with $a\in A_1$ and $b\in B_0$ or with $a\in A_0$ and $b\in B_1$ is met by precisely one path of $\P$, and
	\item $Q_{ab}$ with $a\in A_1$ and $b\in B_1$ is met by at least two paths of $\P$.
\end{enumerate}

In order to describe precisely how we link the paths of $\P$, we fix some notation.
Since $|A_0| + |A_1| = |A| = 2k = |\P| = |A_1| + |B_1|$, we have $|A_0| = |B_1|$ and similarly $|A_1| = |B_0|$.
Without loss of generality we may assume that $|B_0| \geq |A_0| = |B_1|$ and therefore $|B_0|\geq k$.
So we can pick $k$ distinct vertices $b_1,\ldots, b_k\in B_0$ and an arbitrary bijection $\varphi: B_1\to A_0$.
For $x\in S\cup T$ denote by $P_x$ the unique path of $\P$ starting in $x$ and by $x'$ its end vertex in $A\cup B$.

For each $i$ and $x=s_i$ or $x=t_i$ set
\[
R_x \coloneq \begin{cases} Q_{x'b_i} & x'\in A_1\\ Q_{\varphi(x')x'} \cup Q_{\varphi(x')b_i} & x'\in B_1\end{cases}.
\]
By construction $R_x$ and $R_y$ intersect if and only if $x,y\in \{s_i,t_i\}$ for some $i$, i.e.\ they are equal or meet exactly in $b_i$.
The paths $P_x$ and $R_y$ intersect if and only if $P_x$ ends in $y'$, that is, if $x = y$.
Thus for each $i=1,\ldots, k$ the subgraph $C_i\coloneq P_{s_i} \cup R_{s_i'} \cup R_{t_i'} \cup P_{t_i}$ of $G$ is a tree containing $s_i$ and $t_i$.
Furthermore, these trees are pairwise disjoint, finishing the proof.
\end{proof}

We now give the proof of the main theorem, \autoref{thm:main}.  We restate the theorem before proceeding with the proof.
\newtheorem*{thmmain}{\autoref{thm:main}}
\begin{thmmain}
For all integers $k$ and $w$ there exists an integer $N$ such that a graph $G$ is $k$-linked if
\[\kappa(G)\geq 2k+3,\qquad \tw(G)<w,\quad\text{and}\quad |G|\geq N.\]
\end{thmmain}

\begin{proof}
Let $k$ and $w$ be given and let $f$ be the function from the statement of \autoref{thm:corecase} with $n\coloneq w$.
Set
\begin{align*}
n_0 &\coloneq (2k+1)(n_1-1) + 1\\
n_1 &\coloneq \max\{(2k-1)\binom{w}{2k}, 2k(k+3) + 1, 12k+ 4, 2f(k)+1\}\\
\end{align*}
We claim that the theorem is true for the integer $N$ returned by \autoref{thm:stablepd} for parameters $a = 2k$, $l = n_0$, $p = 2k + 3$, and $w$.
Suppose that $G$ is a $(2k+3)$-connected graph of tree-width less than $w$ on at least $N$ vertices.
We want to show that $G$ is $k$-linked.
If $G$ contains a subdivision of $K_{2k,2k}$, then this follows from \autoref{thm:link}.
We may thus assume that $G$ does not contain such a subdivision, in particular it does not contain a subdivision of $K_{a,p}$.
	
Let $S = (s_1, \ldots, s_k)$ and $T=(t_1, \ldots, t_k)$ be disjoint $k$-tuples of distinct vertices of $G$.
Assume for a contradiction that $G$ does not contain disjoint paths $P_1,\ldots, P_k$ such that the end vertices of $P_i$ are $s_i$ and $t_i$ for $i=1,\ldots, k$ (such paths will be called the \emph{desired paths} in the rest of the proof).
	
By \autoref{thm:stablepd} there is a stable regular decomposition of $G$ of length at least $n_0$, of adhesion $q\leq w$, and of attachedness at least $2k+3$.
Since this decomposition has at least $(2k+1)(n_1-1)$ inner bags, there are $n_1-1$ consecutive inner bags which contain no vertex of $(S\cup T)$ apart from those coinciding with trivial paths.
In other words, this decomposition has a contraction $(\W, \P)$ of length $n_1$ such that $S\cup T\subseteq W_0\cup W_{n_1}$.
By \autoref{thm:contraction} this contraction has the same attachedness and adhesion as the initial decomposition and the stability is preserved.
Set $\theta\coloneq \{\alpha\mid P_{\alpha}\text{ is trivial}\}$ and $\lambda\coloneq \{\alpha\mid P_{\alpha}\text{ is non-trivial}\}$.
	
\begin{clm}\label{clm:main_path}
$\lambda\neq\emptyset$.
\end{clm}

\begin{proof}
If $\lambda = \emptyset$, or equivalently, $\P=\P_{\theta}$, then all adhesion sets of $\W$ equal $V(\P_{\theta})$.
So by (L2) no vertex of $G -\P_{\theta}$ is contained in more than one bag of~$\W$.
On the other hand, (L4) implies that every bag $W$ of $\W$ must contain a vertex $w\in W\setminus V(\P_{\theta})$.
Since $V(\P_{\theta})$ separates $W$ from the rest of $G$ and $G$ is $2k$-connected, there is a $w$--$\P_{\theta}$ fan of size $2k$ in $G[W]$.
For different bags, these fans meet only in $\P_{\theta}$.

Since $\W$ has more than $(2k-1)\binom{q}{2k}$ bags, the pigeon hole principle implies that there are $2k$ such fans with the same $2k$ end vertices among the $q$ vertices of $\P_{\theta}$.
The union of these fans forms a $TK_{2k,2k}$ in $G$ which may not exist by our earlier assumption.
\end{proof}

\begin{clm}\label{clm:main_nbhd}
Let $\Gamma_0$ be a component of $\Gamma(\W,\P)[\lambda]$.
The following all hold.
\begin{enumerate}[(i)]

\item\label{itm:nbhd_component}
$|N(\alpha)\cap \theta|\leq 2k-2$ for every vertex $\alpha$ of $\Gamma_0$.

\item\label{itm:nbhd_edge}
$|N(\alpha)\cap N(
\beta)\cap \theta|\leq 2k-4$ for every edge $\alpha\beta$ of $\Gamma_0$.

\item\label{itm:nbhd_vertex}
$2|N(\alpha)\cap \lambda| + |N(\alpha)\cap\theta|\leq 2k$ for every vertex $\alpha$ of $\Gamma_0$.

\item\label{itm:nbhd_block}
$2|D| + |N(D)|\leq 2k+2$ for every block $D$ of $\Gamma_0$ that contains a triangle.
\end{enumerate}
\end{clm}

Note that (\ref{itm:nbhd_vertex}) implies (\ref{itm:nbhd_component}) unless $\Gamma_0$ is a single vertex  and (\ref{itm:nbhd_vertex}) implies (\ref{itm:nbhd_edge}) unless $\Gamma_0$ is a single edge.
We need precisely these two cases in the proof of \autoref{clm:main_triangle}.

\begin{proof}
The proof is almost identical for all cases so we do it only once and point out the differences as we go.
Denote by $\Gamma_1$ the union of $\Gamma_0$ with all its incident edges of $\Gamma(\W,\P)$.
Set $L\coloneq W_0\cap W_1\cap V(G_{\Gamma_1})$ and $R\coloneq W_{n_1-1}\cap W_{n_1}\cap V(G_{\Gamma_1})$.
In case (\ref{itm:nbhd_block}) let $\alpha$ be any vertex of $D$.
Let $p$ and $q$ be the first and last vertex of $P_{\alpha}$.
Then $(L\cup R)\setminus\{p,q\}$ separates $G_{\Gamma_1}^{\P}-\{p,q\}$ from $S\cup T$ in $G-\{p,q\}$.
Hence by the connectivity of $G$ there is a set $\Q$ of $2k$ disjoint $(S\cup T)$--$(L\cup R)$ paths in $G-\{p,q\}$, each meeting $G_{\Gamma_1}^{\P}$ only in its last vertex.
For $i=1,\ldots,k$ denote by $s_i'$ the end vertex of the path of $\Q$ that starts in $s_i$ and by $t_i'$ the end vertex of the path of $\Q$ that starts in $t_i$.

Our task is to find disjoint $s'_i$--$t'_i$ paths for $i=1,\ldots,k$ in $G_{\Gamma_1}^{\P}$ and we shall now construct sets $X,Y\subseteq V(\Gamma_1)$ and an $X$--$Y$ pairing $L$ ``encoding'' this by repeating the following step for each $i\in \{1,\ldots, k\}$.
Let $\beta,\gamma\in V(\Gamma_1)$ such that $s'_i$ lies on $P_{\beta}$ and $t'_i$ lies on $P_{\gamma}$.
If $s'_i\in L$, then add $\beta$ to $X$ and set $\bar{s}_i\coloneq (\beta,0)$.
Otherwise $s'_i\in R\setminus L$ and we add $\beta$ to $Y$ and set $\bar{s}_i\coloneq (\beta,\infty)$.
Note that $s'_i \in L\cap R$ if and only if $\beta\in\theta$.
In this case our decision to add $\beta$ to $X$ is arbitrary and we could also add it to $Y$ instead (and setting $\bar{s}_i$ accordingly) without any bearing on the proof.
Handle $\gamma$ and $t'_i$ similarly.
Then $\{\bar{s}_i\bar{t}_i\mid i=1,\ldots, k\}$ is the edge set of an $(X,Y)$-pairing which we denote by~$L$.

We claim that there is an $L$-movement of length at most $(n_1-1)/2\geq f(k)$ on $H\coloneq\Gamma_1$ such that the vertices of $A\coloneq V(\Gamma_1)\cap \theta$ are singular.
Clearly $H-A = \Gamma_0$ is connected and every vertex of $A$ has a neighbour in $\Gamma_0$ so $A$ is marginal in $H$.
The existence of the desired $L$-movement follows from \autoref{thm:stars} if (\ref{itm:nbhd_component}) or (\ref{itm:nbhd_edge}) are violated, from \autoref{thm:maxdeg} if (\ref{itm:nbhd_vertex}) is violated, and from \autoref{thm:corecase} if (\ref{itm:nbhd_block}) is violated (note that $|H|\leq w$).
But then \autoref{thm:tokengame} applied to $L$ implies the existence of disjoint $s'_i$--$t'_i$ paths in $G_{\Gamma_1}^{\P}$ for $i=1,\ldots,k$ contradicting our assumption that $G$ does not contain the desired paths.
This shows that all conditions must hold.
\end{proof}
	
\begin{clm}\label{clm:main_minsize}
We have $2|\Gamma_0|+|N(\Gamma_0)| \geq 2k+3$ (and necessarily $N(\Gamma_0)\subseteq \theta$) for every component $\Gamma_0$ of $\Gamma(\W,\P)[\lambda]$.
\end{clm}
\begin{proof}
Let $\Gamma_1$ be the union of $\Gamma_0$ with all incident edges of $\Gamma(\W,\P)$.
Set $L\coloneq W_0\cap W_1\cap V(G_{\Gamma_1}^{\P})$, $M\coloneq W_1\cap W_2\cap V(G_{\Gamma_1}^{\P})$, and $R\coloneq W_{n_1-1}\cap W_{n_1}\cap V(G_{\Gamma_1}^{\P})$.
If $G-G_{\Gamma_1}^{\P}$ is non-empty, then $L\cup R$ separates it from $M$ in $G$.
Otherwise $M$ separates $L$ from $R$ in $G=G_{\Gamma_1}^{\P}$.
By the connectivity of $G$ we have $2|\Gamma_0|+|N(\Gamma_0)| = |L\cup R| \geq 2k+3$ in the former case and $|M| = |\Gamma_0|+|N(\Gamma_0)|\geq 2k+3$ in the latter.
\end{proof}

We now want to apply \autoref{thm:rural_bridge} and \autoref{thm:rural_cutpath}.
At the heart of both is the assertion that a certain society is rural and we already limited the number of their ``ingoing'' edges by \autoref{thm:euler}.
To obtain a contradiction we shall find societies exceeding this limit.
Tracking these down is the purpose of the notion of ``richness'' which we introduce next.

Let $\Gamma\subseteq \Gamma(\W,\P)[\lambda]$.
We say that $\alpha\in V(\Gamma)$ is \emph{rich in $\Gamma$} if the inner vertices of $P_{\alpha}$ that have a neighbour in both $G_{\lambda}-G_{\Gamma}^{\P}$ and $G_{\Gamma}^{\P}-P_{\alpha}$ have average degree at least $2+|N_{\Gamma}(\alpha)|(2+\varepsilon_{\alpha})$ in $G_{\Gamma}^{\P}$ where $\varepsilon_{\alpha}\coloneq 1/|N(\alpha)\cap\lambda|$.
A subgraph $\Gamma\subseteq \Gamma(\W,\P)[\lambda]$ is called \emph{rich} if every vertex $\alpha\in V(\Gamma)$ is rich in $\Gamma$.
	
\begin{clm}\label{clm:main_outdeg}
For $\Gamma\subseteq \Gamma(\W,\P)[\lambda]$ and $\alpha\in V(\Gamma)$ the following is true.
\begin{enumerate}[(i)]

\item\label{itm:out_all_rich}
If $\Gamma$ contains all edges of $\Gamma(\W,\P)[\lambda]$ that are incident with $\alpha$, then $\alpha$ is rich in $\Gamma$.

\item\label{itm:out_average_rich}
If $\alpha$ is rich in $\Gamma$, then the inner vertices of $P_{\alpha}$ that have a neighbour in $G_{\Gamma}^{\P}-P_{\alpha}$ have average degree at least $2+|N_{\Gamma}(\alpha)|(2+\varepsilon_{\alpha})$ in $G_{\Gamma}^{\P}$. 

\item\label{itm:out_union_rich}
Suppose that $\Gamma$ is induced in $\Gamma(\W,\P)[\lambda]$ and that there are subgraphs $\Gamma_1,\ldots,\Gamma_m\subseteq \Gamma$ such that $\alpha$ separates any two of them in $\Gamma(\W,\P)[\lambda]$ and $\bigcup_{i=1}^m\Gamma_i$ contains all edges of $\Gamma$ that are incident with $\alpha$.
If $\alpha$ is rich in $\Gamma$, then there is $j\in\{1,\ldots,m\}$ such that $\alpha$ is rich in $\Gamma_j$.
\end{enumerate}
\end{clm}

\begin{proof}~
\begin{enumerate}[(i)]
\item
The assumption implies that $G_{\Gamma}^{\P}$ contains every edge of $G_{\lambda}$ that is incident with $P_{\alpha}$ so no vertex of $P_{\alpha}$ has a neighbour in $G_{\lambda}-G_D^{\P}$ and therefore the statement is trivially true.

\item
The inner vertices of $P_{\alpha}$ that have a neighbour in $G_{\lambda}-G_{\Gamma}^{\P}$ and in $G_{\Gamma}^{\P}-P_{\alpha}$ have the desired average degree by assumption.
We show that each inner vertex of $P_{\alpha}$ that has no neighbour in $G_{\lambda}-G_{\Gamma}^{\P}$ has at least the desired degree.
Clearly we have $d_{G_{\Gamma}^{\P}}(v) = d_{G_{\lambda}}(v)$ for such a vertex $v$.
Furthermore, $d_G(v)\geq 2k+3$ since $G$ is $(2k+3)$-connected.
Every neighbour of $v$ in $\P_{\theta}$ gives rise to a neighbour of $\alpha$ in $\theta$ and by \autoref{clm:main_nbhd} (\ref{itm:nbhd_vertex}) there can be at most $|N(\alpha)\cap\theta|\leq 2k - 2|N(\alpha)\cap\lambda|$ such neighbours.
This means
\[d_{G_{\Gamma}^{\P}}(v) = d_{G_{\lambda}}(v) \geq 2k+3 - |N(\alpha)\cap\theta| \geq 2|N(\alpha)\cap\lambda| + 3\]
so (\ref{itm:out_average_rich}) clearly holds.

\item
We may assume that $\alpha$ is not isolated in $\Gamma$ and that each of the graphs $\Gamma_1,\ldots, \Gamma_m$ contains an edge of $\Gamma$ that is incident with $\alpha$ by simply forgetting those graphs that do not.

For $i=0,\ldots, m$ denote by $Z_i$ the inner vertices of $P_{\alpha}$ that have a neighbour in $G_{\lambda}-G_{\Gamma_i}^{\P}$ and in $G_{\Gamma_i}^{\P}-P_{\alpha}$ where $\Gamma_0\coloneq \Gamma$ and set $Z\coloneq \bigcup_{i=1}^m Z_i$.
Clearly $P_{\alpha}\subseteq G_{\Gamma_i}^{\P}$ for all $i$.
Each edge $e$ of $G_{\Gamma}^{\P}$ that is incident with an inner vertex of $P_{\alpha}$ but does not lie in $P_{\alpha}$ is in a $\P$-bridge that realises an edge of $\Gamma$ by (L6) and \autoref{thm:stable_bridges} since \autoref{clm:main_nbhd} (\ref{itm:nbhd_vertex}) implies that $|N(\alpha)\cap\theta|\leq 2k-2$.
So at least one of the graphs $G_{\Gamma_i}^{\P}$ contains $e$.
On the other hand, we have $G_{\Gamma_i}^{\P}\subseteq G_{\Gamma}^{\P}$ for $i=1,\ldots, m$.
This implies $Z_0\subseteq Z$.

By the same argument as in the proof of (\ref{itm:out_average_rich}) the vertices of $Z$ have average degree at least $2 + |N_{\Gamma}(\alpha)|(2+\varepsilon_{\alpha})$ in $G_{\Gamma}^{\P}$.
In other words, $G_{\Gamma}^{\P}$ contains at least $|Z| |N_{\Gamma}(\alpha)|(2+\varepsilon_{\alpha})$ edges with one end on $P_{\alpha}$ and the other in $G_{\Gamma}^{\P}-P_{\alpha}$.

By assumption we have $|N_{\Gamma}(\alpha)| = \sum_{i=1}^m |N_{\Gamma_i}(\alpha)|$ and so the pigeon hole principle implies that there is $j\in\{1,\ldots, m\}$ such that $G_{\Gamma_j}^{\P}$ contains a set $E$ of at least $|Z| |N_{\Gamma_j}(\alpha)|(2+\varepsilon_{\alpha})$ edges with one end on $P_{\alpha}$ and the other in $G_{\Gamma_j}^{\P}-P_{\alpha}$.

By assumption and \autoref{clm:rc_stable} the path $P_{\alpha}$ separates $G_{\Gamma_i}^{\P}$ from 
$G_{\Gamma_j}^{\P}$ in $G_{\lambda}$ for $i\neq j$.
For any vertex $z\in Z\setminus Z_j$ there is $i\neq j$ with $z\in Z_i$, so $z$ has a neighbour in $G_{\Gamma_i}^{\P}-P_{\alpha}\subseteq G_{\lambda}-G_{\Gamma_j}^{\P}$.
Then the only reason that $z$ is not also in $Z_j$ is that it has no neighbour in $G_{\Gamma_j}^{\P}-P_{\alpha}$, in particular, it is not incident with an edge of $E$.
So the vertices of $Z_j$ have average degree at least $2+ \frac{|Z|}{|Z_j|} |N_{\Gamma_j}(\alpha)|(2+\varepsilon_{\alpha})$ in $G_{\Gamma_j}^{\P}$ which obviously implies the claimed bound.\end{enumerate}
\end{proof}
	
\begin{clm}\label{clm:main_rich}
Every component of $\Gamma(\W,\P)[\lambda]$ contains a rich block.
\end{clm}
\begin{proof}
Let $\Gamma_0$ be a component of $\Gamma(\W,\P)[\lambda]$.
Suppose that $\alpha$ is a cut-vertex of $\Gamma_0$ and let $D_1,\ldots, D_m$ be the blocks of $\Gamma_0$ that contain $\alpha$. 
Clearly $N(\alpha)\cap \lambda\subseteq V(\bigcup_{i=1}^m D_i)$ so \autoref{clm:main_outdeg} implies that $\alpha$ is rich in $\bigcup_{i=1}^m D_i$ by (\ref{itm:out_all_rich}) and hence there is $j\in\{1,\ldots, m\}$ such that $\alpha$ is rich in $D_j$ by (\ref{itm:out_union_rich}).

We define an oriented tree $R$ on the set of blocks and cut-vertices of $\Gamma_0$ as follows.
Suppose that $D$ is a block of $\Gamma_0$ and $\alpha$ a cut-vertex of $\Gamma_0$ with $\alpha\in V(D)$.
If $\alpha$ is rich in $D$, then we let $(\alpha, D)$ be an edge of $R$.
Otherwise we let $(D,\alpha)$ be an edge of $R$.
Note that the underlying graph of $R$ is the block-cut-vertex tree of $\Gamma_0$ and by the previous paragraph every cut-vertex is incident with an outgoing edge of $R$.
But every directed tree has a sink, so there must be a block $D$ of $\Gamma_0$ such that every $\alpha\in\kappa$ is rich in $D$ where $\kappa$ denotes the set of all cut-vertices of $\Gamma_0$ that lie in $D$.

But the only vertices of $G_D^{\P}$ that may have a neighbour in $G_{\lambda}-G_D^{\P}$ are on paths of $\P_{V(D)}$ by \autoref{thm:block_separates} and of these clearly only the paths of $\P_{\kappa}$ may have neighbours in $G_{\lambda}-G_D^{\P}$.
So all vertices of $V(D)\setminus\kappa$ are trivially rich in $D$ and hence $D$ is a rich block.
\end{proof}

\begin{clm}\label{clm:main_triangle}
	Every rich block $D$ of $\Gamma(\W,\P)[\lambda]$ contains a triangle.
\end{clm}
\begin{proof}
Suppose that $D$ does not contain a triangle.
By \autoref{clm:main_minsize} and \autoref{clm:main_nbhd} (\ref{itm:nbhd_component}) we may assume $D$ is not an isolated vertex of $\Gamma(\W,\P)[\lambda]$, that is, $D$ contains an edge.
We shall obtain contradicting upper and lower bounds for the number
\[x\coloneq \sum_{v\in V(\P_{V(D)})}(d_{G_D^{\P}}(v) - d_{\P_{V(D)}}(v)).\]
For every $\alpha\in V(D)$ denote by $V_{\alpha}$ the subset of $V(P_{\alpha})$ that consists of the ends of $P_{\alpha}$ and all inner vertices of $P_{\alpha}$ that have a neighbour in $G_D^{\P}-P_{\alpha}$.
Set $V\coloneq \bigcup_{\alpha\in V(D)}V_{\alpha}$.

For the upper bound let $\alpha\beta$ be an edge of $D$.
Then $N_{\alpha\beta} \coloneq N(\alpha)\cap N(\beta)\subseteq \theta$ as a common neighbour of $\alpha$ and $\beta$ in $\lambda$ would give rise to a triangle in~$D$.
Furthermore, $|N_{\alpha\beta}|\leq 2k-4$ by \autoref{clm:main_nbhd} (\ref{itm:nbhd_edge}).
By \autoref{thm:rural_bridge} the society $(G_{\alpha\beta}^{\P},P_{\alpha}P_{\beta}^{-1})$ is rural if $\alpha$ and $\beta$ are not twins.
But if they are, then $N(\alpha)\cup N(\beta) = N_{\alpha\beta}\cup \{\alpha,\beta\}$.
This means that $D$ is a component of $\Gamma(\W,\P)[\lambda]$ that consists only of the single edge $\alpha\beta$.
So by \autoref{clm:main_minsize} we have $|N_{\alpha\beta}|=|N(D)|\geq 2k-1$, a contradiction.
Hence $(G_{\alpha\beta}^{\P},P_{\alpha}P_{\beta}^{-1})$ is rural.

The graph $G-\P_{N_{\alpha\beta}}$ contains $G_{\alpha\beta}^{\P}$ and has minimum degree at least $2k+3-|\P_{N_{\alpha\beta}}|\geq 6$ by the connectivity of $G$.
By \autoref{clm:main_nbhd} (\ref{itm:nbhd_component}) we have $|N(\gamma)\cap\theta|\leq 2k-2$ for every $\gamma\in\lambda$ so \autoref{thm:stable_bridges} implies that every non-trivial $\P$-bridge in an inner bag of $\W$ attaches to at least two paths of $\P_{\lambda}$ or to none.
A vertex $v$ of $G_{\alpha\beta}^{\P} - (P_{\alpha}\cup P_{\beta})$ is therefore an inner vertex of some non-trivial $\P$-bridge $B$ that attaches to $P_{\alpha}$ and $P_{\beta}$ and has all its inner vertices in $G_{\alpha\beta}^{\P}$.
This means that a neighbour of $v$ outside $G_{\alpha\beta}^{\P}$ must be an attachment of $B$ on some path $P_{\gamma}$ and hence $\gamma\in N_{\alpha\beta}\subseteq \theta$.
So all vertices of $G_{\alpha\beta}^{\P}- (P_{\alpha}\cup P_{\beta})$ have the same degree in $G_{\alpha\beta}^{\P}$ as in $G-\P_{N_{\alpha\beta}}$, namely at least~$6$.

The vertices of $G_{\alpha\beta}^{\P}-(P_{\alpha}\cup P_{\beta})$ retain their degree if we suppress all inner vertices of $P_{\alpha}$ and $P_{\beta}$ that have degree~$2$ in $G_D^{\P}$.
Since the paths of $\P$ are induced by (L6) an inner vertex of $P_{\alpha}$ has degree~$2$ in $G_D^{\P}$ if and only if it has no neighbour in $G_D^{\P}-P_{\alpha}$.
So we suppressed precisely those inner vertices of $P_{\alpha}$ and $P_{\beta}$ that are not in $V_{\alpha}$ or $V_{\beta}$.
By \autoref{thm:better_society} the society obtained from $(G_{\alpha\beta}^{\P}, P_{\alpha}P_{\beta}^{-1})$ in this way is still rural so \autoref{thm:euler} implies

\[\sum_{v\in V_{\alpha}\cup V_{\beta}}d_{G_{\alpha\beta}^{\P}}(v)\leq 4|V_{\alpha}| + 4|V_{\beta}|-6.\]

Clearly $G_D^{\P} = \bigcup_{\alpha\beta\in E(D)} G_{\alpha\beta}^{\P}$ and $P_{\alpha}\subseteq G_{\alpha\beta}^{\P}$ for all $\beta\in N_D(\alpha)$ and thus
\begin{align*}
x	&=\sum_{v\in V}\left(d_{G_D^{\P}}(v) -d_{\P_{V(D)}}(v)\right)\\
	&\leq\sum_{\alpha\in V(D)}\sum_{\beta\in N_D(\alpha)}\sum_{v\in V_{\alpha}} (d_{G_{\alpha\beta}^{\P}}(v) -d_{P_{\alpha}}(v))\\
	&=\sum_{\alpha\beta\in E(D)} \sum_{v\in V_{\alpha}\cup V_{\beta}} d_{G_{\alpha\beta}^{\P}}(v) - \sum_{\alpha\in V(D)}|N_D(\alpha)|\cdot(2 |V_{\alpha}| - 2)\\
	&\leq \sum_{\alpha\beta\in E(D)} \left(4|V_{\alpha}| + 4|V_{\beta}| - 6\right) - \sum_{\alpha\in V(D)}|N_D(\alpha)|\cdot(2 |V_{\alpha}| - 2)\\
	&=\sum_{\alpha\in V(D)} |N_D(\alpha)|\left(4|V_{\alpha}|-3\right) - \sum_{\alpha\in V(D)}|N_D(\alpha)|\cdot(2 |V_{\alpha}| - 2)\\
	&< \sum_{\alpha\in V(D)} 2|N_D(\alpha)|\cdot |V_{\alpha}|.
\end{align*}
To obtain the lower bound for $x$ note that \autoref{clm:main_outdeg} (\ref{itm:out_average_rich}) says that for any $\alpha\in V(D)$ the vertices of $V_{\alpha}$ without the two end vertices of $P_{\alpha}$ have average degree $2+|N_D(\alpha)|(2+\varepsilon_{\alpha})$ in $G_D^{\P}$ where $\varepsilon_{\alpha}\geq 1/k$ by \autoref{clm:main_nbhd} (\ref{itm:nbhd_vertex}).
Clearly every inner bag of $\W$ must contain a vertex of $V_{\alpha}$ as it contains a $\P$-bridge realising some edge $\alpha\beta\in E(D)$.
This means $|V_{\alpha}|\geq n_1/2\geq 4k+2$ and thus
\begin{align*}
x	&= \sum_{\alpha\in V(D)}\sum_{v\in V_{\alpha}} \left(d_{G_D^{\P}}(v) -d_{P_{\alpha}}(v)\right)\\
	&\geq \sum_{\alpha\in V(D)} (|V_{\alpha}| - 2) \cdot|N_D(\alpha)| \cdot(2 + \varepsilon_{\alpha})\\
	&\geq \sum_{\alpha\in V(D)} |N_D(\alpha)| \cdot\left(2|V_{\alpha}| - 4 + 4k\varepsilon_{\alpha}\right)\\
	&\geq \sum_{\alpha\in V(D)} 2|N_D(\alpha)|\cdot |V_{\alpha}|.\qedhere
\end{align*}
\end{proof}

\begin{clm}\label{clm:main_largeblock}
Every rich block $D$ of $\Gamma(\W,\P)[\lambda]$ satisfies $2|D| + |N(D)|\geq 2k+3$.
\end{clm}

\begin{proof}
Suppose for a contradiction that $2|D| + |N(D)|\leq 2k+2$.
By \autoref{thm:rural_cutpath} there is a $V(D)$-compressed $(\P, V(D))$-relinkage $\Q$ with properties as listed in the statement of \autoref{thm:rural_cutpath}.
Let us first show that we are done if $D$ is \emph{rich w.r.t.\ to $\Q$}, that is, for every $\alpha\in V(D)$ the inner vertices of $Q_{\alpha}$ that have a neighbour in $G_{\lambda}-G_D^{\Q}$ and in $G_D^{\Q}-Q_{\alpha}$ have average degree at least $2+|N_D(\alpha)|(2+\varepsilon_{\alpha})$ in $G_D^{Q}$.

Denote the cut-vertices of $\Gamma(\W,\P)[\lambda]$ that lie in $D$ by $\kappa$.
For $\alpha\in\kappa$ let $V_{\alpha}$ be the set consisting of the ends of $Q_{\alpha}$ and of all inner vertices of $Q_{\alpha}$ that have a neighbour in $G_D^{\Q}-Q_{\alpha}$ and set $V\coloneq \bigcup_{\alpha\in\kappa}V_{\alpha}$.
Pick $\alpha\in\kappa$ such that $|V_{\alpha}|$ is maximal.
By \autoref{thm:compressed} (with $p=2k+3$) every vertex of $G_D^{\Q}$ lies on a path of $\Q_{V(D)}$ and we have $|Q_{\beta}| < |V_{\alpha}|$ for all $\beta\in V(D)\setminus\kappa$.

The paths of $\Q$ are induced in $G$ as $\Q[W]$ is $(2k+3)$-attached in $G[W]$ for every inner bag $W$ of $\W$.
Hence $V_{\alpha}$ contains precisely the vertices of $Q_{\alpha}$ that are not inner vertices of degree~$2$ in $G_D^{\Q}$.
By the same argument as in the proof of \autoref{clm:main_outdeg} (\ref{itm:out_average_rich}) the vertices of $V_{\alpha}$ that are not ends of $Q_{\alpha}$ have average degree at least $2 + |N_D(\alpha)|(2+\varepsilon_{\alpha})$ in $G_D^{\Q}$.

We want to show that the average degree in $G_D^{\Q}$ taken over all vertices of $V_{\alpha}$ is larger than $2+2|N_D(\alpha)|$.
Clearly the end vertices of $Q_{\alpha}$ have degree at least~$1$ in $G_D^{\Q}$ so both lack at most $1+2|N_D(\alpha)|\leq 3|N_D(\alpha)|$ incident edges to the desired degree.
On the other hand, the degree of every vertex of $V_{\alpha}$ that is not an end of $Q_{\alpha}$ is on average at least $|N_D(\alpha)|\cdot\varepsilon_{\alpha}$ larger than desired.
But $\varepsilon_{\alpha}\geq 1/k$ by \autoref{clm:main_nbhd} (\ref{itm:nbhd_vertex}) and by \autoref{thm:rural_cutpath} (\ref{itm:rc_blockZ}) the path $Q_{\alpha}[W]$ contains a vertex of $V_{\alpha}$ for every inner bag of $\W$, in particular, $|V_{\alpha}|\geq n_1/2 > 6k+2$ and hence $(|V_{\alpha}|-2)\varepsilon_{\alpha}> 6$.

This shows that there are more than $2|V_{\alpha}|\cdot |N_D(\alpha)|$ edges in $G_D^{\Q}$ that have one end on $Q_{\alpha}$ and the other on another path of $\Q_{V(D)}$.
By \autoref{thm:global_disturbance} these edges can only end on paths of $\Q_{N_D(\alpha)}$ so by the pigeon hole principle there is $\beta\in N_D(\alpha)$ such that $G_D^{\Q}$ contains more than $2|V_{\alpha}|$ edges with one end on $Q_{\alpha}$ and the other on $Q_{\beta}$.

Hence the society $(H,\Omega)$ obtained from $(G_{\alpha\beta}^{\Q}, Q_{\alpha}Q_{\beta}^{-1})$ by suppressing all inner vertices of $Q_{\alpha}$ and $Q_{\beta}$ that have degree~$2$ in $G_{\alpha\beta}^{\Q}$ has more than $2|V_{\alpha}| + 2|V_{\beta}| - 2$ edges and all its $|V_{\alpha}|+|V_{\beta}|$ vertices are in $\bar{\Omega}$.
So by \autoref{thm:euler} $(H,\Omega)$ cannot be rural.
But it is trivially $4$-connected as all its vertices are in $\bar{\Omega}$ and must therefore contain a cross by \autoref{thm:cross}.
The paths of $\Q$ are induced so this cross consists of two edges which both have one end on $Q_{\alpha}$ and the other on $Q_{\beta}$.
Such a cross gives rise to a linkage $\Q'$ from the left to the right adhesion set of some inner bag $W$ of $\W$ such that the induced permutation of $\Q'$ maps some element of $V(D)\setminus\{\alpha\}$ (not necessarily $\beta$) to $\alpha$ and maps every $\gamma\notin V(D)$ to itself.
Since $\alpha$ has a neighbour outside $D$ this is not an automorphism of $\Gamma(\W,\P)[\lambda]$ and therefore $\Q'$ is a twisting disturbance contradicting the stability of $(\W,\P)$.

It remains to show that $D$ is rich w.r.t.\ $\Q$.
Suppose that it is not.
By the same argument as for \autoref{clm:main_outdeg} (\ref{itm:out_all_rich}) there must be $\alpha\in\kappa$ such that the inner vertices of $Q_{\alpha}$ that have a neighbour in $G_{\lambda}-G_D^{\Q}$ and in $G_D^{\Q}-Q_{\alpha}$ have average degree less than $2+|N_D(\alpha)|(2+\varepsilon_{\alpha})$ in $G_D^{Q}$.
Let $(\lambda_1,\lambda_2)$ be a separation of $\Gamma(\W,\P)[\lambda]$ with $\lambda_1\cap\lambda_2 =\{\alpha\}$ and $N(\alpha)\cap \lambda_2 = N(\alpha)\cap V(D)$.
Let $H$, $Z_1$, $Z_2$, $\P'$, $q_1$, and $q_2$ be as in the statement of \autoref{thm:rural_cutpath}.
We shall obtain contradicting upper and lower bounds for the number
\[x\coloneq \sum_{v\in V(P'_{\alpha}\cup Q_{\alpha})}\left(d_H(v)-d_{P'_{\alpha}\cup Q_{\alpha}}(v)\right).\]
Denote by $H_1,\ldots, H_m$ the blocks of $H$ that are not a single edge and for $i=1,\ldots, m$ let $V_i$ be the set of vertices of $C_i\coloneq H_i\cap (P'_{\alpha}\cup Q_{\alpha})$ that are a cut-vertex of $H$ or are incident with some edge of $H_i$ that is not in $P'_{\alpha}\cup Q_{\alpha}$ and set $V\coloneq \bigcup_{i=1}^m V_i$.
By definition we have $d_H(v)=d_{P'_{\alpha}\cup Q_{\alpha}}(v)$ for all vertices $v$ of $(P'_{\alpha}\cup Q_{\alpha})-V$.

Note that $H$ is adjacent to at most $|N(\alpha)\cap\theta|$ vertices of $\P_{\theta}$ by \autoref{thm:rural_cutpath} (\ref{itm:rc_connected}) and \autoref{thm:global_disturbance}.
So \autoref{clm:main_nbhd} (\ref{itm:nbhd_vertex}) and the connectivity of $G$ imply that every vertex of $H$ has degree at least $2k+3-|N(\alpha)\cap\theta|\geq 2|N(\alpha)\cap\lambda| + 3$ in $G_{\lambda}$.

To obtain an upper bound for $x$ let $i\in \{1,\ldots, m\}$.
By \autoref{thm:rural_cutpath} (\ref{itm:rc_rural}) $C_i$ is a cycle and the society $(H_i, \Omega(C_i))$ is rural where $\Omega(C_i)$ denotes one of two cyclic permutations that $C_i$ induces on its vertices.
Since $|N(\alpha)\cap\lambda|\geq 2$ every vertex of $H_i-C_i$ has degree at least~$6$ in $H_i$ by the previous paragraph.
This remains true if we suppress all vertices of $C_i$ that have degree~$2$ in $H_i$.
The society obtained in this way is still rural by \autoref{thm:better_society}.
Since we suppressed precisely those vertices of $C_i$ that are not in $V_i$ \autoref{thm:euler} implies $\sum_{v\in V_i}d_{H_i}(v)\leq 4|V_i| - 6$.
By definition of $V$ we have $d_H(v) = d_{P'_{\alpha}\cup Q_{\alpha}}(v)$ for all vertices $v$ of $P'_{\alpha}\cup Q_{\alpha}$ that are not in $V$.
Hence we have
\[x = \sum_{v\in V}\left(d_H(v)-d_{P'_{\alpha}\cup Q_{\alpha}}(v)\right)
	= \sum_{i=1}^m \sum_{v\in V_i} \left(d_{H_i}(v) - d_{C_i}(v)\right)
	\leq\sum_{i=1}^m \left(2|V_i| - 6\right).\]

Let us now obtain a lower bound for $x$.
Clearly $G_D^{\P}\subseteq G_{\lambda_2}^{\P}$ and $G_D^{\Q}\subseteq G_{\lambda_2}^{\Q}$.
To show that $d_{G_D^{\Q}}(v) = d_{G_{\lambda_2}^{\Q}}(v)$ for all $v\in V(H)$ (we follow the general convention that a vertex has degree~$0$ in any graph not containing it) it remains to check that an edge of $G_{\lambda}$ that has precisely one end in $H$ but is not in $G_D^{\Q}$ cannot be in $G_{\lambda_2}^{\Q}$.
Such an edge $e$ must be in a $\Q$-bridge that attaches to $Q_{\alpha}$ and some $Q_{\beta}$ with $\beta\in\lambda\setminus V(D)$.
But $N(\alpha)\cap\lambda_2 = V(D)$ and hence $\beta\in\lambda_1$.
So $e$ is an edge of $G_{\lambda_1}^{\Q}$ but not on $Q_{\alpha}$ and therefore not in $G_{\lambda_2}^{\Q}$.
This already implies $d_{G_D^{\P}}(v) = d_{G_{\lambda_2}^{\P}}(v)$ for all $v\in V(H)$ since $G_D^{\Q}\subseteq G_D^{\P}$ and $G_{\lambda_2}^{\P} = H \cup G_{\lambda_2}^{\Q}$ (see the proof of \autoref{thm:rural_cutpath} (\ref{itm:rc_Z}) for the latter identity).
The next equality follows directly from the definition of $H$.
\[d_H(v) + d_{G_{\lambda_2}^{\Q}}(v) = d_{G_{\lambda_2}^{\P}}(v) + d_{Q_{\alpha}}(v)\quad\forall v\in V(H).\]

Denote by $U_1$ the set of inner vertices of $P_{\alpha}$ that have a neighbour in both $G_{\lambda}-G_D^{\P}$ and $G_D^{\P}-P_{\alpha}$ and by $U_2$ the set of inner vertices of $Q_{\alpha}$ that have a neighbour in both $G_{\lambda}-G_D^{\Q}$ and $G_D^{\Q}-Q_{\alpha}$.
In other words, $U_1$ and $U_2$ are the sets of those vertices of $P_{\alpha}$ and $Q_{\alpha}$, respectively, that are relevant for the richness of $\alpha$ in $D$.
Set $V'\coloneq (V\setminus\{q_1,q_2\})\cup (Z_1\cap Z_2))$, $V_P\coloneq V'\cap V(P'_{\alpha})$, and $V_Q\coloneq V'\cap V(Q_{\alpha})$.
Then $U_1 = (V\cap Z_1)\cup (Z_1\cap Z_2) = V'\cap Z_1\subseteq V_P$ and $U_2 = (V\cap Z_2)\cup (Z_1\cap Z_2)\subseteq V_Q$.

By our earlier observation every vertex of $H$ has degree at least $2|N(\alpha)\cap\lambda|+3$ in $G_{\lambda}$ and therefore every vertex of $V_P\setminus Z_1$ must have at least this degree in $G_{\lambda_2}^{\P}$.
Since $U_1\subseteq V_P$ and $\alpha$ is rich in $D$ this means that
\[\sum_{v\in V_P}d_{G_D^{\P}}(v) \geq |V_P|\left(2 + |N_D(\alpha)|\cdot(2+\varepsilon_{\alpha})\right).\]
Similarly, we have $U_2\subseteq V_Q\subseteq V(Q_{\alpha})$ and every vertex $v\in V_Q\setminus Z_2$ satisfies $d_{G_{D}^{\Q}}(v) = 2 = d_{Q_{\alpha}}(v)$.
So by the assumption that $\alpha$ is not rich in $D$ w.r.t.~$\Q$ we have
\[\sum_{v\in V_Q}\left(d_{G_D^{\Q}}(v) - d_{Q_{\alpha}}(v)\right) < |V_Q|\cdot|N_D(\alpha)|\cdot (2+\varepsilon_{\alpha}).\]
Observe that
\[2|N(\alpha)\cap\lambda| + 3 = 2 + |N(\alpha)\cap\lambda_1|\cdot(2+\varepsilon_{\alpha}) + |N(\alpha)\cap\lambda_2|\cdot(2+\varepsilon_{\alpha})\]
and recall that $N_D(\alpha) = N(\alpha)\cap\lambda_2$.
Combining all of the above we get

\begin{align*}
x	&\geq \sum_{v\in V'}\left(d_H(v)-d_{P'_{\alpha}\cup Q_{\alpha}}(v)\right)\\
	&=\sum_{v\in V'} \left(d_{G_D^{\P}}(v) - d_{G_D^{\Q}}(v) + d_{Q_{\alpha}}(v) - d_{P'_{\alpha}\cup Q_{\alpha}}(v)\right)\\
	&=\sum_{v\in V_P}d_{G_D^{\P}}(v) + \sum_{v\in V'\setminus V_P} d_{G_D^{\P}}(v) - \sum_{v\in V_Q}\left(d_{G_D^{\Q}}(v) - d_{Q_{\alpha}}(v)\right)-2|V'|-2m\\
	&>|V_P|\cdot |N_D(\alpha)|\cdot (2+\varepsilon_{\alpha}) + 2|V_P| + |V'\setminus V_P|\cdot (2|N(\alpha)\cap\lambda| + 3)\\
	&\quad -|V_Q|\cdot|N_D(\alpha)|\cdot (2+\varepsilon_{\alpha})- 2|V'| - 2m\\
	&=|V'\setminus V_Q|\cdot |N_D(\alpha)|\cdot (2+\varepsilon_{\alpha}) + |V'\setminus V_P|\cdot |N(\alpha)\cap\lambda_1|\cdot(2+\varepsilon_{\alpha})-2m\\
	&> 2|V'\setminus V_Q| + 2|V'\setminus V_P| - 2m = \sum_{i=1}^m \left(2|V_i|-6\right)
\end{align*}
This shows that $D$ is rich w.r.t.\ $\Q$ as defined above.
So \autoref{clm:main_largeblock} holds.
\end{proof}

By \autoref{clm:main_path} the graph $\Gamma(\W,\P)[\lambda]$ has a component.
This component has a rich block $D$ by \autoref{clm:main_rich}.
By \autoref{clm:main_triangle} and \autoref{clm:main_largeblock} we have a triangle in $D$ and $|D|+|N(D)|\geq 2k+3$.
This contradicts \autoref{clm:main_nbhd} (\ref{itm:nbhd_block}) and thus concludes the proof of \autoref{thm:main}.
\end{proof}

\section{Discussion}\label{sec:discussion}
In this section we first show that \autoref{thm:main} is almost best possible (see \autoref{thm:tightness} below) and then summarise where our proof uses the requirement that the graph $G$ is $(2k+3)$-connected.

\begin{pro}\label{thm:tightness}
	For all integers $k$ and $N$ with $k\geq 2$ there is a graph $G$ which is not $k$-linked such that
	\[\kappa(G)\geq 2k+1,\qquad \tw(G)\leq 2k + 10,\quad\text{and}\quad |G|\geq N.\]
\end{pro}

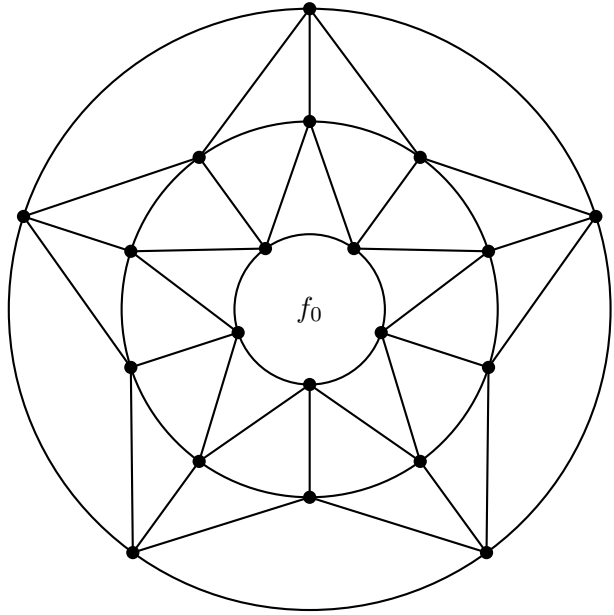
\begin{figure}[htpb]
	\begin{center}
	\begin{tikzpicture}[style=thick]
	\def\iradius{1cm}
	\def\mradius{2.5cm}
	\def\oradius{4cm}
	
	\foreach \x in {18,90,...,360}{
	\node [fill,circle,inner sep=0pt,minimum size=5pt] at (\x:\mradius){};
	\node [fill,circle,inner sep=0pt,minimum size=5pt] at (\x:\oradius){};
	\draw (\x:\mradius) -- (\x:\oradius);
	\draw (\x-36:\mradius) -- (\x:\oradius) -- (\x+36:\mradius);
	}
	\foreach \x in {54,126,...,360}{
	\node [fill,circle,inner sep=0pt,minimum size=5pt] at (\x:\mradius){};
	\node [fill,circle,inner sep=0pt,minimum size=5pt] at (\x:\iradius){};
	\draw (\x:\mradius) -- (\x:\iradius);
	\draw (\x-36:\mradius) -- (\x:\iradius) -- (\x+36:\mradius);
	}
	\draw circle(\iradius);
	\draw circle(\mradius);
	\draw circle(\oradius);
	\node {$f_0$};

	\end{tikzpicture}
	\caption{The $5$-connected graph $H_0$ and its inner face $f_0$.}
	\label{fig:fatstar}
	\end{center}
\end{figure}

\begin{proof}
We reduce the assertion to the case $k=2$, that is, to the claim that there is a graph $H$ which is not $2$-linked but satisfies
\[\kappa(H) = 5,\qquad \tw(H)\leq 14,\quad\text{and}\quad |H|\geq N.\]
For any $k\geq 3$ let $K$ be the graph with $2k-4$ vertices and no edges.
We claim that $G\coloneq H * K$ (the disjoint union of $H$ and $K$ where every vertex of $H$ is joined to every vertex of $K$ by an edge) satisfies the assertion for $k$.

Clearly $|G| = |H| + 2k - 4 \geq N$.
Taking a tree-decomposition of $H$ of minimal width and adding $V(K)$ to every bag gives a tree-decomposition of~$G$, so $\tw(G)\leq\tw(H) + 2k - 4\leq 2k + 10$.
To see that $G$ is $(2k+1)$-connected, note that it contains the complete bipartite graph with partition classes $V(H)$ and $V(K)$, so any separator $X$ of $G$ must contain $V(H)$ or $V(K)$.
In the former case we have $|X|\geq N$ and we may assume that this is larger than $2k$.
In the latter case we know that $G-X\subseteq H$, in particular $X\cap V(H)$ is a separator of $H$ and hence must have size at least $5$, implying $|X|\geq |K| + 5 = 2k+1$ as required.

Finally, $G$ is not $k$-linked:
By assumption there are vertices $s_1$, $s_2$, $t_1$, $t_2$ of $H$ such that $H$ does not contain disjoint paths $P_1$ and $P_2$ where$P_i$ ends in $s_i$ and $t_i$ for $i=1,2$.
If $G$ was $k$-linked, then for any enumeration $s_3,\ldots,s_k,t_3,\ldots,t_k$ of the $2k-4$ vertices of $V(K)$
there were disjoint paths $P_1, \ldots, P_k$ in $G$ such that $P_i$ has end vertices $s_i$ and $t_i$ for $i=1,\ldots, k$.
In particular, $P_1$ and $P_2$ do not contain a vertex of $K$ and are hence contained in $H$, a contradiction.

It remains to give a counterexample for $k=2$.
The planar graph $H_0$ in \autoref{fig:fatstar} is $5$-connected.
Denote the $5$-cycle bounding the outer face of $H_0$ by $C_1$ and the $5$-cycle bounding $f_0$ by $C_0$.
Then $(V(H_0-C_0), V(H_0-C_1))$ forms a separation of $H_0$ of order $10$, in particular, $H_0$ has a tree-decomposition of width $14$ where the tree is $K_2$.
Draw a copy $H_1$ of $H_0$ into $f_0$ such that the cycle $C_0$ of $H_0$ gets identified with the copy of $C_1$ in $H_1$.
Since $H_0\cap H_1$ has $5$ vertices, the resulting graph is still $5$-connected and has a tree-decomposition of width~$14$.
We iteratively paste copies of $H_0$ into the face $f_0$ of the previously pasted copy as above until we end up with a planar graph $H$ such that
\[\kappa(H) = 5,\qquad \tw(H)\leq 14,\quad\text{and}\quad |H|\geq N.\]
Still the outer face of $H$ is bounded by a $5$-cycle $C_1$, so we can pick vertices $s_1,s_2,t_1,t_2$ in this order on $C_1$ to witness that $H$ is not $2$-linked (any $s_1$--$t_1$ path must meet any $s_2$--$t_2$ path by planarity).
\end{proof}

Where would our proof of \autoref{thm:main} fail for a $(2k+2)$-connected graph~$G$?
There are several instances where we invoke $(2k+3)$-connectivity as a substitute for a minimum degree of at least $2k+3$.
The only place where minimum degree $2k+2$ does not suffice is the proof of \autoref{clm:main_outdeg}.
We need minimum degree $2k+3$ there to get the small ``bonus'' $\varepsilon_{\alpha}$ in our notion of richness.
Richness only allows us to make a statement about the inner vertices of a path and the purpose of this bonus is to compensate for the end vertices.
Therefore the arguments involving richness in the proofs of \autoref{clm:main_triangle} and \autoref{clm:main_largeblock} would break down if we only had minimum degree $2k+2$.

But even if the suppose that $G$ has minimum degree at least $2k+3$ there are still two places where our proof of \autoref{thm:main} fails:
The first is the proof of \autoref{clm:main_minsize} and the second is the application of \autoref{thm:compressed} in the proof of \autoref{clm:main_largeblock}.

We use \autoref{clm:main_minsize} in the proof of \autoref{clm:main_triangle}, to show that no component of $\Gamma(\W,\P)[\lambda]$ can be a single vertex or a single edge.
In both cases we do not use the full strength of \autoref{clm:main_minsize}.
So although we formally rely on $(2k+3)$-connectivity for \autoref{clm:main_minsize} we do not really need it here.

However, the application of \autoref{thm:compressed} in the proof of \autoref{clm:main_largeblock} does need $(2k+3)$-connectivity.
Our aim there is to get a contradiction to \autoref{clm:main_nbhd} (\ref{itm:nbhd_block}) which gets the bound $2k+3$ from the token game in \autoref{thm:corecase}.
This bound is sharp:
Let $H$ be the union of a triangle $D=d_1d_2d_3$ and two edges $d_1a_1$ and $d_2a_2$ and set $A\coloneq\{a_1,a_2\}$.
Clearly $H-A=D$ is connected and $A$ is marginal in $H$.
For $k=3$ we have $2|D|+ |N(D)| = 8 = 2k+2$.
Let $L$ be the pairing with edges $(a_1,0)(a_2,0)$ and $(d_i,0)(d_i,\infty)$ for $i=1,2$.
It is not hard to see that there is no $L$-movement on $H$ as the two tokens from $A$ can never meet.

So the best hope of tweaking our proof of \autoref{thm:main} to work for $(2k+2)$-connected graphs is to provide a different proof for \autoref{clm:main_largeblock}.
This would also be a chance to avoid relinkages, that is, most of \autoref{sec:relinkage}, and the very technical \autoref{thm:rural_cutpath} altogether as they only serve to establish \autoref{clm:main_largeblock}.

%
%


\begin{thebibliography}{ddd}
	\bibitem{kak}
		\textsc{T.~B\"ohme, K.~Kawarabayashi, J.~Maharry} and \textsc{B.~Mohar}:
		Linear Connectivity Forces Large Complete Bipartite Minors,
		\textit{J.~Combin.\ Theory~B} \textbf{99} (2009), 557--582.
	\bibitem{kakbdtw}
		\textsc{T.~B\"ohme, J.~Maharry} and \textsc{B.~Mohar}:
		$K_{a,k}$ Minors in Graphs of Bounded Tree-Width,
		\textit{J.~Combin.\ Theory~B} \textbf{86} (2002), 133--147.			
	\bibitem{BT}
 		\textsc{B.~Bollob\'as} and \textsc{A. Thomason}:
  		Highly Linked Graphs, 
		\textit{Combinatorica} \textbf{16} (1996) 313--320. 
	\bibitem{diestel}
		\textsc{R.~Diestel}:
		\textit{Graph Theory} (3rd edition), GTM 173, Springer--Verlag, Heidelberg (2005).
	\bibitem{structlargetw}
		\textsc{R.~Diestel, K.~Kawarabayashi, T.M\"uller} and \textsc{P.~Wollan}:
		On the excluded minor structure theorem for graphs of large tree width,  
		\textit{J.~Combin.\ Theory~B} \textbf{102} (2012), 1189--1210.
	\bibitem{15puzzle1}
		\textsc{W.~W.~Johnson}:
		Notes on the ``15'' Puzzle. {I}.
		\textit{Amer.~J.~Math.} \textbf{2} (1879), 397--399.
	\bibitem{jorg}
		\textsc{L.~J\o rgensen}:
		Contraction to $K_8$, 
		\textit{J.~Graph Theory} \textbf{18} (1994), 431--448.
	\bibitem{J}
  		\textsc{H.~A.~Jung}: 
  		Verallgemeinerung des n-Fachen Zusammenhangs fuer Graphen, 
		\textit{Math.~Ann.} \textbf{187} (1970), 95--103.
	\bibitem{k6bdtw}
		\textsc{K.~Kawarabayashi, S.~Norine, R.~Thomas} and \textsc{P.~Wollan}:
		$K_6$ minors in $6$-connected graphs of bounded tree-width,
		\href{http://arxiv.org/abs/1203.2171}{\textit{arXiv}:\textbf{1203.2171}} (2012).
	\bibitem{k6large}
		\textsc{K.~Kawarabayashi, S.~Norine, R.~Thomas} and \textsc{P.~Wollan}:
		$K_6$ minors in large $6$-connected graphs,
		\href{http://arxiv.org/abs/1203.2192}{\textit{arXiv}:\textbf{1203.2192}} (2012).
	\bibitem{kornhauser}
		\textsc{D.~Kornhauser}, \textsc{G.~L.~Miller}, and \textsc{P.~G.~Spirakis}:
		Coordinating Pebble Motion on Graphs, the Diameter of Permutation Groups, and Applications,
		\textit{25th Annual Symposium on Foundations of Computer Science (FOCS 1984)} (1984), 241--250.
	\bibitem{Kostochka}
		\textsc{A.~Kostochka}, 
		A Lower Bound for the Hadwiger Number of a Graph as a Function of the Average Degree of Its Vertices, 
		\textit{Discret.~Analyz,~Novosibirsk}, \textbf{38} (1982, 37--58.
	\bibitem{LM}
  		\textsc{D.~G.~Larman} and \textsc{P.~Mani}: 
  		On the Existence of Certain Configurations Within Graphs and the 1-Skeletons of Polytopes,
		\textit{Proc.~London~Math.~Soc.} \textbf{20} (1974), 144--160.
	\bibitem{typical}
		\textsc{B.~Oporowski, J.~Oxley} and \textsc{R.~Thomas}:
		Typical Subgraphs of $3$- and $4$-connected graphs,
		\textit{J.~Combin.\ Theory~B} \textbf{57} (1993), 239--257.
	\bibitem{gm9}
		\textsc{N.~Robertson} and \textsc{P.~D.~Seymour}:
		Graph Minors. {IX}. {D}isjoint Crossed Paths,
		\textit{J.~Combin.\ Theory~B} \textbf{49} (1990), 40--77.
	\bibitem{RS13}
 		\textsc{N.~Robertson} and \textsc{P.D.~Seymour}: 
		Graph Minors {XIII}. The Disjoint Paths Problem, 
		\textit{J.~Combin.\ Theory~B} \textbf{63} (1995), 65--110.
	\bibitem{RS16}
 		\textsc{N.~Robertson} and \textsc{P.D.~Seymour}: 
		Graph Minors. {XVI}. Excluding a non-planar graph, 
		\textit{J.~Combin.\ Theory~B} \textbf{89} (2003), 43--76.
	\bibitem{lean}
		\textsc{R.~Thomas}:
		A Menger-like Property of Tree-Width: The Finite Case,
		\textit{J.~Combin.\ Theory~B} \textbf{48} (1990), 67--76.
	\bibitem{Thomas}
		\textsc{R.~Thomas}:
		Configurations in Large $t$-connected Graphs [presentation],
		\textit{SIAM Conference on Discrete Mathematics}, Austin, Texas, USA, June 14--17 2010, (accessed 2014-02-21)\newline
		\texttt{\small\url{https://client.blueskybroadcast.com/siam10/DM/SIAM_DM_07_IP6/slides/Thomas\_IP6\_6\_16\_10\_2pm.pdf}}.
	\bibitem{TW}
		\textsc{R.~Thomas} and \textsc{P.~Wollan}:
		An improved linear edge bound for graph linkages, 
		\textit{Europ.~J.~Comb.} \textbf{26} (2005), 309--324.
	\bibitem{Thomason}
  		\textsc{A.~Thomason}: 
		An Extremal Function for Complete Subgraphs, 
		\textit{Math.~Proc.~Camb.~Phil.~Soc.}  \textbf{95} (1984), 261--265.
	\bibitem{thomassen}
		\textsc{C.~Thomassen}:
		2-linked graphs, 
		\textit{Europ.~J.~Comb.} \textbf{1} (1980), 371--378.
	\bibitem{thomassen_personal}
		\textsc{C.~Thomassen}:
		\textit{personal communication} (2014).
	\bibitem{watkins}
		\textsc{M.~E.~Watkins}:
		On the existence of certain disjoint arcs in graphs,
		\textit{Duke~Math.~J.} \textbf{35} (1968), 231--246.
	\bibitem{wilson}
		\textsc{R.~M.~Wilson}:
		Graph Puzzles, Homotopy, and the Alternating Group,
		\textit{J.~Combin.\ Theory~B} \textbf{16} (1974), 86--96.
\end{thebibliography}
\end{document}